\documentclass[12pt,a4paper]{amsart}

\usepackage[margin=3.4cm]{geometry}
\usepackage{amscd,amsmath,amssymb,amsthm,amsfonts}
\usepackage{mathrsfs}
\usepackage[shortlabels]{enumitem}
\usepackage[all]{xy}
\usepackage{tikz-cd}
\usetikzlibrary{arrows, matrix}
\usepackage[foot]{amsaddr}
\usepackage{bbm}

\usepackage{xcolor}
\definecolor{green(munsell)}{rgb}{0.0, 0.66, 0.47}
\definecolor{BlueGreenn}{rgb}{0.3,0.5,0.8}
\definecolor{darkblue}{rgb}{0.1,0.12,0.24}
\definecolor{DB}{rgb}{0.3,0.3,0.3}
\definecolor{DOr}{rgb}{0.7,0.3,0.3}
\definecolor{DGr}{rgb}{0.3,0.7,0.3}
\definecolor{DBl}{rgb}{0.1,0.3,0.5}
\definecolor{arylideyellow}{rgb}{0.91, 0.84, 0.42}
\definecolor{burntorange}{rgb}{0.8, 0.33, 0.0}
\definecolor{chromeyellow}{rgb}{1.0, 0.65, 0.0}

\usepackage[hypertexnames=true, citecolor = burntorange, colorlinks = true, linkcolor = darkblue, pdffitwindow = false, urlbordercolor = white,pdfborder={0 0 0}]{hyperref}
\RequirePackage{doi}
\usepackage{hyperref}
\usepackage[T1]{fontenc}
\usepackage{lmodern}
\usepackage{microtype}
\usepackage{color}
\usepackage[utf8]{inputenc}
\usepackage{framed}
\usepackage{centernot}

\numberwithin{equation}{section}



\newtheorem{theorem}{Theorem}[section]
\newtheorem{proposition}[theorem]{Proposition}
\newtheorem{lemma}[theorem]{Lemma}

\newtheorem{corollary}[theorem]{Corollary}
\newtheorem{definition}[theorem]{Definition}
\newtheorem{definitionlemma}[theorem]{Definition and Lemma}

\theoremstyle{remark}
\newtheorem{remark}[theorem]{Remark}
\newtheorem{example}[theorem]{Example}

\def\Aut{\operatorname{Aut}}

\def\End{\operatorname{End}}
\def\newspan{\operatorname{span}}

\def\ker{\operatorname{ker}}
\def\dim{\operatorname{dim}}

\def\id{\operatorname{id}}

\def\Ind{\operatorname{Ind}}
\DeclareMathOperator*{\wlim}{w-lim}
\DeclareMathOperator*{\slim}{s-lim}

\def\id{\operatorname{id}}

\DeclareMathOperator{\Tr}{Tr}

\DeclareMathOperator{\ad}{ad}

\newcommand{\ot}{\otimes}
\newcommand{\tp}[1]{^{\otimes #1}}    

\newcommand{\Cl}{\mathbb{C}}
\newcommand{\Rl}{\mathbb{R}}
\newcommand{\Nl}{\mathbb{N}}
\newcommand{\Zl}{\mathbb{Z}}

\def\C{\mathbb{C}}

\def\N{\mathbb{N}}
\def\Z{\mathbb{Z}}
\def\T{\mathbb{T}}

\newcommand{\Hil}{\mathcal{H}}

\newcommand{\Om}{\Omega}
\newcommand{\om}{\omega}
\newcommand{\la}{\lambda}
\newcommand{\eps}{\varepsilon}

\newcommand{\CA}[0]{\mathcal{A}} \newcommand{\CB}[0]{\mathcal{B}}
 \renewcommand{\CD}[0]{\mathcal{D}}
 \newcommand{\CF}[0]{\mathcal{F}}

\newcommand{\CK}[0]{\mathcal{K}} \newcommand{\CL}[0]{\mathcal{L}}
\newcommand{\CM}[0]{\mathcal{M}} \newcommand{\CN}[0]{\mathcal{N}}
\newcommand{\CO}[0]{\mathcal{O}} 
 \newcommand{\CR}[0]{\mathcal{R}}
 
\newcommand{\CU}[0]{\mathcal{U}}

\usepackage{scalerel,stackengine}
\def\apeqA{\SavedStyle\sim}
\def\Msim{\setstackgap{L}{\dimexpr.5pt+1.5\LMpt}\ensurestackMath{%
  \ThisStyle{\mathrel{\Centerstack{{\apeqA} {\apeqA} {\apeqA}}}}}}
\def\Nsim{\approx}

\newcommand{\bigboxplus}{
  \mathop{
    \vphantom{\bigoplus} 
    \mathchoice
      {\vcenter{\hbox{\resizebox{\widthof{$\displaystyle\bigoplus$}}{!}{$\boxplus$}}}}
      {\vcenter{\hbox{\resizebox{\widthof{$\bigoplus$}}{!}{$\boxplus$}}}}
      {\vcenter{\hbox{\resizebox{\widthof{$\scriptstyle\oplus$}}{!}{$\boxplus$}}}}
      {\vcenter{\hbox{\resizebox{\widthof{$\scriptscriptstyle\oplus$}}{!}{$\boxplus$}}}}
  }\displaylimits 
}

\makeatletter
\patchcmd{\@maketitle}
  {\ifx\@empty\@dedicatory}
  {\ifx\@empty\@date \else {\vskip3ex \centering\footnotesize\@date\par\vskip1ex}\fi
   \ifx\@empty\@dedicatory}
  {}{}
\patchcmd{\@maketitle}   
  {\ifx\@empty\@date\else \@footnotetext{\@setdate}\fi}
  {}{}{}
\makeatother

\begin{document}

 \vspace*{-7mm}
\title[Yang-Baxter endomorphisms]{Yang-Baxter endomorphisms}

\author[R.~Conti]{Roberto Conti}
\address[RC]{Dipartimento SBAI
\\ Sapienza Universit\`{a} di Roma \\ Italy}
\email{roberto.conti@sbai.uniroma1.it}

\author[G.~Lechner]{Gandalf Lechner}
\address[GL]{ School of Mathematics\\ Cardiff University  \\ UK}
\email{LechnerG@cardiff.ac.uk}

\begin{abstract}
    Every unitary solution of the Yang-Baxter equation (R-matrix) in dimension $d$ can be viewed as a unitary element of the Cuntz algebra $\CO_d$ and as such defines an endomorphism of $\CO_d$. These Yang-Baxter endomorphisms restrict and extend to endomorphisms of several other $C^*$- and von Neumann algebras and furthermore define a II$_1$ factor associated with an extremal character of the infinite braid group. This paper is devoted to a detailed study of such Yang-Baxter endomorphisms.
    
    Among the topics discussed are characterizations of Yang-Baxter endomorphisms and the relative commutants of the various subfactors they induce, an endomorphism perspective on algebraic operations on R-matrices such as tensor products and cabling powers, and properties of characters of the infinite braid group defined by R-matrices. In particular, it is proven that the partial trace of an R-matrix is an invariant for its character by a commuting square argument.
    
    Yang-Baxter endomorphisms also supply information on R-matrices themselves, for example it is shown that the left and right partial traces of an R-matrix coincide and are normal, and that the spectrum of an R-matrix can not be concentrated in a small disc. Upper and lower bounds on the minimal and Jones indices of Yang-Baxter endomorphisms are derived, and a full characterization of R-matrices defining ergodic endomorphisms is given.
    
    As examples, so-called simple R-matrices are discussed in any dimension~$d$, and the set of all Yang-Baxter endomorphisms in $d=2$ is completely analyzed.
\end{abstract}

\date{September 9, 2019}
\maketitle

\vspace*{-8mm}
\tableofcontents

\section{Introduction}

This article is motivated by two circles of questions --- one pertaining to the Yang-Baxter equation and one to endomorphisms of the Cuntz algebras and related operator algebras --- that are brought into contact by the so-called Yang-Baxter endomorphisms. As the name suggests, these are endomorphisms of various $C^*$- and von Neumann algebras, as explained below, defined by unitary solutions of the Yang-Baxter equation.

To introduce the subject, recall that the Yang-Baxter equation (YBE) is a cubic equation for an endomorphism $R\in\End(V\ot V)$ of the tensor square of a vector space $V$, namely
\begin{align}\label{eq:YBE--intro}
    (R\ot\id_V)(\id_V\ot R)(R\ot \id_V)=(\id_V\ot R)(R\ot \id_V)(\id_V\ot R).
\end{align}
This equation and its solutions play a prominent role in many different areas of physics and mathematics. It has its origins in statistical mechanics and quantum mechanics \cite{Yang:1967,Baxter:1972}, but is long since known to also be closely connected to braid group representations and knot theory \cite{Jones:1987,Turaev:1988}, von Neumann algebras and subfactors \cite{Jones:1983_2}, and braided categories 
\cite{Longo:1992,TubaWenzl:2005_2,EvansPugh:2012,GiorgettiRehren:2018}. Representations of quantum groups \cite{Drinfeld:1986,Jimbo:1986} are a rich source of solutions for the Yang-Baxter equation.

In many of these fields, one is mostly interested in the case that $V$ is a finite-dimensional Hilbert space and $R$ is a {\em unitary} solution of \eqref{eq:YBE--intro}. Also in the present article, we will only be concerned with such {\em R-matrices}, henceforth always assumed to be unitary, and refer to $d:=\dim V$ as the dimension of $R$. The set of all R-matrices of dimension $d$ will be denoted $\CR(d)$.

Unitary R-matrices are of great interest in several applications to quantum physics. For example, in topological quantum computation they serve as quantum gates \cite{KauffmanLomonacoJr:2004_2,BurtonGould:2006,RowellWang:2012}, and in the context of integrable quantum field theories on two-dimensional Minkowski space, unitary solutions of a more involved Yang-Baxter equation involving a spectral parameter play the role of two-particle collision operators \cite{AbdallaAbdallaRothe:2001}. Unitary solutions of \eqref{eq:YBE--intro}, without spectral parameter, then describe the structure of short distance scaling limits of such theories \cite{LechnerScotford:2019}. 

Furthermore, as will be explained further below, R-matrices give rise to certain endomorphisms of von Neumann algebras that share many structural properties with endomorphisms appearing in quantum field theories with braid group statistics \cite{Frohlich:1988_3,FredenhagenRehrenSchroer:1989,Longo:1992}.

\medskip

Despite this widespread interest in the Yang-Baxter equation, only relatively little is known about its solutions, and in particular about its unitary solutions, which are very difficult to find in general. In dimension $d=2$, all solutions are known \cite{Hietarinta:1992_8} but already for $d=3$, this is no longer the case. For special classes of solutions, see e.g. \cite{GoldschmidtJones:1989,Bytsko:2019}.

The only general class of R-matrices that seems to be under good control are the {\em involutive} R-matrices (that is, $R^2=1$) which have recently been completely classified by one of us \cite{LechnerPennigWood:2019} up to an equivalence relation originating from algebraic quantum field theory \cite{AlazzawiLechner:2016}. This classification relied crucially on the fact that involutive R-matrices define extremal characters of the infinite symmetric group, a classification of which is known \cite{Thoma:1964}.

This state of affairs provides one of the main motivations for this article: To develop tools that can be used to understand the set of R-matrices in the vastly more general non-involutive case. Although often times the braid group representations associated with an R-matrix are emphasized, these are by no means the only interesting algebraic structure attached to an R-matrix, and in this article, our focus is on certain endomorphisms and subfactors defined by~$R$. 

\medskip

In order to introduce these endomorphisms, we recall some facts about the Cuntz algebras, see Section~\ref{section:RandCuntz} for precise definitions and details. The Cuntz algebras $\CO_d$, $d\in\{2,3,\ldots\}$ \cite{Cuntz:1977} are a family of $C^*$-algebras that play a prominent role in various fields -- for example, in superselection theory and duality for compact groups \cite{DoplicherRoberts:1987}, wavelets \cite{BratteliJorgensen:1999}, and twisted cyclic cocycles in noncommutative geometry \cite{CareyPhillipsRennie:2008_2}, to name just a very few samples from different areas.

There are two fundamental features of $\CO_d$ that underlie the main concept of this article: First, its unitary elements $u\in\CU(\CO_d)$ are in bijection with its (unital, ${}^*$-) endomorphisms $\la_u\in\End(\CO_d)$ \cite{Cuntz:1980}. As $\CO_d$ is a simple $C^*$-algebra, these are automatically injective. Second, the Cuntz algebra $\CO_d$ can be thought of as being generated by a $d$-dimensional Hilbert space $V$, namely it contains all linear maps $V\tp{n}\to V\tp{m}$, $n,m\in\Nl_0$. In particular, there is a UHF subalgebra $\CF_d$ isomorphic to the infinite $C^*$-tensor product of $\End V$. 

In view of these facts, we may view an R-matrix $R$, which is in particular a unitary element of $\End (V\ot V)$, as a unitary in $\CO_d$ (with $d=\dim V$) and consider the corresponding endomorphism $\la_R\in\End\CO_d$.They will be called {\em Yang-Baxter endomorphisms}, and their analysis is the main subject of this paper.

The Cuntz algebra $\CO_d$ can be completed in a natural way to a type III$_{1/d}$ factor $\CM$, and its subalgebra $\CF_d$ completes to a type II$_1$ factor $\CN\subset\CM$. Any endomorphism of the form $\la_u$ with $u\in\CU(\CF_d)$ leaves the UHF subalgebra $\CF_d\subset\CO_d$ invariant, extends to endomorphisms of their weak closures $\CM$ and $\CN$ (all denoted by the same symbol $\la_u$), and thus provides us with the subfactors
\begin{align}\label{eq:subfactor1-intro}
	\la_u(\CM)\subset\CM,\qquad 
	\la_u(\CN)\subset\CN.
\end{align}
These and related subfactors have been studied by several researchers, often times with the aim of determining their indices \cite{ContiPinzari:1996,Akemann:1997,Izumi:1993}.

Whereas general endomorphisms of Cuntz algebras have a very rich structure with many different facets \cite{ContiRordamSzymanski:2010,ContiSzymanski:2011}, Yang-Baxter endomorphisms (that is, $u=R\in\CR(d)$) and their subfactors have more special properties. For instance, as an additional structure present in the Yang-Baxter case there is a von Neumann algebra $\CL_R\subset\CN$ generated by the braid group representation associated with $R$, and $\la_R$ restricts to the canonical endomorphism $\varphi$ on $\CL_R$. We will show that $\CL_R$ is a factor, so that any R-matrix $R$ provides us with yet another subfactor
\begin{align}\label{eq:subfactor2-intro}
    \varphi(\CL_R)\subset\CL_R.
\end{align}
We are thus in a situation where to any R-matrix we may associate various operator-algebraic structures, derived from their endomorphisms. On the one hand, these data provide interesting invariants of R-matrices (such as Jones indices, commuting squares, fixed point algebras, etc.) that go beyond the trivial spectral and dimension data of the R-matrix itself. On the other hand, the analysis of Yang-Baxter endomorphisms contributes to the understanding of endomorphisms of $\CO_d$ in general, which is an area in full swing on its own.

\bigskip

Since this is a long article, we now give a fairly detailed overview of its contents and main results. 

{\bf Section~\ref{section:RandCuntz}} introduces R-matrices, Cuntz algebras, and the associated von Neumann algebras $\CL_R\subset\CN\subset\CM$ in more detail. We recall in particular that if one takes $R$ to be one of the most basic R-matrices, namely the tensor flip~$F$, one obtains the canonical endomorphism $\varphi=\la_F\in\End\CO_d$, acting as a shift on the UHF subalgebra. Drawing on the interplay of $\la_R$ and $\varphi$, we give three different characterizations of the subset of Yang-Baxter endomorphisms of $\End\CO_d$ (Prop.~\ref{prop:ybe-od}), two of which are due to Cuntz \cite{Cuntz:1998} and one of us \cite{ContiHongSzymanski:2012}, respectively. A notable feature is that a Yang-Baxter endomorphism is an automorphism if and only if $R$ is a multiple of the identity (Cor.~\ref{cor:no-auto}).

With the framework set up in this manner, we consider in {\bf Section~\ref{section:towers}} the three towers of relative commutants of the subfactors \eqref{eq:subfactor1-intro} (for $u=R\in\CR(d)$) and \eqref{eq:subfactor2-intro}. We give explicit characterizations of all three relative commutants. The characterizations of the relative commutants of \eqref{eq:subfactor1-intro} rely strongly on results from \cite{ContiPinzari:1996,Akemann:1997,Longo:1994_2}, but the characterization of the relative commutant $\CL_{R,n}:=\varphi^n(\CL_R)'\cap\CL_R$ (Prop.~\ref{prop:LRn}) is new: We characterize it as an intersection of $\CL_R$ with a matrix algebra, and as the fixed point algebra of~$\CL_R^{\la_{\varphi^n(R)}}$, reminiscent of work of Gohm and K\"ostler in noncommutative probability \cite{GohmKostler:2009_2}.

The section concludes with a structural result on the algebras $\CL_{R,n}$: For any $n\in\Nl$, the diagrams
\begin{align}\label{eq:commutingsquare2-pre}
    \begin{array}{ccc}
        \CF_d^n &\subset &\CN
        \\
        \cup & & \cup
        \\
        \CL_{R,n} &\subset &\CL_R
    \end{array}
    \qquad\qquad\qquad
    \begin{array}{ccc}
        \varphi^n(\CN) &\subset &\CN
        \\
        \cup & & \cup
        \\
        \varphi^n(\CL_R) &\subset &\CL_R
    \end{array}
\end{align}
are commuting squares (Thm.~\ref{thm:commuting-squares}), where $\CF_d^n$ is the subalgebra of $\CF_d$ isomorphic to $\End V\tp{n}$. This implies in particular that the left inverses of $\la_R$ and $\varphi$ coincide on $\CL_R$, and is later used as a basic tool for computing braid group characters and invariants for $R$.
    
{\bf Section~\ref{section:operations}} discusses three algebraic operations on the set of all R-matrices: A tensor product, Wenzl's cabling powers \cite{Wenzl:1990}, and a kind of direct sum. We relate these operations on R-matrices $R$ to operations on Yang-Baxter endomorphisms $\la_R$: The tensor product of R-matrices turns out to correspond to the tensor product of endomorphisms (on the level of the II$_1$ factor $\CN$) and the cabling power $R^{(n)}$ turns out to correspond to the $n$-fold power $\la_R^n$ (again, on the type II$_1$ factor). At the time of writing, our understanding of the ``box sum'' $R\boxplus S$ on the level of endomorphisms is still incomplete, but we show how it is reflected in the relative commutant of $\la_{R\boxplus S}$.
    
In {\bf Section~\ref{section:equivalence}} we introduce three equivalence relations on R-matrices $R,S\in\CR(d)$, each of which formalizes that one of their subfactors \eqref{eq:subfactor1-intro}, \eqref{eq:subfactor2-intro} are equivalent. Several different scenarios for these equivalences are discussed. 
The equivalence relation relating to the $\CL_R$-subfactor \eqref{eq:subfactor2-intro}, denoted $\sim$, is taken from \cite{LechnerPennigWood:2019} and shown to exactly capture the braid group character defined by~$R$. We compare with the classification of involutive R-matrices in {\bf Section~\ref{subsection:EquivalencesAndPartialTraces}} and prove that equivalent R-matrices $R\sim S$ have similar partial traces. In this context, we also show that the left and right partial traces of an R-matrix always coincide and are normal (Thm.~\ref{theorem:phiR}), which provides direct information on the R-matrices themselves.

{\bf Section~\ref{section:irreducibility}}: As a unital normal endomorphism of the type III factor $\CM$ with finite-dimensional relative commutant, a Yang-Baxter endomorphism can be decomposed into finitely many irreducible endomorphisms of $\CM$, unique up to inner automorphisms (i.e. as sectors in quantum field theory language) \cite{Longo:1989,Longo:1991}. The main difficulty is that the decomposition of a Yang-Baxter endomorphism does typically not respect the Yang-Baxter equation, that is, its irreducible components are no longer of Yang-Baxter form. Nonetheless, such a decomposition provides information on the underlying R-matrix; for example we find upper and lower bounds on the minimal and Jones indices of the subfactors \eqref{eq:subfactor1-intro} in terms of spectral data of $R$ and its partial trace (Cor.~\ref{corollary:spectralindexbounds}). Another corollary is that an R-matrix whose eigenvalues are concentrated in a sufficiently small disk around 1 is necessarily the identity (Cor.~\ref{corollary:no-spectral-concentration}).

In {\bf Section~\ref{section:reduction-of-involutives}}, we present a reduction scheme that does respect the Yang-Baxter structure and works directly on the level of the R-matrix by restricting it to tensor product subspaces defined by projections in the relative commutant $\la_R(\CM)'\cap\CM$. This scheme is currently under control for the special class of involutive R-matrices and sheds new light on the classification of involutive R-matrices from the point of view of endomorphisms.

{\bf Section~\ref{section:ergodicity}} is about fixed points of Yang-Baxter endomorphisms. Our first result in this direction is that on the level of the type II factor $\CN$, the relative commutant $\CL_R'\cap\CN$ coincides with the fixed point algebra $\CN^{\la_R}$ (Prop.~\ref{proposition:generalfixedpointresults}). Moreover, $\la_R$ is ergodic as an endomorphism of $\CM$ if and only if it is ergodic in restriction to $\CN$ (Prop.~\ref{prop:ergodicityFvsO}). This structure enables us to obtain a clear picture of ergodicity and fixed point algebras for Yang-Baxter endomorphisms which is not known for general elements of $\End\CO_d$ or $\End\CM$. In particular, we give a complete characterization of ergodic Yang-Baxter endomorphisms in Thm.~\ref{thm:ergodicity} in terms of a condition that only involves the adjoint action of $R$ on $\End V$. We also explain that ergodicity on the level of the $C^*$-algebras $\CO_d$ or $\CF_d$ is quite different from ergodicity on the level of the corresponding von Neumann algebras $\CM$ or $\CN$.

The article concludes in {\bf Section~\ref{section:2dRmatrices}}, devoted to an analysis of the family of all R-matrices of dimension $d=2$. Strengthening a theorem of Dye \cite{Dye:2003_2} (building on Hietarinta's classical \cite{Hietarinta:1993_3_2}), we show that $\CR(2)$ is the disjoint union of four families that could be called trivial R-matrices, diagonal R-matrices, off-diagonal R-matrices, and a special case (see Thm.~\ref{thm:2dRmatrices} for details). We then use the results of the previous sections to analyse the properties of the corresponding endomorphisms in detail. In particular, we discuss the special case, an R-matrix that has appeared in various places in the literature (see, for example \cite{ContiFidaleo:2000,FrankoRowellWang:2006,RehrenSchroer:1987}), explain why it is special from the point of view of endomorphisms, and compute its (infinite-dimensional) fixed point algebra~$\CN^{\la_R}$.

\bigskip

As mentioned before, we expect that the results in this article will be important for the classification of R-matrices, or a more detailed analysis of the structure of $\CR=\bigcup_{d\in\Nl}\CR(d)$, a topic that is not touched upon in the present work. Another interesting aspect not covered is the $C^*$-tensor category naturally generated by an R-matrix, to which we hope to return in a future investigation.

\section{R-matrices and Cuntz algebras}\label{section:RandCuntz}

The algebraic structures investigated in this article are all derived from unitary solutions of the Yang-Baxter equation (YBE), which we will refer to as {\em R-matrices}. 

\begin{definition}\label{def:R}
    Let $V$ be a finite dimensional Hilbert space. An {\em R-matrix on $V$} is a unitary $R:V\ot V\to V\ot V$ such that
    \begin{align}\label{eq:YBE}
        (R\ot\id_V)(\id_V\ot R)(R\ot \id_V)=(\id_V\ot R)(R\ot \id_V)(\id_V\ot R).
    \end{align}
    The {\em dimension of $R$} is defined as $\dim R:=\dim V$. The set of all R-matrices on Hilbert spaces of dimension $d\in\Nl$ is denoted $\CR(d)$, and the set of all R-matrices (of any dimension) is denoted $\CR$.
\end{definition}
Many examples of R-matrices exist, but the general structure of $\CR$ is not known. Very simple R-matrices that can be produced in any dimension are multiples of the identity, $R=q\cdot1$ (such R-matrices will be called {\em trivial}), and multiples of the tensor flip, i.e. $R=q\cdot F$, where $F(v\ot w)=w\ot v$, $v,w\in V$. Here $q$ lies in $\T$, the unit circle in the complex plane\footnote{A richer class of examples is presented in Def.~\ref{def:diagonalR}}.

As is well known and will be recalled later, any $R\in\CR$ defines representations of the braid groups. However, this is by no means the only interesting algebraic structure attached to an R-matrix, and in this article, we emphasize certain endomorphisms and subfactors defined by $R$. To introduce these, we have to recall some well-known facts about Cuntz algebras.

\medskip

The Cuntz algebra $\CO_d$, $d\in\Nl$, is the unital $C^*$-algebra generated by $d$ isometries $S_1,\ldots,S_d$ such that $S_i^*S_j=\delta_{ij}1$ and $\sum_{i=1}^dS_iS_i^*=1$ \cite{Cuntz:1977}. Using standard notation for multi indices $\mu = (\mu_1,\ldots,\mu_n)$, we set $S_\mu:=S_{\mu_1} \cdots S_{\mu_n}$ and refer to $|\mu|:=n$ as the length of $\mu$.

The subalgebra $\CF_d^n:=\newspan\{S_\mu S_\nu^*\,:\,|\mu|=|\nu|=n\}$ is naturally isomorphic to the $n$-fold tensor power $M_d^{\tp n}$ of the full matrix algebra\footnote{We will suppress this isomorphism in our notation. For instance, the matrix units $E_{ij}\in M_d$ are identified with the Cuntz algebra elements $S_iS_j^*\in\CF_d^1$, and $R\in M_d\ot M_d$ is identified with $R=\sum_{i,j,k,l=1}^dR^{ij}_{kl}S_iS_jS_l^*S_k^*\in\CF_d^2$, where $R^{ij}_{kl}=\langle e_i\ot e_j,R(e_k\ot e_l)\rangle$ and $\{e_i\}_{i=1}^d$ is the standard basis of $\Cl^d$.} $M_d$. In particular, we may view R-matrices $R\in\CR(d)$ as elements of $\CF_d^2\subset\CO_d$. The norm closure of the increasing family $\CF_d^n\subset\CF_d^{n+1}\subset\ldots$ is a UHF algebra of type $d^\infty$ which we denote $\CF_d$.

An important feature of $\CO_d$ that we will rely on throughout is that its unitary elements $u\in\CU(\CO_d)$ are in bijection with its (unital, injective) endomorphisms $\la_u\in\End\CO_d$ \cite{Cuntz:1980}. On generators, the endomorphism $\la_u$ corresponding to $u\in\CU(\CO_d)$ is defined by
\begin{align*}
 \la_u(S_i):=uS_i,
 \end{align*}
and every endomorphism of $\CO_d$ is of this form. 

We can now introduce our central object of interest.

\begin{definition}
    A {\em Yang-Baxter endomorphism of $\CO_d$} is an endomorphism of the form~$\la_R$, $R\in\CR(d)$.
\end{definition}

An important example is the so-called {\em canonical endomorphism} $\varphi:=\la_F$ given by the flip $F$, which takes the explicit form $\varphi(x)=\sum_{i=1}^d S_i x S_i^*$, $x\in\CO_d$. This endomorphism satisfies $S_i x = \varphi(x) S_i$ for all $x \in \CO_d$ and $i=1,\ldots,d$,  and restricts to the one-sided shift $x\mapsto\id_{M_d}\ot x$ on the infinite tensor product UHF algebra $\CF_d \simeq M_d \otimes M_d \otimes \ldots$, which indicates its relevance for R-matrices in view of \eqref{eq:YBE}. In fact, the YBE takes the form 
\begin{align}\label{eq:YBE-Cuntz-Form}
    R\varphi(R)R=\varphi(R)R\varphi(R)
\end{align}
when $R$ is viewed as an element of $\CF_d^2\subset\CO_d$.

\medskip

Without further mentioning, we will often use two basic consequences of the definition of $\la_u$ (for general unitary $u\in\CU(\CO_d)$) and $\varphi$: The composition law 
\begin{align}
 \la_u\la_v=\la_{\la_u(v)u}
 ,\qquad 
 u,v\in\CU(\CO_d),
\end{align}
and an explicit formula for the action of $\la_u$ on $\CF_d^n$: Given arbitrary unitary $u \in \CU(\CO_d)$ and an integer $n\geq 1$, we define two elements of $\CF_d^{n+1}$,
\begin{align}\label{eq:Un-nU}
 u_n &:= u \varphi(u) \cdots \varphi^{n-1}(u),\qquad 
 {}_nu:=\varphi^{n-1}(u)\cdots u=({u^*}_n)^*
\end{align}
and see that 
\begin{align}\label{eq:lau-explicit}
	\la_u(x) &= (\ad u_n)(x)               &&\hspace*{-3cm}\text{for } x\in\CF_d^k,\quad n\geq k,\\
	\la_u(x) &= \lim_{n\to\infty}(\ad u_n)(x) &&\hspace*{-3cm}\text{for } x\in\CF_d.
 \label{eq:lau-limit-formula}
\end{align}
The latter limit exists in the norm topology of $\CO_d$ \cite{Cuntz:1998}, and we note that for $u\in\CF_d$, the endomorphism $\la_u$ leaves $\CF_d$ invariant, $\la_u(\CF_d)\subset\CF_d$.

We now recall some properties and characterizations of Yang-Baxter endomorphisms and add a new one. 

\begin{proposition}\label{prop:ybe-od}
    Let $R \in \CU(\CF_d^2)$. The following conditions are equivalent:
    \begin{enumerate}
        \item\label{item:YBE} $R\in\CR(d)$, namely $R \varphi(R) R = \varphi(R) R \varphi(R)$,
        \item\label{item:Cuntz-YBE} $\lambda_R(R) = \varphi(R)$ \cite{Cuntz:1998},
        \item\label{item:RinL2L2} $R$ commutes with every element $x\in\la_R^2(\CO_d)$ \cite{ContiHongSzymanski:2012},
        \item\label{item:LR2} $\lambda_R^2 = \lambda_{\varphi(R)R}$.
    \end{enumerate}
\end{proposition}
\begin{proof}
    The equivalence \ref{item:YBE}$\iff$\ref{item:Cuntz-YBE} was shown in \cite{Cuntz:1998}, and \ref{item:YBE}$\iff$\ref{item:RinL2L2} was shown in \cite{ContiHongSzymanski:2012}. To show \ref{item:YBE}$\iff$\ref{item:LR2}, note that one has for any $R\in\CU(\CF_d^2)$
    \begin{align*}
        \lambda_R^2 = \lambda_{\lambda_R(R) R} = \lambda_{R \varphi(R) R \varphi(R)^* R^* R} = \lambda_{R \varphi(R) R \varphi(R)^*}\,, 
    \end{align*}
    which coincides with $\la_{\varphi(R)R}$ if and only if $R \varphi(R) R \varphi(R)^* = \varphi(R)R$. This condition is clearly equivalent to the YBE as expressed in \ref{item:YBE}.  
\end{proof}

\begin{remark}
	Characterization \ref{item:RinL2L2} could be phrased as $R\in(\la_R^2,\la_R^2)$ in standard notation for intertwiner spaces for endomorphisms, which emphasizes the similarity of our setup to algebraic quantum field theory and subfactors \cite{DoplicherHaagRoberts:1971,Longo:1992,FredenhagenRehrenSchroer:1989}. We reserve this notation for a von Neumann algebraic version introduced later on.
\end{remark}

It is a natural question to ask whether Yang-Baxter endomorphisms can be automorphisms, i.e. surjective. Whereas it is well known and easy to check that for $u\in\CF_d^1$, the associated endomorphism $\la_u$ is an automorphism\footnote{These automorphisms are usually referred to as quasi-free automorphisms~\cite{Evans:1980}.}, with inverse $\la_u^{-1}=\la_{u^*}$, the problem to recognize which endomorphisms $\la_u$ are automorphisms is delicate in general \cite{ContiSzymanski:2011}. For Yang-Baxter endomorphisms the answer is however a straightforward consequence of Prop.~\ref{prop:ybe-od}~\ref{item:RinL2L2} \cite{ContiHongSzymanski:2012}.

\begin{corollary}\label{cor:no-auto}
    A Yang-Baxter endomorphism $\la_R$ is an automorphism if and only if $R$ is trivial. 
\end{corollary}
\begin{proof}
    If $\la_R$ is an automorphism, then so is $\la_R^2$, and hence $\la_R^2(\CO_d)=\CO_d$. But $R$ commutes with $\la_R^2(\CO_d)$, and $\CO_d$ has trivial center. Hence $R$ is trivial. The other direction is evident.
\end{proof}

With the help of the canonical endomorphism $\varphi$, we may also conveniently introduce the previously mentioned braid group representations associated with $R\in\CR(d)$. Let $B_n=\langle b_1,\ldots,b_{n-1}\rangle$ denote the braid group on $n$ strands with its standard Coxeter generators~$b_i$, and let $B_\infty$ denote the infinite braid group, namely the inductive limit of the family $B_n\subset B_{n+1}\subset\ldots$. Given $R\in\CR(d)$, the multiplicative extension of 
\begin{align}
 \rho_R(b_k):=\varphi^{k-1}(R)\in\CF_d^{k+1}\subset\CF_d
 , \qquad k\in\Nl,
\end{align}
is a group homomorphism $\rho_R:B_\infty\to\CU(\CF_d)$. We will frequently consider the $C^*$-algebra generated by $\rho_R$, namely
\begin{align}\label{eq:BR}
 \CB_R
 :=
 C^*\{\varphi^n(R)\,:\,n\in\Nl_0\}
 \subset
 \CF_d,
\end{align}
and the closely related $C^*$-algebras
\begin{align}\label{eq:AR}
 \CA_R
 :=
 \{x\in\CO_d\,:\,\la_R(x)=\varphi(x)\}
 ,\qquad
 \CA_R^{(0)}
 :=
 \CA_R\cap\CF_d
 .
\end{align}

\begin{lemma}\label{lemma:YBEndos}
    Let $R \in \CR(d)$ be an R-matrix and $\la_R$ its corresponding Yang-Baxter endomorphism.
    \begin{enumerate}
        \item\label{item:ARinclusions} $\CB_R\subset\CA_R^{(0)}$, i.e. 
        \begin{align}\label{eq:lambdaRrestrictstovarphionB}
         \la_R(x)=\varphi(x),\qquad x\in\CB_R.
        \end{align}
        \item\label{item:lambdaRrestrictstovarphi} $\la_R$ restricts to an endomorphism of $\CF_d$, $\CA_d$, $\CA_d^{(0)}$, and $\CB_R$.
        \item\label{item:PowersofLaR} For any $n\in\Nl$, one has
        \begin{align}\label{eq:PowersOfLaR}
        \lambda_R^n 
        = 
        \lambda_{{}_nR}
        =
        \lambda_{\rho_R(b_n\cdots b_1)},\qquad n\in\Nl.
    \end{align}
    \end{enumerate}
\end{lemma}
\begin{proof}
    We first prove that $\la_R$ restricts to $\CA_R$, and to this end recall that for general $u\in\CU(\CO_d)$, one has $\la_u\circ\varphi=\ad u\circ\varphi\circ\la_u$. This implies that if $x\in\CO_d$ satisfies $\lambda_R(x)=\varphi(x)$, then 
	\begin{equation}\label{eq:lambdaR-AR}
	\begin{aligned}
	 \lambda_R(\lambda_R(x)) 
	 &= 
	 \lambda_R(\varphi(x)) 
	 = 
	 R\varphi(\lambda_R(x))R^* 
	 \\
	 &=
	 R\varphi(\varphi(x))R^* 
	 =
	 \varphi(\varphi(x)) = \varphi(\lambda_R(x))
	 ,
	\end{aligned}
    \end{equation}
	where the next to last step follows from the general fact that $\CF_d^n$ commutes with $\varphi^n(\CO_d)$.
 
	This argument yields $\lambda_R(\CA_R)\subset\CA_R$. As $R\in\CF_d$, we also have $\la_R(\CF_d)\subset\CF_d$ and therefore $\lambda_R(\CA_R^{(0)})\subset\CA_R^{(0)}$ as well.
	
	Regarding $\CB_R$, the argument \eqref{eq:lambdaR-AR} can be used to prove $\la_R^n(R)=\varphi^n(R)$ by induction in $n\in\Nl$, the case $n=1$ being settled by~Prop.~\ref{prop:ybe-od}~\ref{item:Cuntz-YBE}. This implies, $n\in\Nl_0$,
    \begin{align*}
     \la_R(\varphi^n(R))=\la_R^{n+1}(R)=\varphi^{n+1}(R)=\varphi(\varphi^n(R)).
    \end{align*}
    As $\CB_R$ is generated by $\varphi^n(R)$, $n\in\Nl_0$, we have shown both \ref{item:ARinclusions} and \ref{item:lambdaRrestrictstovarphi}.
    
    For \ref{item:PowersofLaR}, we note that ${}_nR=\varphi^{n-1}(R)\cdots R=\rho_R(b_n\cdots b_1)$ by definition of ${}_nR$ and~$\rho_R$, and carry out another induction in $n$ to show $\la_R^n=\la_{\rho_R(b_n\cdots b_1)}$. In fact, $\la_R^{n+1}=\la_R\la_{\rho_R(b_n\cdots b_1)}=\la_{\la_R(\varphi^{n-1}(R)\cdots R)R}=\la_{\varphi^n(R)\cdots R}=\la_{\rho_R(b_{n+1}\cdots b_1)}$.    
\end{proof}

\begin{remark}
	For general $R$, the algebra $\CA_R$ is not contained in $\CF_d$ (take $R=F$ with $\CA_F=\CO_d$ as a counterexample). We also mention that our later results will imply that in general, $\CA_R^{(0)}$ is strictly larger than $\CB_R$. In the special case $R=F$, the $C^*$-algebra $\CB_F$	has been shown in \cite{DoplicherRoberts:1987} to be equal to $\CO_{U(d)}$, namely the fixed point algebra of $\CO_d$ under the canonical action of $U(d)$ by quasi-free automorphisms.
\end{remark}
 
Any R-matrix defines several $C^*$-algebra inclusions, namely $\la_R(\CO_d)\subset\CO_d$, $\la_R(\CF_d)\subset\CF_d$, $\la_R(\CB_R)=\varphi(\CB_R)\subset\CB_R$, etc. We now recall further structure that will allow us to promote these inclusions to subfactors of von Neumann algebras.

Trivial R-matrices $R=d^{-it}1$, $t\in\Rl$, define a $\frac{2\pi}{\log d}$-periodic one-parameter group of automorphisms $\sigma_t:=\la_{d^{-it}1}$ of $\CO_d$, and we define the spectral subspaces 
\begin{align}\label{eq:SpectralSubspaces}
 \CO_d^{(n)}:=\{x\in\CO_d\,:\,\sigma_t(x)=d^{-itn}x\},\qquad n\in\Zl.
\end{align}
Sometimes it will be more convenient to work with a rescaled version of $\sigma$, namely the $(2\pi)$-periodic gauge action $\alpha_t:=\sigma_{-t/\log d}=\la_{e^{it}}$.

One has $\CO_d^{(0)}=\CF_d$, and $E^0:\CO_d\to\CF_d$, $E^0(x):=\frac{1}{2\pi}\int_0^{2\pi}\alpha_t(x)dt$ is a conditional expectation onto the UHF subalgebra.

Viewing $\CF_d$ as an infinite tensor product, we have the canonical normal normalized trace state $\tau:\CF_d\to\Cl$, and define $\om:=\tau\circ E^0$. This is a KMS state on $\CO_d$ with modular group $\sigma_t$, $t\in\Rl$, and we denote the von Neumann algebras generated by its GNS representation $(\pi_\om,\Hil_\om,\Om_\om)$ as
\begin{align}
  \CM:=\pi_\om(\CO_d)''\,,\qquad 
  \CN:=\pi_\om(\CF_d)''\subset\CM. 
\end{align}
It is well known that $\CM$ is a factor of type III$_{1/d}$ and $\CN$ is a factor of type II$_1$. We will use the same symbols $\om$, $\tau$ and $E^0:\CM\to\CN$ \cite{Haagerup:1989} to denote the extensions of these maps to the weak closures $\CM$ and $\CN$. 

For our purposes, it is important to note that for any $u\in\CF_d$ (and in particular, for any R-matrix), the corresponding endomorphism $\la_u$ extends to a normal endomorphism of $\CM$ leaving $\om$ invariant \cite{Longo:1994_2}. Also here, we will use the same symbol for the extension. 

To complete the picture, we also introduce the von Neumann algebra $\CL_R$ generated by the $C^*$-algebra $\CB_R$ corresponding to some R-matrix $R\in\CR$, i.e. 
\begin{align}
 \CL_R:=\pi_\om(\CB_R)''\subset\CN\subset\CM.
\end{align}
As an immediate consequence of \eqref{eq:lambdaRrestrictstovarphionB}, we observe 
\begin{align}
	\la_R|_{\CL_R} = \varphi|_{\CL_R}.
\end{align}

Further structural elements relevant for our analysis are conditional expectations and left inverses. Because $\la_R$ commutes with the modular group, Takesaki's theorem provides us with a unique $\om$-preserving conditional expectation $E_R:\CM\to\la_R(\CM)$, which is faithful and normal and has the form 
\begin{align}
 E_R=\la_R\circ\phi_R
\end{align}
with $\phi_R$ the corresponding $\om$-preserving left inverse of $\la_R$. Recall that $\phi_R:\CM\to\CM$ is a completely positive normal linear map that satisfies
\begin{align}\label{eq:phiRproperty}
 \phi_R(\la_R(x)y\la_R(z))=x\phi_R(y)z,\qquad x,y,z\in\CM.
\end{align}
These properties of $\phi_R$ and the limit formula \eqref{eq:lau-limit-formula} imply 
\begin{align}
 \phi_R(x)
 =
 \wlim_{n\to\infty}{R_n}^*xR_n,\qquad x\in\CN.
\end{align}
As $R_n\in\CL_R\subset\CN$, this yields in particular
\begin{align}\label{eq:phiRpreserves}
 \phi_R(\CN)=\CN,\qquad \phi_R(\CL_R)=\CL_R.
\end{align}
The left inverse $\phi_R$ is usually difficult to evaluate explicitly. However, in the case of the flip $R=F$, one finds $\phi_F(x)=\frac{1}{d}\sum_{k=1}^n S_k^* x S_k$, $x\in\CM$, which restricts to the normalized partial trace on the first tensor factor on $\CN\cong M_d\ot M_d\ot\ldots$, namely 
\begin{align}\label{eq:phiF}
    \phi_F(a_1\ot a_2\ot a_3\ldots\,)=\tau(a_1)\cdot a_2\ot a_3\ot\ldots, \qquad a_i\in M_d.
\end{align}

We summarize these structures in the following proposition in terms of commuting squares of von Neumann algebras \cite{GoodmandelaHarpeJones:1989}.

\begin{proposition}\label{proposition:subfactor-diagram}
    Let $R\in\CR(d)$ and consider the diagram
    \begin{equation}\label{diag:commsquares}
		\begin{array}{ccc}
            \la_R(\CM) &\subset & \CM 
            \\
            \cup && \cup
            \\
            \la_R(\CN) &\subset & \CN 
            \\
            \cup && \cup
            \\
            \varphi(\CL_R) &\subset & \CL_R .
		\end{array}
    \end{equation}
    \begin{enumerate}
     \item\label{item:factors} All von Neumann algebras in the diagram are hyperfinite factors. $\CM$, $\la_R(\CM)$ are of type III$_{1/d}$ and $\CN,\la_R(\CN)$ are of type II$_1$. If $R$ is non-trivial, $\CL_R,\varphi(\CL_R)$ are of type II$_1$ as well.
     \item\label{item:expectations} Both squares in the diagram are commuting squares.
    \end{enumerate}
\end{proposition}
\begin{proof}
    \ref{item:factors} All we need to show is that $\CL_R$ is a factor. So let $x\in\CL_R\cap\CL_R'$. Then $x$ commutes with $R_n\in\CL_R$ for all $n\in\Nl$, and we have $\la_R(x)=\lim_n(\ad R_n)(x)=x$. But since $\la_R$ restricts to $\varphi$ on $\CL_R$, we get $\varphi(x)=\la_R(x)=x$. The canonical endomorphism~$\varphi$ is well known to have only trivial fixed points, hence $x\in\Cl1$.

    \ref{item:expectations} By Takesaki's theorem, the conditional expectation $E_R:\CM\to\la_R(\CM)$ commutes with the modular group. This implies that $E_R(\CN)\subset\CN\cap\la_R(\CM)=\la_R(\CN)$, i.e. the upper square in the diagram is a commuting square.
        
    Recall that for $x\in\CN$, we have $\phi_R(x)=\text{w-lim}_n(\ad {R_n}^*)(x)$. As $R_n\in\CL_R$, this directly gives invariance of $\CL_R$ under $\phi_R$, and therefore $E_R(\CL_R)\subset\la_R(\CL_R)=\varphi(\CL_R)$. This shows that the lower square in \eqref{diag:commsquares} is a commuting square.
\end{proof}

\begin{remark}
	As just demonstrated, any R-matrix provides us with (at least) three subfactors. Let us point out that the $\CM$- and $\CN$-subfactors contain only partial information about $R$ as an R-matrix. For example, let $R=F$ be the flip, $u\in\CU(\CF_d^1)$ non-trivial, and $\alpha:=\la_u\in\Aut\CM$. Then $\la_R\alpha=\la_S$ with $S=\varphi(u)F$. Diagonalizing $u$, it is easy to see that $S$ is a diagonal R-matrix (cf. Def.~\ref{def:diagonalR}~\ref{item:diagonalR}). Moreover, $\la_R$ and $\alpha$ commute, and therefore $\la_R^n(\CM)=\la_S^n(\CM)$, $\la_R^n(\CN)=\la_S^n(\CN)$ for all $n\in\Nl$. But despite $R$ and $S$ defining identical $\CM$- and $\CN$-subfactors, they are quite different from each other as R-matrices, for instance $R^2=1$ and $S^2\neq1$. 

	On the other hand, the subfactors generated by the braid group representations, $\varphi(\CL_R)\subset\CL_R$ and $\varphi(\CL_S)\subset\CL_S$, differ in this example. For instance, we will see later that the first one is irreducible but the second one is not.
\end{remark}

It is a natural question to ask what the indices of the subfactors in \eqref{diag:commsquares} are. Adopting standard notation, we will write $\Ind_{E_R}(\la_R)$ for the index of $\la_R(\CM)\subset\CM$ taken w.r.t. the $\om$-invariant conditional expectation, $\Ind(\la_R)$ for the minimal index of $\la_R(\CM)\subset\CM$ \cite{Kosaki:1986,Hiai:1988,Longo:1989}, and $[\CN:\la_R(\CN)]$, $[\CL_R:\varphi(\CL_R)]$ for the Jones indices \cite{Jones:1983_2} of the type II$_1$ subfactors $\la_R(\CN)\subset\CN$, $\varphi(\CL_R)\subset\CL_R$, respectively.

Independently of the Yang-Baxter equation, it is known that $\Ind_{E_R}(\la_R)=[\CN:\la_R(\CN)]\leq d^2$ \cite{Longo:1989,ContiPinzari:1996}, and the preceding commuting squares result implies $ [\CL_R:\varphi(\CL_R)]\leq[\CN:\la_R(\CN)]$ by a Pimsner-Popa inequality \cite{PimsnerPopa:1986}. We thus have
\begin{align}\label{eq:basic-index-bounds}
	[\CL_R:\varphi(\CL_R)]
	\leq 
	[\CN:\la_R(\CN)] 
	= 
	\Ind_{E_R}(\CM) 
	\leq 
	d^2
	<\infty.
\end{align}
New results on indices will be presented in Section~\ref{section:irreducibility}.

\medskip

We close this section by presenting a large family of R-matrices that can be built with the flip and partitions of unity. 

\begin{definitionlemma}\label{def:diagonalR}\leavevmode
    \begin{enumerate}
     \item\label{item:simpleR} Let $\{p_i\}_{i=1}^N$ be a partition of unity in $\CF_d^1$, i.e. the $p_i$ are orthogonal projections in $\CF_d^1$ such that $p_ip_j=\delta_{ij}p_i$ and $\sum_{i=1}^Np_i=1$. Let $c_{ij}\in\T$, $i,j\in\{1,\ldots,N\}$, be arbitrary parameters. Then
    \begin{align}\label{eq:simpleR}
        R
        :=
        \sum_{i=1}^N c_{ii}\,p_i\varphi(p_i)+\sum_{
        \genfrac{}{}{0pt}{2}{i,j=1}{i\neq j}
        }^N c_{ij}p_i\varphi(p_j)F
    \end{align}
    is an R-matrix.  Such R-matrices will be referred to as {\em simple}.
    
    \item\label{item:diagonalR} If $R\in\CR(d)$ is a simple R-matrix with only one-dimensional projections, i.e. $\tau(p_i)=1/d$ for all $i$, then there exists a unitary $u\in\CU(\CF_d^1)$ such that $p_i=uS_iS_i^*u^*$, and 
    \begin{align}\label{eq:diagonalR}
        R=\la_u(DF),\qquad D=\sum_{i,j=1}^dc_{ij}S_iS_jS_j^*S_i^*.
    \end{align}
    Such R-matrices will be referred to as {\em diagonal}.
    \end{enumerate}
\end{definitionlemma}
\begin{proof}
    \ref{item:simpleR} It is clear that \eqref{eq:simpleR} defines a unitary in $\CF_d^2$. The verification of the Yang-Baxter equation \eqref{eq:YBE-Cuntz-Form} is a tedious but straightforward calculation that we omit here. In Section~\ref{section:sums} we will see a more conceptual argument for $R\in\CR(d)$.
    
    \ref{item:diagonalR} The statement about the existence of $u$ such that $p_i=uS_iS_i^*u^*=\la_u(S_iS_i^*)$ is clear, and we then also have $\varphi(p_i)=\la_u(\varphi(S_iS_i^*))$. We have to verify that \eqref{eq:simpleR} simplifies to \eqref{eq:diagonalR} if all $p_i$ are one-dimensional. To this end, note that $\la_u(F)=F$ and $S_iS_i^*\varphi(S_iS_i^*)F=S_iS_i^*\varphi(S_iS_i^*)$. We get   
    \begin{align*}
        R
        &=
        \sum_i c_{ii}\la_u(S_iS_i^*\varphi(S_iS_i^*)F)
        +
        \sum_{i\neq j}c_{ij}\la_u(S_iS_i^*\varphi(S_jS_j^*)F)
        \\
        &=
        \sum_{i,j}c_{ij}\la_u(S_iS_i^*\varphi(S_jS_j^*)F)
        =
        \la_u(DF)
        ,
    \end{align*}
    as claimed.
\end{proof}

We will frequently use simple R-matrices as examples. Note that trivial R-matrices are simple (choose $N=1$, $p_1=1$) and the flip is diagonal (choose $N=d$, $p_i=S_iS_i^*$, $c_{ij}=1$ for all $i,j$). The term ``simple'' should not be understood in a mathematical sense -- in fact, all non-trivial simple R-matrices define reducible endomorphisms and can be decomposed into smaller R-matrices, as we shall explain later. There exist (more interesting) R-matrices that are not simple.

\section{Towers of algebras associated with R-matrices}\label{section:towers}

Having established the basic subfactors associated with R-matrices, we now turn to their analysis, in particular of their relative commutants. As the basis of our following arguments, we recall some known facts about relative commutants of localized endomorphisms (i.e., endomorphisms of the form $\la_u$, $u\in\CF_d$) of Cuntz algebras. 

For any two endomorphisms $\la,\mu$ of $\CM$, we write
\begin{align*}
 (\la,\mu):=\{T\in\CM\,:\,T\la(x)=\mu(x)T\,\quad\forall x\in\CM\}
\end{align*}
for the space of intertwiners from $\la$ to $\mu$. In particular, $(\la,\la)=\la(\CM)'\cap\CM$ is the relative commutant of $\la(\CM)\subset\CM$.

For an arbitrary unitary $u\in\CU(\CO_d)$, one has \cite[Prop.~2.5]{Longo:1994_2}
\begin{align}\label{eq:SelfIntertwinersGeneral}
 (\la_u,\la_u)=\{x\in\CM\,:\,\varphi(x)=u^*xu\}
 =
 \CM^{\ad u\circ\varphi}.
\end{align}
If, more specifically, $u\in\CU(\CF_d^n)$ for some $n\in\Nl$, one furthermore has \cite[Prop.~4.2]{ContiPinzari:1996}
\begin{align}\label{eq:IntertwinersDecompositionInSpectralSubspaces}
 (\la_u,\la_u) &= 
 \bigoplus_{k=-n+2}^{n-2}(\la_u,\la_u)^{(k)},
 \\
 (\la_u,\la_u)^{(k)}
 &\subset
 \begin{cases}
  (\varphi^{n-1},\varphi^{n-1+k}) & k\geq0
  \\
  (\varphi^{n-1-k},\varphi^{n-1}) & k<0
 \end{cases}
 .
\end{align}
From this we see in particular
\begin{align}\label{eq:IntertwinersInFdn}
 (\la_u,\la_u)^{(0)}
 &=
 \la_u(\CM)'\cap\CN
 \subset(\varphi^{n-1},\varphi^{n-1})=\CF_d^{n-1},
 \\
 (\la_u,\la_u)
 &\subset
 \CF_d^1,\qquad u\in\CU(\CF_d^2),
 \label{eq:IntertwinersInFd2}
\end{align}
and note that \eqref{eq:IntertwinersInFd2} occurs in particular for R-matrices $u=R\in\CU(\CF_d^2)$.

Having recalled these facts, we now turn to study the subfactors given by $\la_R$ and introduce their relative commutants, $n\in\Nl_0$,
\begin{align*}
 \CM_{R,n}
 &:=
 \la_R^n(\CM)'\cap\CM,
 \qquad 
 \CN_{R,n}
 :=
 \la_R^n(\CN)'\cap\CN,
 \qquad
 \CL_{R,n}
 :=
 \varphi^n(\CL_R)'\cap\CL_R.
\end{align*}
Thus $\CM_{R,n}=(\la_R^n,\la_R^n)$, but we prefer the notation $\CM_{R,n}$ in order to distinguish the three different levels of relative commutants $\CM_{R,n}$, $\CN_{R,n}$, $\CL_{R,n}$.

We clearly have three ascending towers of algebras:
\begin{equation} \label{eq:thethreetowers}
    \begin{array}{ccccccccccccccc}
        \Cl =
        & \CM_{R,0}
        &\subset
        & \CM_{R,1}
        &\subset
        & \!\!...\!\!
        &\subset
        & \CM_{R,n}
        &\subset
        & \CM_{R,n+1}
        &\subset
        & \!\!...\!\!
        &\subset
        & \CM
        \\
         \Cl =
        & \CN_{R,0}
        &\subset
        & \CN_{R,1}
        &\subset
        & \!\!...\!\!
        &\subset
        & \CN_{R,n}
        &\subset
        & \CN_{R,n+1}
        &\subset
        & \!\!...\!\!
        &\subset
        & \CN
        \\
        \Cl =
        & \CL_{R,0}
        &\subset
        & \CL_{R,1}
        &\subset
        & \!\!...\!\!
        &\subset
        & \CL_{R,n}
        &\subset
        & \CL_{R,n+1}
        &\subset
        & \!\!...\!\!
        &\subset
        & \CL_R
    \end{array}
\end{equation}

In the following, we will derive various relations/inclusions between these algebras, and realise them as fixed point algebras for certain endomorphisms. In particular, it is not clear from the outset if there are inclusions one way or the other between $\CM_{R,n}$, $\CN_{R,n}$, $\CL_{R,n}$. 

We begin with the relative commutants at the highest level, i.e. the $\CM_{R,n}$.

\begin{proposition}\label{proposition:MRn}
 Let $R\in\CR(d)$ and $n\in\Nl$. Then 
 \begin{align}
  \CM_{R,n}
  &=
  \CM^{\ad {}_nR\circ\varphi}
  =
  \bigoplus_{k=-n+1}^{n-1}
  (\CM^{(k)})^{\ad {}_nR\circ\varphi},
  \\
  \CM_{R,n}^{(0)}
  &=
  \la_R^n(\CM)'\cap\CN
  =
  (\CF_d^n)^{\ad{}_nR\circ\varphi}
  =
  \{x\in\CF_d^n\,:\,\varphi(x)=\la_{R^*}(x)\},
 \end{align}
 and in particular for $n=1$,
 \begin{align}\label{eq:MR1}
 \CM_{R,1}
 =
 \CM_{R,1}^{(0)}
 =
 \{x\in\CF_d^1\,:\,\varphi(x)=R^*xR\}.
\end{align}
\end{proposition}
\begin{proof}
 Recall that $\la_R^n=\la_{{}_nR}$ \eqref{eq:PowersOfLaR} and ${}_nR=\varphi^{n-1}(R)\cdots\varphi(R)R\in\CF_d^{n+1}$. Then the two equalities in the first line immediately follow from \eqref{eq:SelfIntertwinersGeneral} and \eqref{eq:IntertwinersDecompositionInSpectralSubspaces}.
 
 In the second line, the first equality is the definition of $\CM_{R,n}^{(0)}$ and the second equality follows by combining \eqref{eq:SelfIntertwinersGeneral} with \eqref{eq:IntertwinersInFdn} and ${}_nR\in\CF_d^{n+1}$. To get the last equality, note that for $x\in\CF_d^n$, 
 \begin{align*}
  \la_{R^*}(x)
  =\ad (R^*)_n (x)
  =\ad ({}_nR)^* (x),
 \end{align*}
and therefore $x\in(\CF_d^n)^{\ad{}_nR\circ\varphi}$ is equivalent to $x\in\CF_d^n$ with $\varphi(x)=\ad({}_nR)^*(x)=\la_{R^*}(x)$. The special case $n=1$ now follows from the previous statements.
\end{proof}

As an example, we determine the structure of $\CM_{R,1}$ for a class of simple R-matrices.

\begin{proposition}\label{prop:structureofMR1}
    Let $R$ be a simple R-matrix (Def.~\ref{def:diagonalR}~\ref{item:simpleR}) with projections $p_1,\ldots,p_N\in\CF_d^1$ and parameters $c_{ij}$, $i,j\in\{1,\ldots,N\}$, such that $c_{ij}=1$ for $i\neq j$. Define 
    \begin{align*}
        m
        &:=
        \left|\{i\in\{1,\ldots,N\}\,:\,\tau(p_i)=1/d,\;c_{ii}=1\}\right|.
    \end{align*}
    Then
    \begin{align}\label{eq:structureofMR1fornormalforms}
     \CM_{R,1}
     \cong 
     \underbrace{\C\oplus\ldots\oplus\C}_{N-m\text{ \rm terms}} \oplus M_m.
    \end{align}
\end{proposition}
\begin{proof}
    Let $x\in\CF_d^1$. We claim that $x\in\CM_{R,1}$ is equivalent to $x$ satisfying the following two conditions:
    \begin{enumerate}
        \item $p_ixp_i\in\Cl p_i$ for all $i$.
        \item Let $i\neq j$. If $\tau(p_i)>d^{-1}$ or $\tau(p_j)>d^{-1}$ or $c_{ii}\neq1$ or $c_{jj}\neq1$, then $p_ixp_j=0$.
    \end{enumerate}
    To verify this, we first calculate from the definition of $R$ \eqref{eq:simpleR}
    \begin{align*}
     y
     &:=(\ad R\circ\varphi)(x)-x
     \\
     &=
     \sum_i p_i\varphi(p_ixp_i)
     +
     \sum_{i\neq j} c_{ii} p_i\varphi(p_ixp_j)F
     +
     \sum_{i\neq j} \overline{c_{ii}} p_jxp_i\varphi(p_i)F
     \\
     &\quad
     -
     \sum_i p_ixp_i\varphi(p_i)
     -
     \sum_{i\neq j} p_jxp_i\varphi(p_i)
     -
     \sum_{i\neq j} p_ixp_j\varphi(p_i)
     .
    \end{align*}
    Vanishing of $y$ is equivalent to $x\in\CM_{R,1}$. We observe that if $y=0$, then for any $i$
    \begin{align*}
        0=
        p_i\varphi(p_i)yp_i\varphi(p_i)
        =
        p_i\varphi(p_ixp_i)-p_ixp_i\varphi(p_i).
    \end{align*}
    Thus $x\in\CM_{R,1}$ implies condition a).
    
    We next consider $i\neq j$. If $y=0$, then
    \begin{align*}
        0
        &=
        p_i\varphi(p_i)yp_j\varphi(p_i)
        =
        c_{ii}\,p_i\varphi(p_ixp_j)F-p_ixp_j\varphi(p_i),
        \\
        0
        &=
        p_j\varphi(p_i)yp_i\varphi(p_i)
        =
        \overline{c_{ii}}\,p_jxp_i\varphi(p_i)F-p_jxp_i\varphi(p_i).
    \end{align*}
    It follows that if either $p_i$ or $p_j$ has dimension greater than $1$, then $p_ixp_j=0$. Furthermore, if $p_i$ and $p_j$ are one-dimensional (i.e. $\tau(p_i)=\tau(p_j)=1/d$) and $c_{ii}\neq1$ or $c_{jj}\neq1$, then $p_ixp_j=0$. That is, $x\in\CM_{R,1}$ implies condition b). 
    
    Conversely, if a) and b) hold, it is easy to check that the above sum vanishes (term 1 cancels term 4, term 2 cancels term 6, and term 3 cancels term 5). Thus $x\in\CM_{R,1}$ is equivalent to $x$ satisfying a) and b).
    
    With $I:=\{i\in\{1,\ldots,N\}\,:\,\tau(p_i)=1/d,\;c_{ii}=1\}$ and $p:=\sum_{i\in I}p_i$, we then have $x\in\CM_{R,1}$ if and only if $x$ is of the form 
    \begin{align*}
        x=\sum_{i\not\in I} \alpha_i\,p_i + pxp,\qquad \alpha_i\in\Cl.
    \end{align*}
    As $|I|=m$ and $|\{1,\ldots,N\}\backslash I|=N-m$ this gives the claimed result.
\end{proof}

We see from this result that the R-matrices considered are all reducible in the sense that $\CM_{R,1}\neq\Cl$; unless $R\in\Cl$ ($N=1$). 

\medskip

In \cite[Prop.~2.3]{ContiRordamSzymanski:2010}, it was shown that for a unitary $u\in\CU(\CO_d)$, one has $\la_u(\CF_d)'\cap\CO_d=\bigcap_{n\in\Nl}(\ad u\circ\varphi)^n(\CO_d)$. We now present a variation of this argument which is also stated in \cite{Akemann:1997} to characterize the relative commutants $\CN_{R,n}$. As \cite{Akemann:1997} is not published, we give most details of the proof.

\begin{proposition}
 Let $R\in\CR(d)$ and $n\in\Nl$. Then 
 \begin{align}\label{eq:NRn}
  \CN_{R,n}
  =
  \bigcap_{k\geq0}(\ad {}_nR\circ\varphi)^k(\CF_d^n)
 \end{align}
 is the largest subalgebra of $\CF_d^n$ that is globally stable under $\ad{}_nR\circ\varphi$. In particular, 
 \begin{align}
  \CM_{R,n}^{(0)}
  &\subset\CN_{R,n},\quad n\in\Nl,
  \qquad
  \CM_{R,1}
  \subset\CN_{R,1}
  \subset\CF_d^1
  .
  \label{eq:MR1inNR1}
 \end{align}
\end{proposition}
\begin{proof}
The ${}^*$-algebra $\bigcup_{k\in\Nl}\CF_d^k$ is weakly dense in $\CN$. This implies that an element $x\in\CN$ commutes with $\la_{{}_nR}(\CN)$ if and only if
 \begin{alignat*}{5}
  &&0&=[x,\la_{{}_nR}(y)]
  =[x,({}_nR)_ky({}_nR)_k^*] 
  &\quad\forall k\in\Nl,y\in\CF_d^k
  \\
  &\Longleftrightarrow&\;
  0&=[({}_nR)_k^*x({}_nR)_k,y]
  &\quad\forall k\in\Nl,y\in\CF_d^k
  \\
  &\Longleftrightarrow&\;
  ({}_nR)_k^*x({}_nR)_k&\in\CN\cap(\CF_d^k)'=\varphi^k(\CF_d)
  &\quad\forall k\in\Nl\\
  &\Longleftrightarrow&\;
  x&\in(\ad ({}_nR)_k\circ\varphi^k)(\CN)=(\ad{}_nR\circ \varphi)^k(\CN)
  &\quad\forall k\in\Nl.
 \end{alignat*}
 This proves $\CN_{R,n}=\bigcap_{k\in\Nl}(\ad {}_nR\circ\varphi)^k(\CN)$. We next show $\CN_{R,n}\subset\CF_d^n$, following \cite[Thm.~3.16]{Akemann:1997}. Namely, we consider the isometry $T$ on $L^2(\CN)$ that is defined by continuous extension of $\CN\ni x\mapsto {}_nR\varphi(x)({}_nR)^*\in\CN$. Then $T$ restricts to a unitary on $\CK:=\bigcap_{k\geq1}T^kL^2(\CN)\supset\CN_{R,n}$ and we have to show $\CK\subset\CF_d^n$. This follows by taking into account that for $x\in\CF_d^m$, $m\in\Nl$, one has $T^{*k}x=(\phi_F\circ\ad({}_nR)^*)^k(x)\in\CF_d^n$ for $k\geq m$, and the finite dimensionality of $\CF_d^n$ \cite[Lemma~3.15]{Akemann:1997}.
 
 Having established $\CN_{R,n}\subset\CF_d^n$, we get together with the previous result
 \begin{align*}
  \CN_{R,n}
  &=\bigcap_{k\in\Nl}(\ad {}_nR\circ\varphi)^k(\CN)\cap\CF_d^n
  =\bigcap_{k\in\Nl}(\ad ({}_nR)_k\circ\varphi^k)(\CN)\cap\CF_d^n.
 \end{align*}
 But inserting the definitions, one sees $({}_nR)_k\in\CF_d^{n+k}$. Thus $({}_nR)_k\varphi^k(x)({}_nR)_k^*=y$ for some $y\in\CF_d^n$ and $x\in\CN$ implies $x\in\CF_d^n$. That is, $(\ad ({}_nR)_k\circ\varphi^k)(\CN)\cap\CF_d^n=(\ad ({}_nR)_k\circ\varphi^k)(\CF_d^n)$, and we arrive at the claimed formula \eqref{eq:NRn}.
 
 To also get the characterization of $\CN_{R,n}$ as the largest subalgebra of $\CF_d^n$ globally invariant under $T=\ad{}_nR\circ\varphi$, it remains to show $T(\CN_{R,n})=\CN_{R,n}$. Note that since $\CF_d^n$ is finite-dimensional, the sequence $\bigcap_{k=0}^mT^k(\CF_d^n)$ stabilizes, i.e. there exists $m_0\in\Nl$ such that $\CN_{R,n}=\bigcap_{k=0}^{m}T^k(\CF_d^n)$ for all $m\geq m_0$. Thus
 \begin{align*}
  T(\CN_{R,n})
  =
  \bigcap_{k=0}^{m_0}T^{k+1}(\CF_d^n)
  \supset
  \bigcap_{k=0}^{m_0+1}T^{k}(\CF_d^n)
  =
  \bigcap_{k=0}^{m_0}T^{k}(\CF_d^n)
  =
  \CN_{R,n},
 \end{align*}
 i.e. the finite-dimensional space $\CN_{R,n}$ is contained in its image under $T$. This implies $T(\CN_{R,n})=\CN_{R,n}$
 
 From this characterization of $\CN_{R,n}$, it is now obvious that it contains the fixed points $(\CF_d^n)^T=\CM_{R,n}^{(0)}$. In the special case $n=1$, we have $\CM_{R,1}=\CM_{R,1}^{(0)}$ by \eqref{eq:MR1}.
\end{proof}

\begin{remark}\label{remark:N1andM1CanDiffer}
    Let us give an example showing that in general, $\CM_{R,1}\neq\CN_{R,1}$. For later use, we actually give two similar examples, both based on the flip $F$ and a unitary $u\in\CF_d^1$, namely
    \begin{align*}
     R:=uF,\qquad S:=uFu^*=u\varphi(u^*)F.
    \end{align*}
    Both $R$ and $S$ are R-matrices, as can be checked by direct verification of the Yang-Baxter equation, or by realizing that they are diagonal (Def.~\ref{def:diagonalR}). For $x\in\CF_d^1$, we have
    \begin{align*}
        (\ad R\circ\varphi)(x)&=RFxFR^*=uxu^*,\\
        (\ad S\circ\varphi)(x)&=SFxFS^*=u\varphi(u^*)x\varphi(u)u^*=uxu^*.
    \end{align*}
    Thus $\CF_d^1$ is globally invariant under $\ad R\circ\varphi$ and $\ad S\circ\varphi$, and therefore $\CN_{R,1}=\CN_{S,1}=\CF_d^1$. But for $u\not\in\Cl$, the above formula shows that not every $x\in\CF_d^1$ is a fixed point of $\ad R\circ\varphi$ or $\ad S\circ\varphi$ i.e. $\CM_{R,1}=\CM_{S,1}$ is a proper subalgebra of~$\CF_d^1$.
\end{remark}

We now move on to the relative commutants $\CL_{R,n}$ on the level of the von Neumann algebra $\CL_R$ generated by the $B_\infty$-representation $\rho_R$. In this representation, $R$ represents the first generator $b_1\in B_\infty$; in particular, $\CL_R=\CL_{R^*}$. 

The following proposition contains in particular the fixed point characterization $\CL_{R,n}=\CL_R^{\la_{\varphi^n(R)}}$ which is similar to the work of Gohm and Köstler \cite{GohmKostler:2009_2}, where analogues of $\la_{\varphi^n(R)}$ are called ``partial shifts''.

\begin{proposition}\label{prop:LRn}
    Let $R\in\CR(d)$ and $n\in\Nl_0$. Then 
    \begin{enumerate}
        \item\label{item:characterization-of-LRn} $\CL_{R,n}=\CF_d^n\cap\CL_R=\CL_R^{\la_{\varphi^n(R)}}=\CM_{R,n}\cap\CL_R$, and all these algebras are invariant under exchanging $R$ and $R^*$.
        \item\label{item:Bn-in-LRn} $C^*(\rho_R(B_n))\subset\CL_{R,n}$, $n\geq1$.
    \end{enumerate}
\end{proposition}
\begin{proof}
    \ref{item:characterization-of-LRn} We will demonstrate the inclusions
    \begin{align*}
        \CL_{R,n}
        \stackrel{\text{(i)}}{\subset}
        \CL_R^{\la_{\varphi^n(R)}}
        \stackrel{\text{(ii)}}{\subset}
        \CM_{R^*,n}\cap\CL_R
        \stackrel{\text{(iii)}}{\subset}
        \CF_d^n\cap\CL_R
        \stackrel{\text{(iv)}}{\subset}
        \CL_{R,n}
        .
    \end{align*}
    Note the appearance of ${R^*}$ instead of $R$ in the third algebra. Nonetheless, this chain of inclusions implies the claimed equalities because we have $\CL_R=\CL_{R^*}$ and may thus run through the chain of inclusions once more with $R$ and $R^*$ interchanged, realizing that all algebras are invariant under replacing $R$ with~$R^*$.
    
    To begin with, we note that $\la_{\varphi^n(R)}(x)$, $x\in\CN$, can be written as
    \begin{align*}
        \la_{\varphi^n(R)}(x)
        &=
        \lim_{k\to\infty}\varphi^n(R)\cdots\varphi^{n+k}(R)x\varphi^{n+k}(R^*)\cdots\varphi^n(R^*)
        \\
        &=
        \varphi^{n-1}(R^*)\cdots R^*\la_R(x)R\cdots\varphi^{n-1}(R)
        \\
        &=
        {}_n(R^*)\la_R(x){}_n(R^*)^*.        
    \end{align*}
    It is apparent from the first line that any $x\in\varphi^n(\CL_R)'$ is a fixed point of $\la_{\varphi^n(R)}$, i.e. we have inclusion (i). 

    Any $x\in\CL_R$ satisfies $\la_R(x)=\varphi(x)$, and thus the above calculation yields
    \begin{align*}
        \CL_R^{\la_{\varphi^n(R)}}
        &\subset
        \{x\in\CL_R\,:\,x=(\ad {}_n(R^*)\circ\varphi)(x)\}
        =
        \CM_{R^*,n}\cap\CL_R,
    \end{align*}
	where we have used Prop.~\ref{proposition:MRn}. This shows the inclusion (ii).
	
	As $\CL_R\subset\CN$, we also have $\CM_{R^*,n}\cap\CL_R=\CM_{R^*,n}^{(0)}\cap\CL_R\subset\CF_d^n\cap\CL_R$ by Prop.~\ref{proposition:MRn}, showing inclusion~(iii). Inclusion (iv) is evident because $\CF_d^n$ and $\varphi^n(\CL_R)$ commute in $\CN$.

	\ref{item:Bn-in-LRn} By definition of $\rho_R$, we have $C^*(\rho_R(B_n))\subset\CF_d^n\cap\CL_R=\CL_{R,n}$.
\end{proof}

Having clarified some of the relations of the relative commutants, in particular 
\begin{align}\label{eq:inclusions-of-relative-commutants}
    \CL_{R,n}\subset\CM_{R,n}^{(0)}\subset\CN_{R,n}\subset\CF_d^n,
    \qquad n\in\Nl,
\end{align}
we comment on the action of $\la_R$ and $\phi_R$ on these three towers.

\begin{lemma}
	Let $R\in\CR(d)$ and $n\in\Nl_0$. Then 
	\begin{align}\label{eq:laRonTowers}
		\la_R(\CM_{R,n})\subset\CM_{R,n+1}
		,\quad
		\la_R(\CN_{R,n})\subset\CN_{R,n+1}
		,\quad
		\la_R(\CL_{R,n})\subset\CL_{R,n+1}
		,\\
		\phi_R(\CM_{R,n+1})=\CM_{R,n}
		,\quad
		\phi_R(\CN_{R,n+1})=\CN_{R,n}
		,\quad
		\phi_R(\CL_{R,n+1})=\CL_{R,n}
		,
		\label{eq:phiRonTowers}
	\end{align}
	and
	\begin{align}\label{eq:Rinsecondcommutant}
		R&\in\CL_{R,2}\subset\CM_{R,2}^{(0)}
		\subset\CM_{R,2}\cap\CN_{R,2},\\
		\phi_R(R)&\in\CL_{R,1}\subset\CM_{R,1}\subset\CN_{R,1}.
		\label{eq:phiRRinfirstcommutant}
	\end{align}
\end{lemma}
\begin{proof}
 	Since $\CM_{R,n}=(\la_R^n,\la_R^n)$, we clearly have $\la_R(\CM_{R,n})\subset\CM_{R,n+1}$, and since $\la_R$ preserves the subalgebras $\CN$ and $\CL_R$ of $\CM$, we also have the other inclusions in \eqref{eq:laRonTowers}. Applying the left inverse $\phi_R$ then gives $\CM_{R,n}\subset\phi_R(\CM_{R,n+1})$, etc. Taking into account that $\phi_R$ preserves $\CN$ and $\CL_R$ \eqref{eq:phiRpreserves}, and its bimodule property \eqref{eq:phiRproperty}, we even get the equalities \eqref{eq:phiRonTowers}.

	By Prop.~\ref{prop:ybe-od}~\ref{item:RinL2L2}, we have $R\in\CM_{R,2}$. Since $R\in\CL_R$, \eqref{eq:Rinsecondcommutant} follows. The second equality \eqref{eq:phiRRinfirstcommutant} is then a consequence of $\phi_R(\CL_{R,2})\subset\CL_{R,1}$.
\end{proof}

We can now conclude that the inclusion $C^*(\rho_R(B_n))\subset\CL_{R,n}$ in Prop.~\ref{prop:LRn}~\ref{item:Bn-in-LRn} is proper in general. For example, for $n=1$ the group $B_n$ is trivial, i.e. $C^*(\rho_R(B_1))=\Cl$, but $\CL_{R,1}$ contains $\phi_R(R)$ which is non-trivial in general.

\medskip

As an aside, we mention that it frequently happens that the braid group representations ${\rho_R}|_{B_n}$ factor through a finite quotient of $B_n$ \cite{GalindoRowell:2013}. The most prominent example of such a quotient is the symmetric group $S_n$; note that $\rho_R$ factors through the symmetric groups if and only if $R$ is {\em involutive}, i.e. $R^2=1$. We record the following characterization of R-matrices with trivial square.

\begin{lemma}\leavevmode
	\begin{enumerate}
        \item\label{item:trivialsquareendo} Let $u\in\CU(\CO_d)$. Then $u^2\in\C$ if and only if $u \in (\lambda_u,\lambda_{\varphi(u)})$.
        \item\label{item:trivialsquareR} 
        Let $R \in \CR(d)$. Then the following conditions are equivalent:
        \begin{itemize}
            \item $R^2 \in\Cl$;
            \item $R \in (\lambda_R,\lambda_{\varphi(R)})$;
            \item $\lambda_R$ and $\lambda_{R^*}$ commute.
        \end{itemize}
	 \end{enumerate}
\end{lemma}
\begin{proof}
	\ref{item:trivialsquareendo} One has ${\rm ad}(u) \lambda_u = \lambda_{u^2 \varphi(u)^*}$, so that it coincides with $\lambda_{\varphi(u)}$ if and only if $u^2 \varphi(u)^* = \varphi(u)$,
	i.e. $u^2 = \varphi(u^2)$. However, it is well known that $\varphi$ admits no nontrivial fixed points.
	
	\ref{item:trivialsquareR} We have $\la_R\circ\la_{R^*}=\la_{\varphi(R^*)R}$ and $\la_{R^*}\circ\la_{R}=\la_{\varphi(R)R^*}$, hence the two endomorphisms commute if and only if $R^2=\varphi(R^2)$. Since $\varphi$ has only trivial fixed points, the conclusion follows.
\end{proof}

Our main results concerning the relative positions of the subalgebras $\varphi^n(\CL_R)$ and $\CL_{R,n}$ in $\CN$ are contained in the following theorem. The $\tau$-preserving conditional expectation $\CN\to\CF_d^n$ will be denoted $E_n$.

\begin{theorem}\label{thm:commuting-squares}
    Let $R\in\CR$ and $n\in\Nl$. Then the squares 
    \begin{align}\label{eq:commutingsquare2}
        \begin{array}{ccc}
            \CF_d^n &\subset &\CN
            \\
            \cup & & \cup
            \\
            \CL_{R,n} &\subset &\CL_R
        \end{array}
        \qquad\qquad\qquad
        \begin{array}{ccc}
            \varphi^n(\CN) &\subset &\CN
            \\
            \cup & & \cup
            \\
            \varphi^n(\CL_R) &\subset &\CL_R
        \end{array}
    \end{align}
    commute, i.e. $E_n(\CL_R)=\CL_{R,n}$ and $\phi_R(x)=\phi_F(x)$, $x\in\CL_R$.
\end{theorem}

The proof splits naturally into two parts, one for each diagram. The proof of the first part (left diagram) is given below. The proof of the second part (right diagram) requires more work and is best done after more structure has been introduced. It is therefore postponed to Section~\ref{section:equivalence} (p.~\pageref{page:proofpart2}).

\medskip

\noindent{\em Proof (first half).} 
	Let $H_{R,n}$ denote the $\tau$-preserving conditional expectation of $\CN^{\la_{\varphi^n(R)}}\subset\CN$. As $\CL_R\subset\CN$ is invariant under $\la_{\varphi^n(R)}$ by Prop.~\ref{prop:LRn}, the map $H_{R,n}$ restricts to the $\tau$-preserving conditional expectation from $\CL_R$ onto $\CL^{\la_{\varphi^n(R)}}=\CL_{R,n}=\CF^n_d\cap\CL_R$.
    
    Given $x\in\CL_R$, we want to show that $H_{R,n}(x)$ coincides with $E_n(x)$. Indeed, both $H_{R,n}(x)$ and $E_n(x)$ lie in $\CF_d^n$, so we only have to show $\tau(yH_{R,n}(x))=\tau(yE_n(x))$ for all~$y\in\CF_d^n$. But $\CF_d^n$ is clearly contained in the fixed point algebra $\CN^{\la_{\varphi^n(R)}}$. Thus, for $x\in\CL_R$, $y\in\CF_d^n$, 
    \begin{align*}
        \tau(yH_{R,n}(x))
        &=
        \tau(H_{R,n}(yx))
        =
        \tau(yx)
        =
        \tau(E_n(yx))
        =
        \tau(yE_n(x)).
    \end{align*}
    This shows $E_n(x)=H_{R,n}(x)\in\CL_{R,n}$, which is equivalent to the left diagram being a commuting square.$\hfill\square$
    
    \medskip

So far, we have concentrated on the ``horizontal inclusions'' in \eqref{eq:thethreetowers} and not mentioned $\CL_R\subset\CN$, $\CL_R\subset\CM$. These ``vertical inclusions'' are closely connected to fixed points of $\la_R$ and will be discussed in Section~\ref{section:ergodicity}.

\section{Algebraic operations on $\CR$}\label{section:operations}

Although the structure of the set $\CR(d)$ of all R-matrices of dimension $d$ is not known, a number of symmetries of $\CR(d)$ are known. For example, $R\mapsto R^*$, $R\mapsto c\cdot R$, $c\in\T$, $R\mapsto(u\ot u)R(u\ot u)^*$, $u\in\CU(\CF_d^1)$, and $R\mapsto FRF$ with the flip $F\in\CR(d)$, are all bijections\footnote{The maps $R\mapsto(u\ot u)R(u\ot u)^*$ and $R\mapsto FRF$ will be discussed in more detail in Section~\ref{section:equivalence}.} $\CR(d)\to\CR(d)$.

However, it is often more interesting to consider algebraic operations that exist only on $\CR=\bigcup_{d}\CR(d)$ and do not preserve the spaces $\CR(d)$ of R-matrices of fixed dimension~$d$. In this section, we will discuss three such structures: A tensor product $R\boxtimes S$ (with $\dim(R\boxtimes S)=\dim R\cdot\dim S$), Wenzl's cabling powers $R^{(n)}$ (with $\dim(R^{(n)})=(\dim R)^n$), and a sum operation $R\boxplus S$ (with $\dim(R\boxplus S)=\dim R+\dim S$). 

On the level of R-matrices, all these operations are known. What is new in our approach is that we relate them to natural operations on the corresponding Yang-Baxter endomorphisms.

In the following, the dimension $d$ will be explicitly indicated in our notation, i.e.
we write $\CN_d$ for the infinite tensor product of matrix algebras $M_d$, and $\tau_d$, $\varphi_d$ for its canonical trace and shift, $F_d\in\CU(\CF_d^2)$ for the flip in dimension $d$, etc.

\subsection{Tensor products of R-matrices}\label{section:tensorproducts}

Let $R \in \CR(d)\subset\End(\Cl^d\ot\Cl^d)$, $\tilde R \in \CR(\tilde d)\subset\End(\Cl^{\tilde d}\ot\Cl^{\tilde d})$ be $R$-matrices. The tensor product of $R$, $\tilde R$ is defined as
\begin{align}
	R \boxtimes \tilde R 
	:= F_{23} (R \otimes \tilde R) F_{23} \in \End((\Cl^d\ot\Cl^{\tilde d})\ot(\Cl^d\ot\Cl^{\tilde d})),
\end{align}
where $F_{23}: \Cl^d\ot\Cl^{\tilde d}\ot\Cl^d\ot\Cl^{\tilde d}\to\Cl^d\ot\Cl^{d}\ot\Cl^{\tilde d}\ot\Cl^{\tilde d}$ is the flip unitary exchanging the two middle factors. Evidently $R\boxtimes\tilde R$ is a unitary R-matrix of dimension $d\tilde d$, i.e. $R\boxtimes\tilde R\in\CR(d\tilde d)$. We will refer to $R \boxtimes \tilde R$ as the {\it tensor product} of $R$ and $\tilde R$ (although it slightly differs from the actual tensor product $R\ot\tilde R$). It is also clear that $(R \boxtimes \tilde R)^* = R^* \boxtimes \tilde R^*$, and that if both $R$ and $\tilde R$ are involutive, then so is $R \boxtimes \tilde R$. 

From the point of view of the Cuntz algebras, we may consider $R \in \CF_d^2$, $S \in \CF_{\tilde d}^2$ and $R \boxtimes \tilde R \in \CF_{d\tilde d}^2$. The following discussion will allow us to get a precise relation between the associated subfactors.

\medskip

Let $\CO_d$ and $\CO_{\tilde d}$ be Cuntz algebras with canonical generators $S_i$, $1 \leq i \leq d$ and $\tilde S_j$, $1 \leq j \leq \tilde d$, respectively.
Namely, all the $S_i$'s and $\tilde S_j$'s are isometries such that $\sum_{i=1}^d S_i S_i^* = 1$, $\sum_{j=1}^{\tilde d} \tilde S_j \tilde S_j^* = 1$, and $\CO_d=C^*(S_1,\ldots,S_d)$, $\CO_{\tilde d} = C^*(\tilde S_1,\ldots, \tilde S_{\tilde d})$.
The tensor product $C^*$-algebra $\CO_d \otimes \CO_{\tilde d}$ is generated by the elements $S_i \otimes 1$ and $1 \otimes \tilde S_j$, $1 \leq i \leq d$, $1 \leq j \leq \tilde d$.\footnote{Since $\CO_d$ is nuclear there is no ambiguity on the choice of the cross-norm on the algebraic tensor product.}
In general, $\CO_d \otimes \CO_{\tilde d}$ is not a Cuntz algebra.\footnote{However, it is well known that $\CO_2 \otimes \CO_d \simeq \CO_2$, for all $d \geq 2$, 
although none of these isomorphisms has been concretely exhibited.}

Consider also the Cuntz algebra $\CO_{d\tilde d}$, with canonical generating isometries $U_{ij}$, $1 \leq i \leq d, 1 \leq j \leq \tilde d$ such that $\sum_{i,j} U_{ij} U_{ij}^* = 1$.
Since, for every $1 \leq i \leq d$ and $1 \leq j \leq \tilde d$, $S_i \otimes \tilde S_j$ is an isometry in $\CO_d \otimes \CO_{\tilde d}$ and, moreover, 
$\sum_{i,j} S_i \otimes \tilde S_j (S_i \otimes \tilde S_j)^* = \big(\sum_i S_i S_i^*\big) \otimes \big(\sum_j \tilde S_j \tilde S_j^* \big) = 1 \otimes 1$,
there is an injective $*$-homomorphism 
\begin{align}
	\iota_{d,\tilde d} : \CO_{d\tilde d} \to \CO_d \otimes \CO_{\tilde d}
\end{align}
such that $\iota_{d,\tilde d}(U_{ij})=S_i \otimes \tilde S_j$.

In order to simplify the notation, in the sequel we will often drop the symbol $\iota_{d,\tilde d}$ and identify accordingly $U_{ij}$ with $S_i \otimes \tilde S_j$. All in all, we have thus identified a copy of $\CO_{d\tilde d}$ inside $\CO_d \otimes \CO_{\tilde d}$, as the $C^*$-subalgebra of the tensor product generated by the isometries $S_i \otimes \tilde S_j$. Moreover, it is not difficult to see that $\CO_{d\tilde d} = (\CO_d \otimes \CO_{\tilde d})^\beta$, where $\beta$ denotes the $2\pi$-periodic ``twisted'' $\Rl$-action $\beta^t := \alpha_d^t\otimes\alpha_{\tilde d}^{-t}=\lambda_{e^{it} 1_d}\ot\lambda_{e^{-it} 1_{\tilde d}}$\cite{Chambers:2014,Morgan:2017}, and there exists a faithful conditional expectation $\CO_d \otimes \CO_{\tilde d} \to \CO_{d\tilde d}$ obtained by averaging~$\beta$. 

Under the identification of $\CO_{d\tilde d}$ with $(\CO_d \otimes \CO_{\tilde d})^\beta$, there are coherent identifications of $\CF_{d\tilde d}^n$ with $\CF_d^n \otimes \CF_{\tilde d}^n$, $n\in\Nl$, such that
\begin{align*}
	& U_{i_1 j_1}U_{i_2 j_2} \cdots U_{i_n j_n} U^*_{i'_n j'_n} \cdots U^*_{i'_2 j'_2} U^*_{i'_1 j'_1} 
	\\
	&=
	(S_{i_1} \otimes \tilde S_{j_1})(S_{i_2} \otimes \tilde S_{j_2}) \cdots (S_{i_n} \otimes \tilde S_{j_n})(S_{i'_n} \otimes \tilde S_{j'_n})^* \cdots (S_{i'_2} \otimes \tilde S_{j'_2})^*(S_{i'_1} \otimes \tilde S_{j'_1})^* 
	\\
	&= 
	(S_{i_1}S_{i_2} \cdots S_{i_n}S^*_{i'_n} \cdots S^*_{i'_2}S^*_{i'_1}) \otimes 
	(\tilde S_{j_1}\tilde S_{j_2} \cdots \tilde S_{j_n}\tilde S^*_{j'_n} \cdots \tilde S^*_{j'_2}\tilde S^*_{j'_1}), 
\end{align*}
and thus of $\CF_{d\tilde d}=\CO_{d\tilde d}^{\alpha_{d\tilde d}}$ with $\CF_d \otimes \CF_{\tilde d} = \CO_d^{\alpha_d} \otimes \CO_{\tilde d}^{\alpha_{\tilde d}}$.

For the following lemma, the Yang-Baxter equation is not needed.
\begin{lemma}\leavevmode
	\begin{enumerate}
	 \item\label{item:otimes-and-restriction} Let $R\in\CU(\CO_d)$ and $\tilde R\in\CU(\CO_{\tilde d})$. Then $\la_R\ot\la_{\tilde R}\in\End(\CO_d\ot\CO_{\tilde d})$ restricts to an endomorphism of $\CO_{d\tilde d}$ if and only if $R\in\CF_d$ and $\tilde R\in\CF_{\tilde d}$.
	 \item\label{item:boxtimes-and-restriction} Let $R\in\CU(\CF_d^2)$, $\tilde R\in\CU(\CF_{\tilde d}^2)$. Then $\iota_{d,\tilde d}(R\boxtimes\tilde R)=R\otimes\tilde R$, and 
	\begin{align}
		(\la_R\otimes\la_{\tilde R})|_{\CO_{d\tilde d}}
		=
		\la_{R\boxtimes\tilde R}
		.
	\end{align}
	\end{enumerate}
\end{lemma}
\begin{proof}
	\ref{item:otimes-and-restriction} On generators, the endomorphism $\la_R\ot\la_{\tilde R}\in\End(\CO_d\ot\CO_{\tilde d})$ acts according to $(\lambda_R \otimes \lambda_{\tilde R})(S_i \otimes \tilde S_j)= (R \otimes \tilde R) (S_i \otimes \tilde S_j)$ for all $i$, $j$. Thus $\lambda_R \otimes \lambda_{\tilde R}$ restricts to~$\CO_{d\tilde d}$, that is $(\lambda_R \otimes \lambda_{\tilde R})(\CO_{d\tilde d}) \subset \CO_{d\tilde d}$, precisely when $R \otimes \tilde R \in \CO_{d\tilde d}$, i.e. precisely when $\alpha_d^t(R) \otimes \alpha_{\tilde d}^{-t}(\tilde R) = R \otimes \tilde R$ for all $t \in \Rl$. This latter condition is satisfied if and only if both $R$ and $\tilde R$ are eigenvectors for $\alpha_d$ and $\alpha_{\tilde d}$, respectively, i.e. $R\in\CO_d^{(n)}$, ${\tilde R}\in\CO_{\tilde d}^{(n)}$ for some $n\in\Zl$. But this is easily seen to be in conflict with the KMS condition for $\om$ if $n\neq0$. Thus $R\in\CO_d^{(0)}=\CF_d$, $\tilde R\in\CO_{\tilde d}^{(0)}=\CF_{\tilde d}$.
	
	\ref{item:boxtimes-and-restriction} Note that the matrix elements of $R\boxtimes\tilde R$ are $(R\boxtimes\tilde R)^{(\alpha i)(\beta j)}_{(\gamma k)(\delta l)}=R^{\alpha\beta}_{\gamma\delta}\tilde R^{ij}_{kl}$, where $\alpha,\beta,\gamma,\delta\in\{1,\ldots,d\}$ and $i,j,k,l\in\{1,\ldots,\tilde d\}$. Thus
	\begin{align*}
		R\boxtimes\tilde R
		&=
		\sum R^{\alpha\beta}_{\gamma\delta}\tilde R^{ij}_{kl} U_{\alpha i}U_{\beta j}U_{\delta l}^*U_{\gamma k}^*
		\\
		&\stackrel{\iota_{d\tilde d}}{\longmapsto}
		\sum R^{\alpha\beta}_{\gamma\delta}\tilde R^{ij}_{kl} S_\alpha S_\beta S_\delta^* S_\gamma^* \otimes S_iS_jS_l^*S_k^*
		=
		R\otimes\tilde R,
	\end{align*}
	and the calculation in \ref{item:otimes-and-restriction} shows $(\la_R\otimes\la_{\tilde R})|_{\CO_{d\tilde d}}=\la_{R\boxtimes\tilde R}$.
\end{proof}

Let us look at two special cases, the identity $1_d\in\CO_d$ and the flip $F_d\in\CO_d$. Then $1_d\boxtimes1_{\tilde d}=1_{d\tilde d}$ and $F_d\boxtimes F_{\tilde d}=F_{d\tilde d}$. For the canonical $2\pi$-periodic actions of $\mathbb R$, this implies that $\lambda_{e^{it} 1_d} \otimes \lambda_{e^{it} 1_{\tilde d}} \in \Aut(\CO_d \otimes \CO_{\tilde d})$ restricts to $\lambda_{e^{2it} 1_{d\tilde d}}$ on $\CO_{d\tilde d}$, and for the canonical shifts, this implies that $\varphi_d\ot\varphi_{\tilde d}$ restricts to $\varphi_{d\tilde d}$. Indeed, for all $i$ and $j$,
\begin{align*}
	\varphi_{d\tilde d} (S_i \otimes \tilde S_j) 
	&= 
	\sum_{i',j'} (S_{i'} \otimes \tilde S_{j'}) (S_i \otimes \tilde S_j) (S_{i'} \otimes \tilde S_{j'})^* 
	\\
	&= 
	\big(\sum_{i'} S_{i'}S_i S_{i'}^* \big) \otimes \big(\sum_{j'} \tilde S_{j'}\tilde S_j \tilde S_{j'}^* \big) 
	= 
	\varphi_d(S_i) \otimes \varphi_{\tilde d}(\tilde S_j).
\end{align*}
Notice that the index of $\varphi_d(\CN_d)\subset\CN_d$ is $d^2$, so that in this example we see immediately that the index of the endomorphism associated to the tensor product $F_d\boxtimes F_{\tilde d}$ is the product of the indices of the endomorphisms given by $F_d$ and $F_{\tilde d}$.

This is an instance of a general fact. Since $\CF_{d\tilde d}$ is identified with $\CF_d \otimes \CF_{\tilde d}$, the same holds on the level of the weak closures, and
\begin{align}
	\lambda_{R \boxtimes \tilde R}(\CN_{d\tilde d})
	=
	(\lambda_{R} \otimes \lambda_{\tilde R})(\CN_d \otimes \CN_{\tilde d})
	= 
	\lambda_R(\CN_d) \otimes \lambda_{\tilde R}(\CN_{\tilde d})
	.
\end{align}
From here we readily get the multiplicativity of the Jones index under the tensor product.
\begin{theorem}
	Let $R \in\CU(\CF_d^2)$, $\tilde R\in\CU(\CF_{\tilde d}^2)$. Then the Jones indices of the type II$_1$ subfactors associated to $R$, $\tilde R$ and $R \boxtimes \tilde R$ are related by
	\begin{align}
		[\CN_{d\tilde d}:\lambda_{R \boxtimes \tilde R}(\CN_{d\tilde d})]
		= 
		[\CN_d:\lambda_R(\CN_d)] \cdot [\CN_{\tilde d}: \lambda_{\tilde R}(\CN_{\tilde d})] \ . 
	\end{align}
\end{theorem}

Since this result applies in particular to R-matrices, we see that the subset of the positive real line ${\mathbb R}_+$ of all Jones indices arising from unitary solutions of the YBE (in any dimension) is closed under taking ordinary products.

\medskip

Concerning the relative commutants associated to the tensor product, we record the following result.
\begin{proposition} 
	Let $R\in\CR(d)$, $\tilde R\in\CR(\tilde d)$. Then 
	\begin{align}
        \CM_{R,1} \otimes \CM_{\tilde R,1} 
        \subseteq 
        \CM_{R \boxtimes \tilde R,1} 
        \subseteq 
        \CN_{R \boxtimes \tilde R,1} 
        = 
        \CN_{R,1} \otimes \CN_{\tilde R,1}  \ . 
	\end{align}
\end{proposition}
\begin{proof}
	On the one hand,
	\begin{align*}
		\CM_{R,1} \otimes \CM_{\tilde R,1} 
		&= 
		\{ x \in \CF_d^1 \ : \ \lambda_{R^*}(x) = \varphi_d(x) \} 
		\otimes 
		\{ y \in \CF_{\tilde d}^1 \ : \ \lambda_{\tilde R^*}(y) = \varphi_{\tilde d}(y) \}
		\\
		& \subseteq  
		\{ T \in \CF_d^1 \otimes \CF_{\tilde d}^1 \ : \ (\lambda_{R^*} \otimes \lambda_{\tilde R^*})(T) = (\varphi_d \otimes \varphi_{\tilde d})(T)\}
		\\
		&= 
		\{T \in \CF_{d\tilde d}^1 \ : \ \lambda_{(R \boxtimes \tilde R)^*}(T) = \varphi_{d\tilde d}(T)\}
		\\
		&= 
		\CM_{R \boxtimes \tilde R,1}. 
	\end{align*}
	On the other hand,
	\begin{align*}
		\CM_{R \boxtimes \tilde R,1}
		&\subseteq 
		\CN_{R \boxtimes \tilde R,1}
		= 
		\{T \in \CF_{d\tilde d}^1 \ : \ [T, \lambda_{R \boxtimes \tilde R}(x)] = 0, \ x \in \CF_{d\tilde d}\ \} 
		\\
		&= 
		\{T \in \CF_d^1 \otimes \CF_{\tilde d}^1 \ : \ [T,(\lambda_R\otimes \lambda_{\tilde R})(x)]=0, \ x \in \CF_d \otimes \CF_{\tilde d}\} 
		\\
		&= 
		\big(\lambda_R(\CF_d) \otimes \lambda_{\tilde R}(\CF_{\tilde d})\big)' \cap \big(\CF_d^1 \otimes \CF_{\tilde d}^1\big) 
		\\
		& 
		= \big(\lambda_R(\CN_d)' \otimes \lambda_{\tilde R}(\CN_{\tilde d})' \big) \cap (\CF_d^1 \otimes \CF_{\tilde d}^1) 
		\\
		& 
		= 
		\CN_{R,1} \otimes \CN_{\tilde R,1}.
\end{align*}
\end{proof}

\subsection{Cabling powers of R-matrices}

The second algebraic operation on $\CR$ that we want to discuss are cabling powers \cite{RehrenSchroer:1989,Wenzl:1990}. Given $d,n\in\Nl$, we define ``cabling maps'' between type II$_1$-factors, $c_n:\CN_d\to\CN_{d^n}$, such that $c_n(\bigotimes_{i=1}^{nm} M_d)=\bigotimes_{i=1}^m M_{d^n}$ for all $m \geq 1$, by linear and weakly continuous extension from algebraic tensor products,
\begin{align}
 c_n(\bigotimes_{i=1}^{nm}x_i)
 :=
 (\bigotimes_{i=1}^n x_i)\otimes (\bigotimes_{i=n+1}^{2n} x_i)
 \otimes\ldots\otimes 
  (\bigotimes_{i=(m-1)n+1}^{nm} x_i)\,,\qquad x_i\in M_d.
\end{align}
It follows that $c_n$ is an isomorphism with the properties 
\begin{align*}
 c_n(1)=1
 \,\qquad 
 \tau_{d^n}\circ c_n=\tau_d\,\qquad
 \varphi_{d^n}\circ c_n=c_n\circ\varphi_d^n,\qquad c_n(\CF_d^{kn})=\CF_{d^n}^k,\;\,k\in\Nl.
\end{align*}

To define the $n$-th cabling power of $R\in\CR(d)$, we also introduce
\begin{align*}
 {}_nR_n
 &:=({}_nR)_n
 \\
 &=
 {}_nR\cdots\varphi^{n-1}({}_nR)
 \\
 &=
 \varphi^{n-1}(R)\cdots R
 \cdot
 \varphi^{n}(R)\cdots \varphi(R)
 \cdots
 \varphi^{2n-2}(R)\cdots \varphi^{n-1}(R)
 \\
 &=
 {}_n(R_n).
\end{align*}
Note that ${}_nR_n$ is a unitary in $\CF_d^{2n}$ which satisfies $({}_nR_n)^*={}_n(R^*)_n$. For low $n$, we have ${}_1R_1=R$ and ${}_2R_2=\varphi(R)R\varphi^2(R)\varphi(R)$. A graphical illustration of ${}_3R_3$ is given in Fig.~\ref{figure:cabling}.

\begin{figure}[h!]
	\includegraphics[width=80mm]{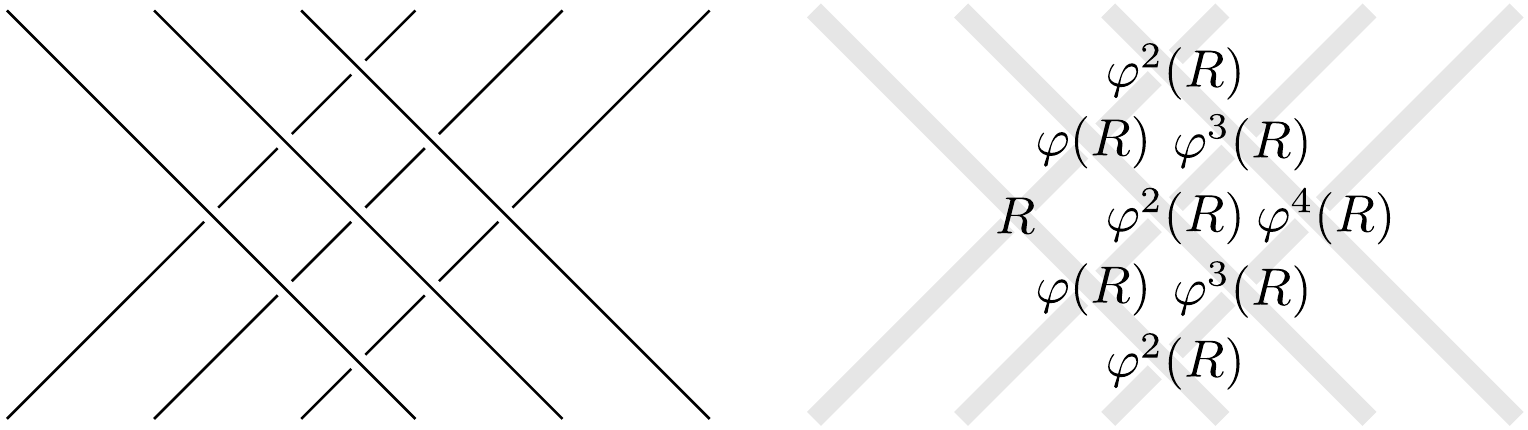}
    \caption{\small Illustration of ${}_3R_3=\varphi^2(R)\varphi(R)R\cdot\varphi^3(R)\varphi^2(R)\varphi(R)\cdot\varphi^4(R)\varphi^3(R)\varphi^2(R)$}
	\label{figure:cabling}
\end{figure}

Wenzl's cabling powers of an R-matrix take in our setting the following form.

\begin{definition}
	Let $R\in\CR(d)$ and $n\in\Nl$. The {\em $n$-th cabling power of $R$ is}
	\begin{align}
		R^{(n)}
		:=
		c_n({}_nR_n)
		\in\CU(\CF_{d^n}^2)
	\end{align}
	$R^{(n)}$ is an R-matrix in $\CR(d^n)$, and $(R^{(n)})^*=(R^*)^{(n)}$.
\end{definition}
The proof that $R^{(n)}\in\CR(d^n)$ can be found in \cite{Wenzl:1990}. 

We now show that at least on the level of the type II factor $\CN$, cabling powers of R-matrices correspond to ordinary powers of their corresponding Yang-Baxter endomorphisms.

\begin{proposition}\label{prop:cabling-and-powers}
	Let $R\in\CR(d)$ and $n\in\Nl$. Then
	\begin{align}\label{eq:cabling-and-powers}
		(c_n^{-1}\la_{R^{(n)}}c_n)(x)=\la_R^n(x),\qquad x\in\CN_d.
	\end{align}
	In particular,
	\begin{align}\label{eq:cabling-and-index}
	[\CN_{d^n}:\la_{R^{(n)}}(\CN_{d^{n}})]
	=
	[\CN_d:\la_R(\CN_d)]^n.
	\end{align}
\end{proposition}
\begin{proof}
	We calculate, $k\in\Nl$,
	\begin{align*}
		c_n^{-1}(&(R^{(n)})_k)
		=
		c_n^{-1}(R^{(n)}\cdots\varphi_{d^n}^{k-1}(R^{(n)}))
		\\
		&=
		{}_nR_n\cdot\varphi_d^n({}_nR_n)\cdots\varphi_d^{n(k-1)}({}_nR_n)\\
		&=
		{}_nR\cdots\varphi^{n-1}({}_nR)\cdot\varphi^n({}_nR)\cdots\varphi^{2n-1}({}_nR)\cdot\ldots\cdot
		\varphi^{n(k-1)}({}_nR)\cdots\varphi^{nk-1}({}_nR)
		\\
		&=
		({}_nR)_{kn}.
	\end{align*}
	Hence, for any $x\in\CN_d$,
	\begin{align*}
	 (c_n^{-1}\la_{R^{(n)}}c_n)(x)
	 &=
	 \lim_{k\to\infty}\ad c_n^{-1}((R^{(n)})_k)(x)
	 \\
	 &=
	\lim_{k\to\infty}\ad (({}_nR)_{nk})(x)=\la_{{}_nR}(x)=\la_R^n(x).
	\end{align*}
	As all the subfactors $\la_R^{k+1}(\CN)\subset\la_R^k(\CN)$, $k\in\Nl_0$, are isomorphic, this implies the index formula \eqref{eq:cabling-and-index}.
\end{proof}

\begin{remark}
	Let $R\not\in\Cl$ be non-trivial, and recall that $\la_R^n$ is reducible for $n\geq2$ in the sense that $\CM_{R,n}\neq\Cl$; namely $R\in\CM_{R,2}\subset\CN_{R,2}$. Thus Prop.~\ref{prop:cabling-and-powers} immediately implies that $\la_{R^{(n)}}$ is reducible as an endomorphism of $\CN_{d^n}$. This remains true on the level of the III$_{1/d^n}$-factor because $c_n(R)\in\CM_{R^{(n)},1}$.
\end{remark}

Our two elementary standard examples, the identity and the flip, reproduce themselves under taking cabling powers, i.e. $1_d^{(n)}=1_{d^n}$ and $F_d^{(n)}=F_{d^n}$. For later reference, we note that this implies in particular
\begin{align}\label{eq:shiftcablingpowers}
    \varphi_{d^n} = \la_{F^{(n)}}\in\End\CN_{d^n},\qquad
    \phi_{F^{(n)}} = c_n\circ\phi_F^n\circ c_n^{-1}. 
\end{align}

\subsection{Sums of R-matrices}\label{section:sums}

The third operation on $\CR$ that we want to discuss is additive on dimension. Given $R\in\CR(d)$, $\tilde R\in\CR(\tilde d)$, we define $R\boxplus\tilde R\in\End((\Cl^d\oplus\Cl^{\tilde d})\ot(\Cl^d\oplus\Cl^{\tilde d}))$ by \cite{LechnerPennigWood:2019}
\begin{align}\label{eq:DefBoxplus}
	R\boxplus \tilde R 
	&:= R\oplus \tilde R\oplus F\quad\text{on}\\
	(\Cl^d\oplus\Cl^{\tilde d})\ot(\Cl^d\oplus\Cl^{\tilde d}) 
	&= (\Cl^d\ot \Cl^{d})\oplus(\Cl^{\tilde d}\ot \Cl^{\tilde d})\oplus((\Cl^{d}\ot \Cl^{\tilde d})\oplus(\Cl^{\tilde d}\ot \Cl^{d})).\nonumber
\end{align}
In other words, $R\boxplus \tilde R$ acts as $R$ on $\Cl^{d}\ot\Cl^{d}$, as $\tilde R$ on $\Cl^{\tilde d}\ot\Cl^{\tilde d}$, and as the flip on the mixed tensors involving factors from both, $\Cl^{d}$ and $\Cl^{\tilde d}$.

If $R,\tilde R$ are R-matrices, then so is $R\boxplus \tilde R$ \cite{LechnerPennigWood:2019}. We also mention that we clearly have $(R\boxplus \tilde R)^*=R^*\boxplus \tilde R^*$, and $F_d\boxplus F_{\tilde d}=F_{d+\tilde d}$. The identity is however not preserved under this sum. For example, we have $1_1\boxplus1_1=F_2$.

Given $R\in\CR(d)$, $\tilde R\in\CR(\tilde d)$, we get an endomorphism $\la_{R\boxplus\tilde R}\in\End(\CO_{d+\tilde d})$. We currently have no detailed picture of $\la_{R\boxplus\tilde R}$. However, it is clear that $\la_{R\boxplus\tilde R}$ is always reducible, as follows from the following result.

\begin{proposition}
	Let $R\in\CR(d),\tilde R\in\CR(\tilde d)$. Then 
	\begin{align}\label{eq:MR+S1}
		\CM_{R,1}\oplus\CM_{\tilde R,1}\subset\CM_{R\boxplus \tilde R,1};
	\end{align}
	in particular $\la_{R\boxplus \tilde R}$ is always reducible. The inclusion \eqref{eq:MR+S1} is proper in general. We also have
	\begin{align}
		\phi_{R\boxplus \tilde R}(R\boxplus \tilde R)
		=
		\frac{d}{d+\tilde d}\,\phi_R(R) \oplus \frac{\tilde d}{d+\tilde d}\,\phi_{\tilde R}(\tilde R).
	\end{align}
\end{proposition}
\begin{proof}
	Let $x\in\CM_{R,1}\subset\CF_d^1$ and $\tilde x\in\CM_{\tilde R,1}\subset\CF_{\tilde d}^1$, i.e. $R^*xR=\varphi_d(x)$ and $\tilde R^*\tilde x\tilde R=\varphi_{\tilde d}(\tilde x)$. We may view $\CF_{d+\tilde d}^1$ as $\End(\Cl^d\oplus\Cl^{\tilde d})$, and define $p:=1\oplus0$, $p^\perp:=1-p=0\oplus1$ to be the orthogonal projections onto the two summands. Then
	\begin{align*}
		(R^*\boxplus\tilde R^*)(x\oplus\tilde x)&(R\boxplus\tilde R)
		=
		(R^*\boxplus\tilde R^*)(pxp+p^\perp \tilde xp^\perp)\varphi_{d+\tilde d}(p+p^\perp)(R\boxplus\tilde R)
		\\
		&=
		p\varphi_d(p)R^*xRp\varphi_d(p)+p^\perp\varphi_{\tilde d}(p^\perp)\tilde R^*\tilde x\tilde Rp^\perp\varphi_{\tilde d}(p^\perp)
		\\
		&\qquad+
		\varphi_d(pxp)p^\perp + \varphi_{\tilde d}(p^\perp\tilde xp^\perp)p
		\\
		&=
		p\varphi_d(pxp)+p^\perp\varphi_{\tilde d}(p^\perp\tilde x p^\perp)
		+
		\varphi_d(pxp)p^\perp + \varphi_{\tilde d}(p^\perp\tilde xp^\perp)p
		\\
		&=
		\varphi_d(pxp)+\varphi_{\tilde d}(p^\perp\tilde x p^\perp)
		\\
		&=
		\varphi_{d+\tilde d}(x\oplus\tilde x).
	\end{align*}
	This proves $x\oplus\tilde x\in\CM_{R\boxplus \tilde R,1}$.
	
	The second statement follows from Thm.~\ref{thm:commuting-squares}: For each R-matrix $R\in\CR(d)$, we have $\phi_R(R)=\phi_F(R)$ with $F\in\CR(d)$ the flip, i.e. $\phi_R(R)$ coincides with the normalized left partial trace of $R$. The claim then follows from the fact that the non-normalized partial trace maps $\boxplus$ sums to direct sums \cite[Lemma 4.2~iv)]{LechnerPennigWood:2019}.
\end{proof}

\begin{remark}
    The sum operation $\boxplus$ allows us to write down many examples of R-matrices and is the concept behind the definition of simple R-matrices (Def.~\ref{def:diagonalR}). Namely, we can start from trivial R-matrices $R=c\cdot1_d\in\CR(d)$, $c\in\T$, and build non-trivial ones by summation, i.e.
    \begin{align*}
        R=c_11_{d_1}\boxplus c_21_{d_2}\boxplus\ldots\boxplus c_N1_{d_N}\in\CR(d_1+\ldots+d_N),\qquad c_1,\ldots,c_N\in\T.
    \end{align*}
    Note that we may describe such R-matrices equivalently as follows: There is a partition of unity in $\CF_d^1$, i.e. pairwise orthogonal projections $p_1,\ldots,p_N\in\CF_d^1$ such that $p_1+\ldots+p_N=1$. To each projection~$p_i$, we have associated a phase factor $c_i\in\T$. Then 
    \begin{align}\label{eq:easyR}
        R=\sum_{i=1}^N c_i\,(p_i\ot p_i)+ F\sum_{\genfrac{}{}{0pt}{2}{i,j=1}{i\neq j}}^N (p_i\ot p_j),
    \end{align}
    which we realize to be a special form of simple R-matrix (Def.~\ref{def:diagonalR}). The more general form \eqref{eq:simpleR} is obtained by a slightly more general form of sum $\boxplus$, involving the parameters $c_{ij}$, $i\neq j$.
\end{remark}

\section{Equivalences of R-matrices}\label{section:equivalence}

In the last section, we related natural operations on R-matrices to operations on their endomorphisms. Conversely, one can start from a natural operation/relation on endomorphisms and relate it to structure on the level of the underlying R-matrices. The most obvious operation, namely composition of endomorphisms, does however not preserve the YBE, i.e. the product of two Yang-Baxter endomorphisms is usually not Yang-Baxter. Instead, we will consider equivalence relations given by conjugation with automorphisms, and define corresponding equivalence relations on $\CR(d)$.

\begin{definition}\label{def:equivalence}
	Let $R,S\in\CR(d)$. 
	\begin{enumerate}
		\item $R,S$ are {\em $\CM$-equivalent} iff there exists an automorphism $\alpha\in\Aut\CM$ such that $\la_R=\alpha\circ\la_S\circ\alpha^{-1}$, and we write $R\Msim S$ in this case.
		\item $R,S$ are {\em $\CN$-equivalent} iff there exists an automorphism $\beta\in\Aut\CN$ such that ${\la_R}|_{\CN}=\beta\circ{\la_S}|_{\CN}\circ\beta^{-1}$, and we write $R\Nsim S$ in this case.
		\item $R,S$ are {\em equivalent} iff there exists an isomorphism $\gamma_{R,S}:\CL_R\to\CL_S$ such that $\gamma_{R,S}(R)=S$ and $\varphi(\gamma_{R,S}(x))=\gamma_{R,S}(\varphi(x))$ for all $x\in\CL_R$, and we write $R\sim S$ in this case.
		\item $R,S$ have {\em equivalent representations} iff for each $n\in\Nl$, the representations $\rho_R^{(n)}$ and $\rho_S^{(n)}$ of the braid group $B_n$ on $n$ strands are unitarily equivalent. 
	\end{enumerate}
\end{definition}

It is clear that the subfactors $\la_R(\CM)\subset\CM$, $\la_R(\CN)\subset\CN$, and $\varphi(\CL_R)\subset\CL_R$ are equivalent to $\la_S(\CM)\subset\CM$, $\la_S(\CN)\subset\CN$, and $\varphi(\CL_S)\subset\CL_S$) if $R\Msim S$, $R\Nsim S$ and $R\sim S$, respectively. It is also clear that the relations $\Msim$, $\Nsim$, $\sim$ are different from each other.

The last equivalence relation (equivalence of representations) was originally introduced in \cite{AlazzawiLechner:2016} and played a prominent role in the classification of involutive R-matrices \cite{LechnerPennigWood:2019}. It essentially captures the {\em character} of an R-matrix, defined as the positive definite normalized class function
\begin{align}
 \tau_R:B_\infty\to\Cl,\qquad \tau_R:=\tau\circ\rho_R.
\end{align}
Equivalence of representations turns out to be the same as equivalence ($\sim$):

\begin{proposition}\label{prop:character}
    Let $R,S\in\CR(d)$. The following are equivalent:
    \begin{enumerate}
		\item\label{item:equivrep} $R$ and $S$ have equivalent representations.
		\item\label{item:sim} $R\sim S$.
		\item\label{item:character} $R$ and $S$ have the same character $\tau_R=\tau_S$.
    \end{enumerate}
\end{proposition}
\begin{proof}
	\ref{item:equivrep}$\implies$\ref{item:sim} If $R$ and $S$ have equivalent representations, there exist unitaries $Y_n\in\CU(\CF_d^n)$ such that $Y_n\varphi^k(R)Y_n^*=\varphi^k(S)$, $k\in\{0,1,\ldots,n-2\}$. This implies that for any $x\in\rho_R(\Cl B_\infty)$,
	\begin{align}\label{eq:gammalimit}
	 \gamma_{R,S}(x):=\lim_{n\to\infty}Y_nxY_n^*
	\end{align}
	exists, and the so defined map $\gamma_{R,S}$ is an isomorphism $\rho_R(\Cl B_\infty)\to\rho_S(\Cl B_\infty)$ with $\gamma_{R,S}(\varphi^k(R))=\varphi^k(S)$, $k\in\Nl_0$. Obviously $\gamma_{R,S}$ preserves $\tau$ and extends to an isomorphism $\CL_R\to\CL_S$ (denoted by the same symbol). 
	
	It remains to show $\varphi(\gamma_{R,S}(x))=\gamma_{R,S}(\varphi(x))$ for all $x\in\CL_R$. Indeed,
	\begin{align*}
		\gamma_{R,S}(\varphi(x))
		&=
		\gamma_{R,S}(\la_R(x))
		=
		\wlim_{n\to\infty}\gamma_{R,S}((\ad R_n)(x))
		\\
		&=
		\wlim_{n\to\infty}(\ad S_n)(\gamma_{R,S}(x))
		=
		\la_S(\gamma_{R,S}(x))
		=
		\varphi(\gamma_{R,S}(x))
		.
	\end{align*}
	Hence $R\sim S$.
	
	\ref{item:sim}$\implies$\ref{item:character} Let $R\sim S$. From the definition of this equivalence relation, we have an isomorphism $\gamma_{R,S}:\CL_R\to\CL_S$ such that $\gamma_{R,S}\circ\rho_R=\rho_S$, and the uniqueness of the trace implies that $\gamma_{R,S}$ preserves $\tau$. Hence, for any $b\in B_\infty$,
	\begin{align*}
	 \tau_S(b)
	 &=
	 \tau(\rho_S(b))
	 =
	 \tau(\gamma_{R,S}(\rho_R(b)))
	 =
	 \tau(\rho_R(b))
	 =
	 \tau_R(b).
	\end{align*}

	\ref{item:character}$\implies$\ref{item:equivrep} 
	Let $R,S$ have coinciding characters $\tau_R=\tau_S$, and pick $n\in\Nl$, $x\in\Cl B_n$. Then
    \begin{align*}
     \tau(\rho_R^{(n)}(x)^*\rho_R^{(n)}(x))=\tau_R(x^*x)=\tau_S(x^*x)=\tau(\rho_S^{(n)}(x)^*\rho_S^{(n)}(x)),
    \end{align*}
    and the faithfulness of $\tau$ yields $\ker\rho_R^{(n)}=\ker\rho_S^{(n)}$. So $\alpha:\rho_R^{(n)}(\Cl B_n)\to\rho_S^{(n)}(\Cl B_n)$, $\rho_R^{(n)}(x)\mapsto \rho_S^{(n)}(x)$, is an isomorphism of finite-dimensional $C^*$-algebras. Furthermore, equality of characters $\tau_R=\tau_S$ implies $\tau\circ\alpha=\tau$ on $\rho_R^{(n)}(\Cl B_n)$.
    
    But a trace-preserving isomorphism of finite-dimensional $C^*$-algebras represented on Hilbert spaces of the same dimension is always implemented by a unitary between these Hilbert spaces, i.e. there exists a unitary $Y_n\in\CF_d^n$ such that $Y_n\rho_R(x)Y_n^{-1}=\rho_S(x)$, $x\in\Cl B_n$. This shows that $R$ and $S$ have equivalent representations.   
\end{proof}

We mention as an aside that we may view $\tau_R$ as a state on $\Cl B_\infty$, and that the von Neumann algebra generated by the GNS construction of $(\Cl B_\infty,\tau_R)$ is naturally isomorphic to the factor $\CL_R$. Thus we see that $\tau_R$ is an {\em extremal} (or indecomposable) character, i.e. an extreme point in the convex set of positive normalized class functions, generalizing a result from \cite{LechnerPennigWood:2019} to non-involutive R-matrices. 

\bigskip

In general, the character equivalence relation $\sim$ does not imply the ``higher'' equivalences $\Nsim$, $\Msim$, but sometimes $\gamma_{R,S}:\CL_R\to\CL_S$ extends to appropriate automorphisms of $\CN$ or $\CM$. In the following, we discuss three example scenarios that we will subsequently refer to as ``type 1--3''.

\begin{description}\label{page:types}
	\item[Type 1] Let $R\in\CR(d)$ and $u\in\CU(\CF_d^1)$. Then $S:=u\varphi(u)R\varphi(u)^*u^*=\la_u(R)\in\CR(d)$ and $R\sim S$. One can choose the intertwiners as $Y_n:=u_n$, and easily verifies that $\la_u$ is an automorphism satisfying $\la_S=\la_u\circ\la_R\circ\la_u^{-1}$. Since $\la_u$ leaves $\CN$ invariant, we have $R\Msim S$ and $R\Nsim S$ in this case, with the isomorphisms $\alpha,\beta,\gamma_{R,S}$ from the various equivalence relations all being given by (restrictions of) $\la_u$.
	
	\item[Type 2] Let $R\in\CR(d)$ and $u\in\CU(\CF_d^1)$ such that $\la_u(R)=R$ (i.e., $R$ commutes with $u\varphi(u)$). Then $S:=\varphi(u)R\varphi(u)^*\in\CR(d)$ and $R\sim S$. One can choose the intertwiners as $Y_n:=u\varphi(u^2)\cdots\varphi^{n-1}(u^n)$. Hence in this case, $\gamma_{R,S}$ is given by 
	\begin{align}
		\Lambda_u:=\lim_{n\to\infty}{\ad}(u\varphi(u^2)\cdots\varphi^{n-1}(u^n)),
	\end{align}
	which trivially exists as an automorphism of $\bigcup_n\CF_d^n\subset\CN$ and extends to $\CN$. Clearly $\Lambda_u$ restricts to an isomorphism $\CL_R\to\CL_S$ matching the representations $\rho_R$ and $\rho_S=\Lambda_u\circ\rho_R$. For $x\in\CF_d^n$, we therefore have
	\begin{align*}
	 \Lambda_u(\la_R(x))
	 =
	 \Lambda_u(R_nx{R_n}^*)
	 =
	 S_n\Lambda_u(x){S_n}^*
	 =
	 \la_S(\Lambda_u(x)).
	\end{align*}
	Hence in this case, we also have $R\Nsim S$.
	
	Note that in this case, we have $\varphi(u)R\varphi(u)^*=u^*Ru$, so exchanging $u$ with $u^*$ we also have the $\CN$-equivalence $R\sim uRu^*$, with isomorphism $\Lambda_{u^*}$.
	
	We give an example to show that $\Lambda_u$ does in general not extend to $\CM$, i.e. to an $\CM$-equivalence $R\Msim S$.
	
	\begin{example}
		Let $u\in\CF_d^1$ and $R:=uFu^*$. Since the flip $F$ commutes with $u\varphi(u)$, we have $R\Nsim F$, and now show $R\not\Msim F$. In fact, if we had $R\Msim F$, then the type III subfactors given by $R$ and $F$ would be equivalent, and in particular their relative commutants $\CM_{R,1}$ and $\CM_{F,1}$ would have the same dimension. Recalling $\CM_{R,1}=\{x\in\CF_d^1:\varphi(x)=R^*xR\}$ \eqref{eq:MR1}, we have $\CM_{F,1}=\CF_d^1$. But as shown in Remark~\ref{remark:N1andM1CanDiffer}, $\CM_{R,1}=\CM_{S,1}\neq\CF_d^1$ if $u\not\in\Cl$. Hence $R\not\Msim F$.
	\end{example}

	\item[Type 3] The third type of equivalence is given by an R-matrix $R$ and its ``flipped'' version $FRF$, where $F$ is the flip \cite{LechnerPennigWood:2019}. The corresponding intertwiners are best described in terms of the so-called {\em fundamental braids} $\Delta_n\in B_n$ \cite{Garside:1969}, defined recursively by
	\begin{align}\label{eq:fundamental-braids}
        \Delta_1:=e,\qquad \Delta_2:=b_1,\qquad \Delta_{n+1}:=b_1\cdots b_n\cdot\Delta_n.
	\end{align}
    The fundamental braids satisfy \cite{KasselTuraev:2008}
    \begin{align}\label{eq:fundamental-braid-property}
        \Delta_n b_k = b_{n-k}\Delta_n,\qquad k\in\{1,\ldots,n-1\}.
    \end{align}
    Moreover, $\Delta_n^2$ generates the center of $B_n$. In particular, $\Delta_n b \Delta_n^{-1}=\Delta_n^{-1}b\Delta_n$ for all $b\in B_n$.
	
	\begin{lemma}\label{lemma:intertwinerlemma}
        Let $R\in\CR(d)$. Then $FRF\in\CR(d)$ and $R\sim FRF$, and the intertwiners can be chosen as 
        \begin{align}\label{eq:R-FRF-intertwiners}
            Y_n 
            :=
            \rho_{FRF}(\Delta_n)\rho_F(\Delta_n)
            ,\qquad n\in\Nl.
        \end{align}
	\end{lemma}
	\begin{proof}
        We skip the straightforward proof of $FRF\in\CR(d)$.
        
        The representative $\rho_F(\Delta_n)\in\End((\Cl^d)\tp{n})$ of the  fundamental braid given by the involutive R-matrix $F$ acts by total inversion permutation of the $n$ tensor factors. In view of the tensor product structure of the representation $\rho_R$,
        \begin{align}
            \rho_F(\Delta_n)\varphi^{k-1}(R)\rho_F(\Delta_n)^{-1}
            =
            \varphi^{n-k-1}(FRF),\qquad k\in\{1,\ldots,n-1\}.
        \end{align}
        Using \eqref{eq:fundamental-braid-property}, this implies
        \begin{align*}
            Y_n\rho_R(b_k)Y_n^{-1}
            &=
            \rho_{FRF}(\Delta_n)\rho_F(\Delta_n)
            \varphi^{k-1}(R)
            \rho_F(\Delta_n)^{-1}\rho_{FRF}(\Delta_n)^{-1}
            \\
            &=
            \rho_{FRF}(\Delta_n)\rho_{FRF}(b_{n-k})\rho_{FRF}(\Delta_n)^{-1}
            \\
            &=
            \rho_{FRF}(b_k).
        \end{align*}
        As $b_1,\ldots,b_{n-1}$ generate $B_n$, this establishes the intertwiner property of $Y_n$.
	\end{proof}

	We add two more remarks that are special to the type 3 equivalence $R\sim FRF$. On the one hand, we note that given $R\in\CR$ and $x\in\CF_d^1$, the equation $\varphi(x)=RxR^*$ is equivalent to $\varphi(x)=FR^*FxFRF$. In view of \eqref{eq:MR1}, this gives an identification of relative commutants,
	\begin{align}
		\CM_{FRF,1}=\{x\in\CF_d^1\,:\,\varphi(x)=FR^*F\}=\CM_{R^*,1}.
	\end{align}
	Our second remark concerns the isomorphism $\gamma_{R,FRF}:\CL_R\to\CL_{FRF}$, which extends to an algebra closely related to the $C^*$-algebra $\CA_R^{(0)}$ introduced in \eqref{eq:AR}.
	
	\begin{lemma}\label{lemma:R-FRF}
		Let $R\in\CR(d)$, $n\in\Nl$, and $x\in\CF_d^n$ such that $\varphi(x)=\la_R(x)$ (this is satisfied in particular by any $x\in\CL_{R,n}$). Then
		\begin{align}
			Y_mxY_m^*=Y_nxY_n^*,\qquad m\geq n,
		\end{align}
		where $Y_m$ is the intertwiner \eqref{eq:R-FRF-intertwiners}. In particular, $\gamma_{R,FRF}=\lim_m \ad Y_m$ extends to such elements $x$, and $\gamma_{R,FRF}(x)=Y_nxY_n^*$ for all $x\in\CL_{R,n}$.
	\end{lemma}
	\begin{proof}
		To prove this lemma, we first establish a recursion relation for the intertwiners $Y_m$. We claim
		\begin{align}\label{eq:Ym-recursion}
			Y_{m+1}
			=
			Y_m\cdot\rho_F(b_1\cdots b_m)^{-1}\rho_R(b_1\cdots b_m),\qquad m\in\Nl.
		\end{align}
		To show this, recall that we already know the identity 
		\begin{align*}
            \rho_F(\Delta_m)\rho_R(b)\rho_F(\Delta_m)^{-1}
            =
            \rho_{FRF}(\Delta_m b\Delta_m^{-1}), \qquad b\in B_m;
        \end{align*}
		this was shown in the proof of Lemma~\ref{lemma:intertwinerlemma}. Thus we may rewrite the intertwiners as $Y_m=\rho_{FRF}(\Delta_m)\rho_F(\Delta_m)=\rho_F(\Delta_m)\rho_R(\Delta_m)$. 
		
		We furthermore note that $\rho_R(\Delta_m)\in\CL_{R,m}$ and therefore
		\begin{align*}
			\ad\rho_R(b_1\cdots b_m)[\rho_R(\Delta_m)]
			&=
			\la_R(\rho_R(\Delta_m))
			=
			\varphi(\rho_R(\Delta_m))
			\\
			&=\ad(\rho_F(b_1\cdots b_m))[(\rho_R(\Delta_m))].
		\end{align*}
		Moreover, since $F^2=1$, we have $\rho_F(\Delta_m)=\rho_F(\Delta_m^{-1})$. Together with the recursion relation $\Delta_{m+1}=b_1\cdots b_m\Delta_m$, this gives
		\begin{align*}
			Y_{m+1}
			&=
			\rho_F(\Delta_{m+1}^{-1})\rho_R(\Delta_{m+1})
			\\
			&=
			\rho_F(\Delta_m^{-1})\rho_F(b_1\cdots b_m)^{-1}\rho_R(b_1\cdots b_m)\rho_R(\Delta_m)
			\\
			&=
			\rho_F(\Delta_m^{-1})\rho_F(b_1\cdots b_m)^{-1}\varphi(\rho_R(\Delta_m))\rho_R(b_1\cdots b_m)
			\\
			&=
			\rho_F(\Delta_m^{-1})\rho_R(\Delta_m)\rho_F(b_1\cdots b_m)^{-1}\rho_R(b_1\cdots b_m)
			\\
			&=
			Y_m\cdot \rho_F(b_1\cdots b_m)^{-1}\rho_R(b_1\cdots b_m),
		\end{align*}
		proving \eqref{eq:Ym-recursion}.
		
		Now let $x\in\CF_d^n$ such that $\varphi(x)=\la_R(x)$. Then $\ad(\rho_R(b_1\cdots b_m))[x]=\la_R(x)=\varphi(x)=\ad(\rho_F(b_1\cdots b_m))[x]$ for any $m\geq n$, and therefore
		\begin{align*}
			\ad(Y_{n+1})(x)
			&=
			\ad(Y_{n})(x).
		\end{align*}
		Clearly, this implies $\ad Y_m(x)=(\ad Y_n)(x)$ for all $m\geq n$.
		
		The isomorphism $\gamma_{R,FRF}$ is defined by the limit formula $\lim_m\ad Y_m$ on $\rho_R(\Cl B_\infty)$ and showed that it uniquely extends to an isomorphism $\CL_R\to\CL_S$. Thus, as $\lim_m(\ad Y_m)(x)$ exists and equals $Y_nxY_n^*$ for $x\in\CF_d^n$ as in the statement of the lemma, we find $\gamma_{R,FRF}(x)=Y_nxY_n^*$ as claimed.	
	\end{proof}
\end{description}

Let us emphasize that in general, it is not known whether the $\sim$ equivalence class of an R-matrix is exhausted by the three cases listed above. Furthermore, in general the equivalences $R\Msim S$ or $R\Nsim S$ do not imply $R\sim S$ (For example, $R\Nsim-R$ for all $R\in\CR$, but usually $R\not\sim-R$.)

\medskip

Making use of the type 3 intertwiners, we can now also give the postponed second part of the proof of Theorem~\ref{thm:commuting-squares}.

\medskip

\noindent{\em Proof of Theorem~\ref{thm:commuting-squares} (second half).} \label{page:proofpart2}
	Let $R\in\CR$ and $S:=FRF$. We want to show that $\CL_R$ is invariant under $\phi_F$. As a preparation, we first show, $n\in\Nl$,
	\begin{align}\label{eq:step1}
		\ad\rho_F(\Delta_{n})(\CL_{R,n})
		=
		\CL_{S,n}.
	\end{align}
	In fact, we know from Lemma~\ref{lemma:R-FRF} that the intertwiner isomorphism $\gamma_{R,S}$ coincides with $\ad Y_n$ on $\CL_{R,n}$, with the intertwiners $Y_n=\rho_{S}(\Delta_n)\rho_F(\Delta_n)$ \eqref{eq:R-FRF-intertwiners}. Thus
	\begin{align*}
		\ad\rho_F(\Delta_n)(\CL_{R,n})
		=
		\ad\rho_S(\Delta_n)^{-1}(\ad Y_n(\CL_{R,n}))
		=
		\ad\rho_S(\Delta_n)^{-1}(\CL_{S,n})
		=
		\CL_{S,n},
	\end{align*}
	where the last step follows from $\ad\rho_S(\Delta_n)^{-1}$ being an inner automorphism of~$\CL_{S,n}$.

	Now let $x\in\CL_{R,n+1}$, $n\in\Nl_0$. As $\phi_F(x)$ acts by tracing out the first tensor factor of $x$ (see~\eqref{eq:phiF}), and $E_n(x)$ acts by tracing out the $(n+1)$st tensor factor of $x$, we have
	\begin{align}
		\phi_F(x)
		=
		E_n({F_n}^*xF_n)
		=
		E_n(\rho_F(b_1\cdots b_n)^{-1}x\rho_F(b_1\cdots b_n)).
	\end{align}
	Using the recursion relation $\Delta_{n+1}=b_1\cdots b_n\cdot \Delta_n$ for the fundamental braids and $\rho_F(\Delta_n)\in\CF_d^n$, we have
	\begin{align*}
		\phi_F(x)
		&=
		E_n(\ad\rho_F(\Delta_n\Delta_{n+1}^{-1})(x))
		=
		\ad\rho_F(\Delta_n)\left[E_n(\ad\rho_F(\Delta_{n+1}^{-1})(x))\right].
	\end{align*}
	In this formula, $\ad\rho_F(\Delta_{n+1}^{-1})(x)\in\CL_{S,n+1}$ by \eqref{eq:step1} (note $\rho_F(\Delta_{n+1}^{-1})=\rho_F(\Delta_{n+1})$), and thus $E_n(\ad\rho_F(\Delta_{n+1}^{-1})(x))\in\CL_{S,n}$ by the first part of Thm.~\ref{thm:commuting-squares}. If we now apply \eqref{eq:step1} once more, with the roles of $R$ and $S$ exchanged, we arrive at $\phi_F(x)\in\CL_{R,n}$.

	Proceeding to general $x\in\CL_R$, we have $E_n(x)\in\CL_{R,n}$ and $E_n(x)\to x$ weakly as $n\to\infty$. As we have just shown $\phi_F(E_n(x))\in\CL_R$ for all $n\in\Nl$ and $\phi_F$ is normal, it follows that $\phi_F(x)\in\CL_R$.

	The uniqueness of the $\tau$-preserving conditional expectation $E_R=\la_R\circ\phi_R$ of $\varphi(\CL_R)\subset\CL_R$ now implies that for any $x\in\CL_R$,
	\begin{align*}
		\varphi(\phi_R(x))=E_R(x)=E_F(x)=\varphi(\phi_F(x)),
	\end{align*}
	and thus $\phi_R(x)=\phi_F(x)$. This shows that the right diagram in \eqref{eq:commutingsquare2} is a commuting square for $n=1$, and the case $n>1$ follows by composing several isomorphic commuting squares.$\hfill\square$

\medskip

Applications of Thm.~\ref{thm:commuting-squares} will appear in the next section.

We now describe a situation in which $R\Msim S$ does imply $R\sim S$. 

\begin{proposition}\label{proposition:conjugation-vs-fixed-points}
    \leavevmode
    \begin{enumerate}
     \item\label{item:conjugationequivalencerelation}  Let $R,w\in\CU(\CO_d)$ such that $\alpha^{-1}:=\lambda_w\in\Aut\CM$. Then  
    \begin{align}\label{eq:conjug-eqn}
    \alpha\circ\lambda_R\circ\alpha^{-1}=\lambda_{\alpha(R)}
    \;\Longleftrightarrow\;
    w\in\CO_d^{\lambda_{\varphi(R)}}.
    \end{align}
    \item\label{item:equivalences} In the same situation as in~\ref{item:conjugationequivalencerelation}, assume in addition that $R\in\CR(d)$ and $S:=\alpha(R)\in\CF_d^2$. Then $S\in\CR(d)$ and $S\sim R$.
    \end{enumerate}
\end{proposition}
\begin{proof}
    \ref{item:conjugationequivalencerelation} We write $\alpha=\lambda_v$ and compute
 \begin{align*}
  \alpha\lambda_R\alpha^{-1}
  &=
  \lambda_v\lambda_R\lambda_w
  =
  \lambda_v\lambda_{\lambda_R(w)R}
  =
  \lambda_{\lambda_v(\lambda_R(w)R)v},
 \end{align*}
 which coincides with $\lambda_{\alpha(R)}=\lambda_{\lambda_v(R)}$ if and only if $\lambda_v(\lambda_R(w)R)v=\lambda_v(R)$. Applying $\lambda_w$ to both sides of this equation and observing that $\lambda_w\lambda_v=\id$ implies $\lambda_w(v)=w^*$, we see that $\alpha\circ\lambda_R\circ\alpha^{-1}=\lambda_{\alpha(R)}$ is equivalent to 
 \begin{align}
  w=(\ad R^*\circ\lambda_R)(w)=\lambda_{\varphi(R)}(w),
 \end{align}
 i.e. $w\in\CO_d^{\lambda_{\varphi(R)}}$.
 
 \ref{item:equivalences} We now assume that $R\in\CR(d)$ is an R-matrix, and set $S:=\alpha(R)$. Then, $n\in\Nl_0$,
 \begin{align*}
  \alpha(\varphi^n(R))
  =
  (\alpha\la_R^n\alpha^{-1})(S)
  =
  \la_S^n(S).
 \end{align*}
	In particular, $\varphi(\alpha(R))=\alpha(\varphi(R))$, which immediately implies $\varphi(S)S\varphi(S)=S\varphi(S)S$. Since $S\in\CF_d^2$ as well, $S$ is also an R-matrix. Thus $\la_S^n(S)=\varphi^n(S)$, i.e. we have $\alpha(\varphi^n(R))=\varphi^n(S)$, which shows that $\alpha$ restricts to an isomorphism $\CL_R\to\CL_S$ such that $\varphi(\alpha(x))=\alpha(\varphi(x))$ for all $x\in\CL_R$. This verifies the definition of $R\sim S$.
\end{proof}

We thus see that the enhanced form of $\Msim$ equivalence spelled out in \eqref{eq:conjug-eqn} is parameterized by the fixed points of $\la_{\varphi(R)}$. The structure of this fixed point algebra is elucidated in the following general lemma.

\begin{lemma}\label{lemma:fixedpointsandphi}
    Let $R\in\CU(\CO_d)$. Then
    \begin{align}
        \CF_d^1 &\subset \CO_d^{\la_{\varphi(R)}},\\
     \varphi(\CO_d^{\lambda_R}) & \subset \CO_d^{\lambda_{\varphi(R)}},
     \label{eq:fixed-point-algebra-inclusion}
     \\
     \phi_F(\CO_d^{\lambda_{\varphi(R)}}) & =\CO_d^{\lambda_R}, 
     \label{eq:fixed-point-algebra-equality}
    \end{align}
    and 
    \begin{align}\label{eq:trivial-fixed-points}
     \CO_d^{\lambda_R}=\C\;\Longleftrightarrow\;\CO_d^{\lambda_{\varphi(R)}}=\CF_d^1.
    \end{align}
\end{lemma}
\begin{proof}
 The first inclusion is trivial. Let $x\in\CO_d^{\lambda_R}$. Then 
 \begin{align*}
    \varphi(x)=\varphi(\lambda_R(x))=({\rm ad} R^*\circ\lambda_R)(\varphi(x))=\lambda_{\varphi(R)}(\varphi(x)),
 \end{align*}
 proving $\varphi(\CO_d^{\lambda_R})\subset \CO_d^{\lambda_{\varphi(R)}}$. Applying $\phi_F$, this also gives $\CO_d^{\lambda_R}\subset\phi_F(\CO_d^{\lambda_{\varphi(R)}})$. 
 
 Now let $i,j\in\{1,\ldots,d\}$ and $w\in\CO_d$. Then
 \begin{align}\label{fixedpointmap}
  \lambda_R(S_i^*wS_j)
  &=
  S_i^*R^*\lambda_R(w)RS_j
  =
  S_i^*\lambda_{\varphi(R)}(w)S_j,
 \end{align}
 and setting $i=j$ and summing over $i$, we find in particular $\phi_F\circ\lambda_{\varphi(R)}=\lambda_R\circ\phi$. For $w\in\CO_d^{\lambda_{\varphi(R)}}$, this implies $\phi_F(w)=\phi_F(\lambda_{\varphi(R)}(w))=\lambda_R(\phi_F(w))$, i.e. $\phi_F(\CO_d^{\lambda_{\varphi(R)}}) \subset \CO_d^{\lambda_R}$. 
 
 In particular, if $\CO_d^{\lambda_{\varphi(R)}}=\CF_d^1$, then $\CO_d^{\lambda_R}=\phi_F(\CF_d^1)=\C$. It remains to show that $\CO_d^{\lambda_R}=\C$ implies $\CO_d^{\lambda_{\varphi(R)}}=\CF_d^1$. Let $w\in\CO_d^{\lambda_{\varphi(R)}}$. In view of \eqref{fixedpointmap}, we then have $S_i^*wS_j\in\CO_d^{\lambda_R}$ for any $i,j$. In case $\CO_d^{\lambda_R}=\C$, this implies $S_i^*wS_j\in\C$ for any $i,j$, and thus
 \begin{align}
  w=\sum_{i,j=1}^dS_i(S_i^*wS_j)S_j^*\in{\rm span}\{S_iS_j^*\,:\,i,j\in\{1,\ldots,d\}\}=\CF_d^1,
 \end{align}
 as claimed.
\end{proof}

The last statement of this lemma implies that for $\CO_d^{\la_R}=\Cl$, the only possibility to satisfy \eqref{eq:conjug-eqn} is by $w\in\CF_d^1$. Another source of fixed points of $\la_{\varphi(R)}$ is $\varphi(\CO_d^{\la_R})$. In both cases, \eqref{eq:conjug-eqn} amounts to the ``type 1'' equivalence: 

\begin{lemma}
    Let $R\in\CR(d)$.
    \begin{enumerate}
     \item\label{item:type1} If $\lambda_R$ is ergodic (i.e. $\CO_d^{\la_R}=\Cl$), then $\alpha\lambda_R\alpha^{-1}=\lambda_{\alpha(R)}$ with  $\alpha\in\Aut\CO_d$ if and only if $\alpha=\la_u$ with $u\in\CU(\CF_d^1)$. In this case, $R\sim\alpha(R)$ is an example of the ``Type~1'' situation (p.~\pageref{page:types}).
     \item\label{item:type2} Let $u\in\CF_d^1$ be a unitary fixed point of $\la_R$ and $w:=\varphi(u)\in\CO_d^{\la_{\varphi(R)}}$. Then $\alpha\lambda_R\alpha^{-1}=\lambda_{\alpha(R)}$ with $\alpha:=\la_w^{-1}$, and $R\sim\alpha(R)$ are again type 1 equivalent.
    \end{enumerate}
\end{lemma}
\begin{proof}
    \ref{item:type1} For ergodic $\lambda_R$, we have $\CO_d^{\lambda_{\varphi(R)}}=\CF_d^1$ by Lemma~\ref{lemma:fixedpointsandphi}. But $\alpha\lambda_R\alpha^{-1}=\lambda_{\alpha(R)}$ is equivalent to $\alpha^{-1}=\lambda_w$ with $w\in\CO_d^{\lambda_{\varphi(R)}}$ (Prop.~\ref{proposition:conjugation-vs-fixed-points}), so that the conjugation equation is satisfied if and only if $\alpha$ is quasi-free. For quasi-free $\alpha$, it is clear that $\alpha(R)\in\CF_d^2$, which implies $\alpha(R)\in\CR_d$ and $R\sim\alpha(R)$ are type 1 equivalent.
    
    \ref{item:type2} Defining $\alpha:=\la_{\varphi(u)}^{-1}$, we have $\alpha=\la_{\varphi(u^*)}$ and $S:=\alpha(R)\in\CF_d^2$. Thus $S\in\CR(d)$ is an R-matrix equivalent to $R$. Since $u\in\CF_d^1$ is a fixed point of $\la_R$, it follows that $u$ and $R$ commute. Hence 
    \begin{align}
     S=\alpha(R)
     =
     \la_{\varphi(u^*)}(R)
     =
     \varphi(u^*)u^*Ru\varphi(u)
     =
     \la_{u^*}(R),
    \end{align}
    i.e. $S\sim R$ are type 1 equivalent.
\end{proof}

These observations show that the equivalence relation \eqref{eq:conjug-eqn} is closely related to type~1 equivalence. It is possible that both notions coincide.

The appearance of fixed points warrants a more systematic look at fixed points and ergodicity of Yang-Baxter endomorphisms. This is done in Section~\ref{section:ergodicity}.

\subsection{Equivalent R-matrices and braid group characters}\label{subsection:EquivalencesAndPartialTraces}

Whereas a classification of all R-matrices seems out of reach, a more accessible (though still challenging) question is to classify all Yang-Baxter {\em characters}, i.e. all traces $\tau_R$, $R\in\CR$, on $B_\infty$. This amounts to classifying R-matrices up to the equivalence relation $\sim$.

In order to explain how our results can contribute to this problem, it is instructive to compare this situation with the special case of {\em involutive} R-matrices (i.e. $R\in\CR(d)$ such that $R^2=1$, equivalently $R=R^*$) which has been studied before. Note that for involutive R-matrices, $\tau_R$ can be viewed as a character of the infinite symmetric group $S_\infty$ rather than the infinite braid group.

In preparation for the following, we define {\em R-matrices of normal form} to be special simple R-matrices (Def.~\ref{def:diagonalR}) with parameters $c_{ij}=1$ for $i\neq j$ and $\eps_i:=c_{ii}\in\{+1,-1\}$ for all $i$. That is, normal form R-matrices are given by a partition of unity $p_1,\ldots,p_N$ in $\CF_d^1$ and signs $\eps_1,\ldots,\eps_N$ such that
\begin{align}\label{eq:def-normal-form}
    R
    =
    \sum_{i=1}^N \eps_{i}\,p_i\varphi(p_i)+\sum_{\genfrac{}{}{0pt}{2}{i,j=1}{i\neq j}}^N p_i\varphi(p_j)F
    =
    \bigboxplus_{i=1}^N\eps_i 1_{d_i},
\end{align}
where $d_i=d\tau(p_i)$ are the dimensions of the projections $p_i$. These normal forms can be described by a pair of Young diagrams with $d$ boxes in total.

\begin{theorem}\label{theorem:involutivecase}\cite{LechnerPennigWood:2019}
    \leavevmode
    \begin{enumerate}
     \item Let $R,S\in\CR(d)$ be involutive. Then $R\sim S$ if and only if $\phi_R(R)\cong\phi_S(S)$ are similar, i.e. $\phi_R(R)=u\phi_S(S)u^*$ for some $u\in\CU(\CF_d^1)$.
     \item\label{item:normalform-invol} Each involutive $R$ is equivalent to a unique R-matrix of normal form. 
     \item Let $R$ be an R-matrix of normal form, with projections $p_1,\ldots,p_N$ and signs $\eps_1,\ldots,\eps_N$. Define the rational numbers 
        \begin{align}\label{eq:Thomaparameters}
            \alpha_i&:=\tau(p_i),\qquad \eps_{i}=+1,\\
            \beta_j&:=\tau(p_j),\qquad \eps_{j}=-1.
        \end{align}
    Then the character $\tau_R(\sigma)$, $\sigma\in S_\infty$, takes the following form: If the disjoint cycle decomposition of $\sigma$ is given by $m_n$ cycles of length $n$, $n\in\Nl$, then 
    \begin{align}\label{eq:ThomaCharacterFormula}
        \tau_R(\sigma)=\prod_n\left(\sum_i\alpha_i^n+(-1)^{n+1}\sum_j\beta_j^n\right)^{m_n}.
    \end{align}
    Furthermore, the signed parameters $\alpha_i$, $-\beta_j$ are exactly the eigenvalues of $\phi_R(R)$.
    \end{enumerate}
\end{theorem}

The proofs of these facts rely crucially on the fact that $\rho_R$ factors through the infinite symmetric group. In particular, i) a parameterization of all extremal characters of $S_\infty$ is known from the work of Thoma \cite{Thoma:1964} (in terms of the {\em Thoma parameters} $\alpha_i$, $\beta_j$ \eqref{eq:Thomaparameters}), ii) $S_\infty$ allows for a disjoint cycle decomposition, iii) for involutive R-matrices, $\phi_R(R)$ is selfadjoint, and iv) for involutive R-matrices, $\la_R$ is completely reducible in a sense to be described in Section~\ref{section:reduction-of-involutives}.

The results of Thm.~\ref{theorem:involutivecase} do not carry over to the case of general (not necessarily involutive) R-matrices. However, certain aspects can be generalized, which is the content of the following theorem.

\begin{theorem}\label{theorem:phiR}
    Let $R,S\in\CR(d)$. 
    \begin{enumerate}
     \item\label{item:left-and-right} $\phi_R(R)=\phi_F(R)=\phi_F(FRF)$ is a normal element of $\CF_d^1$ with norm $\|\phi_R(R)\|\leq1$. In particular, $R$ has identical left and right partial traces\footnote{In matrix notation, $\phi_F(R)=d^{-1}(\Tr\ot\id)(R)$ and $\phi_F(FRF)=d^{-1}(\id\ot\Tr)(R)$ are the normalized left and right partial traces of $R$.}.
     \item\label{item:R-cycles} $\tau(R\varphi(R)\cdots\varphi^{n-1}(R))=\tau(\phi_R(R)^n)$, $n\in\Nl_0$.
     \item\label{item:phiRR-is-an-invariant} If $R\sim S$, then $\phi_R(R)\cong\phi_S(S)$ (unitary similarity).
    \end{enumerate}
\end{theorem}
\begin{proof}
    \ref{item:left-and-right} By Thm.~\ref{thm:commuting-squares}, we know $\phi_F(x)=\phi_R(x)$ for all $x\in\CL_R$, so in particular $\phi_F(R)=\phi_R(R)$. We also know that $E_1(R)=\phi_F(FRF)\in\CL_{R,1}$. Given arbitrary $y\in\CF_d^1$, we compute
    \begin{align*}
        \tau(y\phi_F(FRF))
        =
        \tau(\varphi(y)FRF)
        =
        \tau(yR)
        =
        \tau(\la_R(y)R)
        =
        \tau(y\phi_R(R)),
    \end{align*}
    which shows $\phi_F(FRF)=\phi_R(R)$. 
    
    In general, left inverses/partial traces do not preserve normality, but in our situation, we can show that $\phi_R(R)$ is always normal, i.e. $\phi_R(R)\phi_R(R)^*=\phi_R(R)^*\phi_R(R)$. Since $\phi_R(R)\in\CF_d^1$, it is enough to compare traces against arbitrary elements $x\in\CF_d^1$.
    
    In the following computation, we use the property \eqref{eq:phiRproperty} of $\phi_R$ and $\tau\circ\phi_R=\tau$, the fact that $\la_R=\ad R$ on $\CF_d^1$, and $\la_R(R^*)=\varphi(R^*)$. This yields    
    \begin{align*}
        \tau(x\phi_R(R)\phi_R(R)^*)
        &=
        \tau(\la_R(x\phi_R(R))R^*)
        \\
        &=
        \tau(x\phi_R(R)R^*)
        \\
        &=
        \tau(\la_R(x)R\varphi(R^*))
        \\
        &=
        \tau(Rx\varphi(R^*))
        \\
        &=
        \tau(xR\varphi(R^*)).
    \end{align*}
    On the other hand, using $\phi_R(R)=\phi_F(FRF)$ and $\phi_R(R)=\phi_F(R)=\phi_{R^*}(R)$ (this follows because $R\in\CL_R=\CL_{R^*}$), we find
    \begin{align*}
        \tau(x\phi_R(R)^*\phi_R(R))
        &=
        \tau(x\phi_F(FR^*F)\phi_R(R))
        \\
        &=
        \tau(\varphi(x)FR^*F\varphi(\phi_R(R)))
        \\
        &=
        \tau(xR^*\phi_R(R))
        \\
        &=
        \tau(xR^*\phi_{R^*}(R))
        \\
        &=
        \tau(\la_{R^*}(x)\varphi(R^*)R)
        \\
        &=
        \tau(xR\varphi(R^*)),
    \end{align*}
    which coincides with the previous result. This proves that $\phi_R(R)$ is normal. The norm estimate is a standard property of the conditional expectation $E_R=\la_R\phi_R$.

     \ref{item:R-cycles} For $k,m\in\Nl_0$, define 
 \begin{align}
  t_{k,m}
  :=
  \tau(\varphi^k(R)\varphi^{k-1}(R)\cdots R\cdot\phi_R(R)^m).
 \end{align}
 We will prove $t_{k,m}=t_{k+1,m-1}$, which implies the claim as $t_{n,0}=t_{0,n}$.
 
As before, we use the four facts i) $x\phi_R(y)=\phi_R(\la_R(x)y)$, ii) $\la_R(a)=\varphi(a)$ for $a\in\CL_R$, iii) $\tau\circ\phi_R=\tau$, iv) $\la_R(\phi_R(R))=R\phi_R(R)R^*$, and compute
 \begin{align*}
  t_{k,m}
  &=
  \tau(\varphi^k(R)\cdots R\phi_R(R)^{m-1}\cdot\phi_R(R)).
  \\
  &=
  \tau\left(\phi_R\left(\la_R\left(\varphi^k(R)\cdots R\cdot\phi_R(R)^{m-1}\right)R\right)\right)    
  \\
  &=
  \tau\left(\varphi^{k+1}(R)\cdots \varphi(R)\cdot R\phi_R(R)^{m-1}R^*R\right)    
  \\
  &=
  t_{k+1,m-1}.
 \end{align*}

 \ref{item:phiRR-is-an-invariant} Let $R\sim S$, i.e. $\tau_R=\tau_S$. Then part~\ref{item:R-cycles} implies that $\phi_R(R)^n$ and $\phi_S(S)^n$ have the same trace for any $n\in\Nl_0$. Thus $\phi_R(R)$ and $\phi_S(S)$ have the same characteristic polynomial, and as they are normal by part~\ref{item:left-and-right}, it follows that $\phi_R(R)$ and $\phi_S(S)$ are unitarily equivalent. 
\end{proof}

\begin{remark}\leavevmode
	\begin{enumerate}
	 \item This theorem states in particular that the spectrum of the (left or right) partial trace of an R-matrix is an invariant for $\sim$. Since any normal matrix can be diagonalized by conjugation with a unitary, we also see that given $R\in\CR(d)$, there exists $u\in\CU(\CF_d^1)$ such that $\la_u(R)\sim R$ (``type 1'', see p.~\pageref{page:types}) and $\la_u(R)$ has diagonal left and right partial traces.
	 \item 	Whereas it is known in the setting of involutive R-matrices that $R\sim S$ is {\em equivalent} to $\phi_R(R)\cong\phi_S(S)$, the implication $\impliedby$ does {\em not} hold in general. In fact, it is not difficult to construct unitary R-matrices $R,S$ such that $\phi_R(R)=\phi_S(S)$ (and $R\cong S$), but for example $\tau(R^2\varphi(R))\neq\tau(S^2\varphi(S))$, i.e. $R\not\sim S$.
	 \item In the involutive case, it is furthermore known that $\phi_R(R)$ is always invertible and that all of its eigenvalues lie in $\Zl[\frac{1}{d}]$. We currently do not know whether the first statement (invertibility) holds in general, but it is easy to give examples of non-involutive R-matrices $R$ such that $\sigma(\phi_R(R))\not\subset\Zl[\frac{1}{d}]$.
	 \item Specializing to involutive R-matrices, part~\ref{item:R-cycles} recovers Thoma's character formula \eqref{eq:ThomaCharacterFormula} for cycles: On an $n$-cycle in $c_n\in S_\infty$, the character $\tau_R$ gives
	 \begin{align}\label{eq:ThomaOnCycle}
		\tau_R(c_n)
		=
		\sum_i \alpha_i^n+(-1)^{n+1}\sum_j\beta_j^n,
	 \end{align}
	 where $\alpha_i$, $\beta_j$ are the Thoma parameters of $R$ \eqref{eq:Thomaparameters}.
	\end{enumerate}
\end{remark}

\section{Irreducibility, Reduction, and Index}\label{section:irreducibility}

In the following we will call an R-matrix $R$ {\em irreducible} iff $\la_R$ is irreducible as an endomorphism of $\CM$, i.e. iff $\CM_{R,1}=\la_R(\CM)'\cap\CM=\Cl$. This does not necessarily mean that $\la_R$ is irreducible as an endomorphism of $\CN$: In view of \eqref{eq:inclusions-of-relative-commutants},
\begin{align}\label{eq:relativecommutantsinF1}
    \CL_{R,1}
    \subset
    \CM_{R,1}
    \subset
    \CN_{R,1}
    \subset
    \CF_d^1,\qquad R\in\CR(d),
\end{align}
and in general, the relative commutants $\CL_{R,1}$, $\CM_{R,1}$ and $\CN_{R,1}$ are all different from each other. It is therefore conceivable that there exist R-matrices such that, for instance, $\la_R$ is irreducible but ${\la_R}|_{\CN}$ is not, or that ${\la_R}|_{\CL_R}$ is irreducible but $\la_R$ is not\footnote{An example for the latter situation is given by $R=F$.}. Our notion of irreducibility always refers to $\la_R\in\End\CM$, and we will explicitly indicate whenever we consider $\la_R$ as an endomorphism of $\CN$ or $\CL_R$ by restriction.

\medskip

A Yang-Baxter endomorphism $\la_R$ is a unital normal endomorphism of the type III factor $\CM$ with finite-dimensional relative commutant $\CM_{R,1}\subset\CF_d^1$ \eqref{eq:relativecommutantsinF1}. We may therefore decompose it into finitely many irreducible endomorphisms of $\CM$, unique up to inner automorphisms (i.e. as sectors). In the following, we will heavily rely on results of R.~Longo, see \cite{Longo:1989,Longo:1991} for the original articles and \cite{Izumi:1991} for a summary, to obtain information about $\la_R$ and the minimal index $\Ind(\la_R)$. 

By a {\em partition of unity} in $\CM_{R,n}$ (for some $n\in\Nl$) we will mean a family $\{p_i\}_{i=1}^{d_1}\subset\CM_{R,1}$ of orthogonal projections such that $p_{i} p_{j}=\delta_{ij}p_{i}$ and $\sum_{i=1}^{d_1}p_{i}=1$. Note that since $\CM_{R,n}$ is finite-dimensional, there always exist finite partitions of unity by minimal projections.

Square brackets $\left[\la\right]$ denote the sector of $\la$, i.e. $[\la]=\{\ad u\circ\la\,:\,u\in\CU(\CM)\}$.

\begin{proposition}
    Let $R\in\CR$, $n\in\Nl$, and $\{p_{n,i}\}_{i=1}^{d_n}$ a partition of unity in $\CM_{R,n}$. Then there exist isometries $v_{n,i}\in\CM$ such that as sectors
    \begin{align}
        [\la_R^n]
        &=
        \bigoplus_{i=1}^{d_n}
        \left[\mu_{n,i}
        \right]
        ,
        \qquad
        \mu_{n,i}(\cdot)=v_{n,i}^*\la_R^n(\cdot)v_{n,i}.
    \end{align}
    The minimal index of $\la_R$ is bounded below by
    \begin{align}
     d_n^{2/n}\leq\Ind\la_R.
    \end{align}
    In case $v_{n,i}\in\CO_d$, we have $\mu_{n,i}=\la_{u_{n,i}}$ with $u_{n,i}=v_{n,i}^*\cdot{}_nR\,\varphi(v_{n,i})$.
\end{proposition}
\begin{proof}
    As $\CM$ is of type III, all projections are Murray-von Neumann equivalent to the identity, i.e. there exist isometries $v_{n,i}\in\CM$ such that $p_{n,i}=v_{n,i}v_{n,i}^*$ and $v_{n,i}^*v_{n,j}=\delta_{ij}1$. This implies that $\mu_{n,i}(x):=v_{n,i}^*\la_R^n(x)v_{n,i}$ are unital normal endomorphisms of $\CM$ -- To show that $\mu_{n,i}$ is an algebra homomorphism, note that, $x,y\in\CM$,
    \begin{align*}
        \mu_{n,i}(x)\mu_{n,i}(y)
        &=
        v_{n,i}^*\la_R^n(x)v_{n,i}v_{n,i}^*\la_R^n(y)v_{n,i}
        =
        v_{n,i}^*\la_R^n(x)p_{n,i}\la_R^n(y)v_{n,i}
        \\
        &=
        v_{n,i}^*\la_R^n(xy)p_{n,i}v_{n,i}
        =
        v_{n,i}^*\la_R^n(xy)v_{n,i}
        =
        \mu_{n,i}(xy),
    \end{align*}
    where we have used that $p_{n,i}$ commutes with $\la_R^n(\CM)$. Analogously one shows $\la_R^n(x)=\sum_iv_{n,i}\mu_{n,i}(x)v_{n,i}^*$. $x\in\CM$. This establishes $[\la_R^n]=\bigoplus_{i=1}^{d_n}\left[\mu_{n,i}\right]$.
    
    The statistical dimension $d(\la_R):=\sqrt{\Ind\la_R}$ is additive w.r.t. direct sums, multiplicative w.r.t. composition of endomorphisms, and bounded below by $1$. This implies
    \begin{align*}
        d(\la_R)=d(\la_R^n)^{1/n}=\left(\sum_{i=1}^{d_n}d(\la_{u_{n,i}})\right)^{1/n}\geq d_n^{1/n}
    \end{align*}
    and $\Ind\la_R=d(\la_R)^2\geq d_n^{2/n}$ as claimed.
    
    If $v_{n,i}\in\CO_d$, we can easily check the equality $\mu_{n,i}=\la_{u_{n,i}}$ by evaluating on generators $S_k$, $k=1,\ldots,d$.
\end{proof}

These estimates give concrete index bounds when applied to spectral decompositions.

\begin{corollary}\label{corollary:spectralindexbounds}
    Let $R\in\CR(d)$ and consider the spectra $\sigma(R)$ of $R$ and $\sigma(\phi_R(R))$ of $\phi_R(R)$. Denoting cardinality by $|\cdot|$, we have
    \begin{align}\label{eq:lowerindexbounds}
        |\sigma(R)|\leq\Ind\la_R,\qquad |\sigma(\phi_R(R))|^2\leq \Ind\la_R.
    \end{align}
\end{corollary}
\begin{proof}
    The R-matrix $R$ is a unitary in $\CM_{R,2}$ (Prop.~\ref{prop:ybe-od}~\ref{item:RinL2L2}), hence its spectral projections define a partition of unity of $d_2=|\sigma(R)|$ many projections in $\CM_{R,2}$. For the second bound, we recall that $\phi_R(R)$ is a normal element in $\CM_{R,1}$ (Thm.~\ref{theorem:phiR}~\ref{item:left-and-right}), hence its spectral projections define a partition of unity of $d_1=|\sigma(\phi_R(R))|$ many projections in~$\CM_{R,1}$.
\end{proof}

We describe the decomposition of $\la_R$ for two classes of simple R-matrices.

\begin{proposition}\label{corollary:relativecommutant}
Let $R\in\CR(d)$ be a simple R-matrix (Def.~\ref{def:diagonalR}) with projections $\{p_i\}_{i=1}^N\subset\CF_d^1$ and parameters $\{c_{ij}\}_{i,j=1}^N\subset\T$.
\begin{enumerate}
    \item\label{item:decomposelambdaR-fornormalforms} If $c_{ij}=1$ for all $i\neq j$, let $m:=|\{i\in\{1,\ldots,N\}\,:\,\tau(p_i)=1/d,\;c_{ii}=1\}|$ and $n:=|\{i\in\{1,\ldots,N\}\,:\,\tau(p_i)=1/d,\;c_{ii}\neq1\}|$. Then there exist $n$ automorphisms $\alpha_1,\ldots,\alpha_n$ and $N-n-m$ irreducible endomorphisms $\beta_1,\ldots,\beta_{N-n-m}$ such that
    \begin{align}
        \la_R\cong \alpha_1\oplus\ldots\oplus\alpha_n\oplus\beta_1\oplus\ldots\oplus\beta_{N-n-m}\oplus\underbrace{\id\oplus\ldots\oplus\id}_{m \;\text{terms}}
        .
    \end{align}
    The $\alpha_i,\beta_j$ are all mutually inequivalent and non-trivial as sectors.
    \item\label{item:decomposelambdaR-fordiagonalR} If all $p_i$ are one-dimensional (that is, if $R$ is diagonal), define the unitaries  $u_i:=\sum_{j=1}^d c_{ij}S_jS_j^*\in\CU(\CF_d^1)$, $i=1,\ldots,d$. Then there exists a unitary $u\in\CF_d^1$ such that
    \begin{align}\label{eq:diagonal-decomposition}
        \lambda_R=\la_u\circ\sum_{i=1}^dS_i\lambda_{u_i}(\cdot)S_i^*\circ\la_u^{-1}
    \end{align}
    decomposes into a sum of $d$ automorphisms. In particular, 
    \begin{align}\label{eq:indexofdiagonalR}
        [\CN:\la_R(\CN)] 
        = 
        \Ind_{E_R}(\CM) 
        =
        d^2
        .
    \end{align}
\end{enumerate}
\end{proposition}
\begin{proof}
    In both cases \ref{item:decomposelambdaR-fornormalforms} and \ref{item:decomposelambdaR-fordiagonalR}, there exists a unitary $u\in\CF_d^1$ such that $p_i=uS_iS_i^*u^*$ for all $i\in\{1,\ldots,N\}$ such that $\tau(p_i)=1/d$. Since $\la_u\circ\la_R\circ\la_u^{-1}=\la_{\la_u(R)}$, we may assume $p_i=S_iS_i^*$ for all one-dimensional projections $p_i$ without loss of generality.

    \ref{item:decomposelambdaR-fornormalforms}  For each one-dimensional projection $p_i$, we define the unitary $u_i:=c_{ii}S_iS_i^*+\sum_{k\neq i}S_kS_k^*=1+(c_{ii}-1)S_iS_i^*\in\CF_d^1$ and claim $S_i\in(\la_{u_i},\la_R)$. To prove this, we note that for arbitrary $j\in\{1,\ldots,d\}$, we have $p_\alpha S_i=\delta_{i\alpha}S_i$ and $p_iS_j=\delta_{ij}S_j$. Then we calculate from the definition of $R$ that, 
    \begin{align*}
        \la_R(S_j)S_i
        =
        RS_jS_i
        &=
        \sum_{\alpha=1}^Nc_{\alpha\alpha}p_\alpha S_j p_\alpha S_i+\sum_{\alpha\neq\beta} p_\alpha S_i\, p_\beta S_j
        =
         \begin{cases}
            c_{ii}\,S_i^2 & i=j\\
            S_iS_j & i\neq j
        \end{cases},
    \end{align*}
    which is easily seen to coincide with $S_iu_iS_j=S_i\la_{u_i}(S_j)$.
    
    Analogously, one shows $\la_R(S_j^*)S_i=S_i\la_{u_i}(S_j^*)$, which then shows that $\la_R$ contains the automorphisms $\la_{u_i}$. 
    
    Since $u_i=1$ if $c_{ii}=1$, this shows that $\la_R$ contains the identity with multiplicity $m$, and $n$ further automorphisms (the $\la_{u_i}$ with $c_{ii}\neq1$), as claimed.
    
    The statement about the remaining irreducible endomorphisms $\beta_k$ now follows from the known structure of $\CM_{R,1}$, namely $\CM_{R,1}\cong\Cl\oplus\ldots\Cl\oplus M_m$, where~$\Cl$ occurs with multiplicity $N-m$ (Prop.~\ref{prop:structureofMR1}).
    
    \ref{item:decomposelambdaR-fordiagonalR} Let $i,j\in\{1,\ldots,d\}$. It is clear that $\lambda_{u_i}$ is an automorphism, with $\lambda_{u_i}(S_j)=c_{ij}S_j$ and $\lambda_{u_i}(S_j^*)=\overline{c_{ij}}S_j^*$. Analogously to part~\ref{item:decomposelambdaR-fornormalforms}, one computes $\lambda_R(S_j)S_i=c_{ij}S_iS_j$ and $\lambda_R(S_j^*)S_i=\overline{c_{ij}}S_iS_j^*$. Hence $S_i\lambda_{u_i}(x)=\lambda_R(x)S_i$ whenever $x=S_j$ or $x=S_j^*$. This implies \eqref{eq:diagonal-decomposition}.
    
    Since each automorphism has dimension $1$, it follows that the minimal index is $\Ind(\la_R)=d^2$. Since $\Ind(\la_R)\leq\Ind_{E_R}(\la_R)=[\CN:\la_R(\CN)]\leq d^2$, \eqref{eq:indexofdiagonalR} follows.
\end{proof}

We see in particular that all simple nontrivial R-matrices are reducible. Irreducible R-matrices do exist (and are in fact likely to be the most interesting ones), but a general overview over irreducible R-matrices is currently not known. In Section~\ref{section:2dRmatrices} we will see an example.

\begin{example}
    The spectral index bounds from Cor.~\ref{corollary:spectralindexbounds} can be fairly weak, as the following example shows. If we take $R=F$, then $\sigma(R)=\{1,-1\}$ and $\phi_R(R)=d^{-1}1$. Thus in this case, the lower bounds \eqref{eq:lowerindexbounds} gives $2\leq\Ind\la_R$ and $1\leq\Ind\la_R$, respectively, to be compared with the exact result $\Ind\la_R=d^2$.
\end{example}

Regarding {\em upper} bounds on the index, we have the completely general bound $[\CN:\la_R(\CN)]\leq d^2$ on the Jones index \cite{ContiPinzari:1996} (and hence on the minimal index). In the special case that $\phi_R(R)=\tau(R)1\neq0$, then it was also shown in \cite{ContiPinzari:1996} that
\begin{align}
	[\CN:\la_R(\CN)]\leq |\tau(R)|^{-2}.	
\end{align}
More generally, if $\phi_R(R)$ is invertible\footnote{For involutive R-matrices, $\phi_R(R)$ is known to be invertible \cite{LechnerPennigWood:2019}. We currently have no proof (but also no counterexample) that this property remains true for general $R\in\CR$.} but not necessarily scalar, then
\begin{align}\label{eq:upperbound}
    [\CN:\la_R(\CN)]\leq \|\phi_R(R)^{-1}\|^4.
\end{align}
This bound is not necessarily sharper than the general bound $d^2$, but has an interesting consequence for R-matrices that we record here, following \cite[Cor.~5.5]{ContiPinzari:1996}. It states that the spectrum of a non-trivial R-matrix can not be concentrated in a disc of radius less than the universal bound $1-2^{-1/4}\approx0.159$ (this value is probably not optimal). 

\begin{corollary}\label{corollary:no-spectral-concentration}
	Let $R\in\CR$ and $\mu\in\T$ such that $\|R-\mu\|<1-2^{-1/4}$. Then $R$ is trivial.
\end{corollary}
\begin{proof}
	Passing from $R$ to $\mu^{-1} R\in\CR$ we may assume $\mu=1$ without loss of generality. 
	
	By assumption, $\|\phi_R(R)-1\|\leq\|R-1\|<1-2^{-1/4}<1$. Hence $\phi_R(R)$ is invertible, and the inverse satisfies $\|\phi_R(R)^{-1}\|\leq(1-\|R-1\|)^{-1}<2^{1/4}$. Thus \eqref{eq:upperbound} implies $[\CN:\la_R(\CN)]<2$, i.e. $[\CN:\la_R(\CN)]=1$ and $\la_R$ is an automorphism. This is only possible for trivial $R$ (Cor.~\ref{cor:no-auto}).
\end{proof}

The estimates \eqref{eq:lowerindexbounds} and \eqref{eq:upperbound} rely only on the spectrum of $R$ or $\phi_R(R)$ and fail to be sharp when multiplicities have to be taken into account. We hope to revisit this question in a future work.

\begin{remark}
     Akemann showed in \cite{Akemann:1997} that if the inclusion diagram
    \begin{equation} \label{diag:akemannssquare}
        \begin{array}{ccc}
            \CF_d^1 & \subset & \CF_d^2 \\
            \cup &\,&\cup\\
            \la_R(\CN)\cap\CF_d^1 & \subset & \la_R(\CN)\cap\CF_d^2
        \end{array}
    \end{equation}
    is a commuting square, then the index $[\CN:\la_R(\CN)]$ is an {\em integer}. 
    
    We remark here that one can show that for arbitrary $R\in\CR$, 
    \begin{align*}
        \CF_d^1\cap\la_R(\CN) = (\CF_d^1)^{\la_R}.
    \end{align*}
    With the results of the next section, it is then easy to check that if $\la_R$ is {\em ergodic} (that is, $\CN^{\la_R}=\Cl$), then \eqref{diag:akemannssquare} commutes and hence $[\CN:\la_R(\CN)]\in\Nl$. However, the square does {\em not} commute for general R-matrices. Any simple R-matrix containing a projection of dimension greater than 1 is a counterexample.
    
    Presently, it is unknown whether $[\CN:\la_R(\CN)]$ is integer\footnote{It is known, however, that $[\CL_R:\varphi(\CL_R)]$ is typically {\em not} integer \cite{Yamashita:2012,Tanner:2019}.} for any $R\in\CR$, and whether $\{[\CN:\la_R(\CN)]:R\in\CR\}=\Nl$.
\end{remark}

\subsection{Reduction of involutive R-matrices}\label{section:reduction-of-involutives}

Our considerations so far show that the decomposition of a Yang-Baxter endomorphism into irreducible endomorphisms does not preserve the Yang-Baxter equation. This can for example be seen from the decomposition of the endomorphism of a diagonal R-matrix \eqref{eq:diagonal-decomposition} which yields non-trivial automorphisms $\la_{U_i}$ -- these are not R-matrices because the only R-matrices giving automorphisms are trivial.

In the context of Yang-Baxter endomorphisms, one would therefore rather like to consider a different reduction scheme that does preserve the YBE. In this section, we present such a scheme for the subclass of involutive R-matrices (i.e. $R^2=1$).

\medskip 

To begin with, we consider R-matrices with the special property that they can be restricted to certain tensor product subspaces, as defined below.

\begin{definition}
    An R-matrix $R\in\CR(V)$ is called {\em restrictable}\footnote{In \cite{Hietarinta:1993_3_2}, such R-matrices are called ``simple solutions''. Note that R-matrices that are simple according to our definition Def.~\ref{def:diagonalR} are restrictable, but not all restrictable R-matrices are simple.} 
    if there exists a non-trivial subspace $W\subset V$ such that $R$ leaves the two subspaces $W\ot W$ and $W^\perp\ot W^\perp$  of $V\ot V$ (with $W^\perp$ the orthogonal complement of $W\subset V$) invariant. 
\end{definition}

Clearly $R$ is restrictable if and only if there exists a non-trivial orthogonal projection $p\in\CF_d^1$ such that 
\begin{align}
 [R,p\ot p]=0,\qquad [R,p^\perp\ot p^\perp]=0,\qquad [R,p\ot p^\perp+p^\perp\ot p]=0.
\end{align}
(Actually the third equation is a consequence of the first two.)

It is clear that in this situation, the restrictions of $R$ to $W\ot W$ and $W^\perp\ot W^\perp$ are again R-matrices, with base spaces $W$ and $W^\perp$, respectively. 

We now look at the special case of involutive R-matrices.

\begin{lemma}\label{lemma:restrictability}
    Let $R\in\CR_0(V)$ be involutive and reducible. Then $R$ is restrictable. More precisely, there exists a nontrivial subspace $W\subset V$ (with orthogonal complement $W^\perp$) such that according to the orthogonal decomposition
 \begin{align}
  V\ot V = (W\ot W)\oplus(W\ot W^\perp)\oplus(W^\perp\ot W)\oplus(W^\perp\ot W^\perp),
 \end{align}
 $R$ takes the form
 \begin{align}\label{eq:Rsplit}
    R = 
    \left(
	\begin{array}{cccc}
		S\\
		&&U^{-1}\\
		&U\\
		&&&T\\
    \end{array}
    \right)
 \end{align}
 with a unitary $U:W\ot W^\perp\to W^\perp\ot W$ and involutive R-matrices $S\in\CR_0(W)$, $T\in\CR_0(W^\perp)$.  
\end{lemma}
\begin{proof}
    Since $\la_R$ is reducible, there exists a non-trivial projection $p\in\CM_{R,1}\subset\CF_d^1$, and we define $W:=pV$. As an element of $(\la_R,\la_R)$, the projection $p$ satisfies $R^*(p\ot1)R=1\ot p$. Furthermore, $R$ is involutive and hence selfadjoint. This implies that we also have $R(1\ot p)R=p\ot 1$ and therefore 
 \begin{align}
  R(p\ot p)R=R(p\ot1)RR(1\ot p)=(1\ot p)(p\ot 1)=p\ot p.
 \end{align}
 We conclude that $R$ leaves the subspaces $W\ot W$ and $W^\perp\ot W^\perp$ invariant and defines the two R-matrices $S$ and $T$ by restriction to these subspaces.
 
 Moreover, we have
 \begin{align}
  R(p\ot p^\perp)R=R(p\ot 1-p\ot p)R=1\ot p-p\ot p=p^\perp\ot p,
 \end{align}
 and analogously $R(p^\perp \ot p)R=p\ot p^\perp$. This shows that $R$ also restricts to unitary maps $U:W\ot W^\perp\to W^\perp\ot W$ and $U':W^\perp\ot W\to W\ot W^\perp$. Since $R^2=1$, we find $U'=U^{-1}$.
\end{proof}

This observation sheds new light onto the decomposition of involutive R-matrices: Whenever an involutive $R$ is reducible, we can split it into two smaller R-matrices $S,T$ and an ``off-diagonal component'' $U$. Since the restrictions $S$ and $T$ are still involutive, this process can be iterated until, after finitely many steps, the restricted R-matrices are irreducible. In this sense involutive R-matrices are completely reducible. 

\begin{remark}\leavevmode
    \begin{enumerate}
        \item We conjecture that in the involutive case, $\la_R$ is irreducible if and only if $R$ is a multiple of the identity. This is certainly true in dimension $d=2$, but we currently have no proof in general dimension.
        \item Equation \eqref{eq:Rsplit} can also be read as a way of constructing R-matrices of larger dimension out of two smaller ones. If the operator $U$ in \eqref{eq:Rsplit} coincides with the flip $F$, then the right hand side of \eqref{eq:Rsplit} equals $S\boxplus T$, which satisfies the YBE if and only if $S$ and $T$ do. For more general $U$, certain commutation relations between $U$ and $S,T$ have to be satisfied in order to ensure the YBE for $R$ \cite{MajidMarkl:1996}.
    \end{enumerate}
\end{remark}

It is instructive to point out how this reduction scheme leads to a normal form for involutive R-matrices up to the equivalence relation $\sim$. Recall that R-matrices of normal form were defined in \eqref{eq:def-normal-form} as simple R-matrices with parameters $c_{ii}\in\{\pm1\}$ for all $i$ and $c_{ij}=1$ for all $i\neq j$.

\begin{proposition}
    Let $R\in\CR(d)$ be involutive.
    \begin{enumerate}
        \item\label{item:boxreduction} In the situation of Lemma~\ref{lemma:restrictability}, we have $R\sim S\boxplus T$.
        \item\label{item:irred-1F} If $R$ is irreducible, then $R=\pm1$ or $R\sim\pm F$.
        \item\label{item:normalform} There exists an R-matrix $N$ of normal form such that $R\sim N$.
    \end{enumerate}
\end{proposition}
\begin{proof}
    \ref{item:boxreduction} We have to show that $R$ and $\tilde R:=S\boxplus T$ have the same character. It is sufficient to show that for any $n\in\N$, we have $\tau(R\varphi(R)\cdots\varphi^n(R))=\tau(\tilde R\varphi(\tilde R)\cdots\varphi^n(\tilde R))$ because both R-matrices are involutive and extremal characters of the infinite symmetric group are fixed by their values on cycles. For the case $U=F$, it was shown in \cite[Prop.~4.4]{LechnerPennigWood:2019} that
    \begin{align*}
    (\dim\tilde R)^{n+1}\tau(\tilde R\cdots\varphi^n(\tilde R))
    =
    d_S^{n+1}\cdot
    \tau(S\cdots\varphi^n(S))
    +
    d_T^{n+1}\cdot
    \tau(T\cdots\varphi^n(T)),
    \end{align*}
    where $d_S=\dim S$, $d_T=\dim T$. This proof carries over without changes to the case of a general unitary $U:W\ot W^\perp\to W^\perp\ot W$, leading to the conclusion that $\tau(R\varphi(R)\cdots\varphi^n(R))=\tau(\tilde R\varphi(\tilde R)\cdots\varphi^n(\tilde R))$ for any $n\in\N$.
    
    \ref{item:irred-1F} If $R$ is irreducible, we have in particular $\CL_{R,1}=\Cl$ and therefore $\phi_R(R)\in\Cl$. This implies the claim, as shown in \cite{LechnerPennigWood:2019}.
    
    \ref{item:normalform} Applying the reduction scheme to $R$ repeatedly yields
    \begin{align*}
        R \sim R^1 \boxplus \ldots \boxplus R^n,
    \end{align*}
    where the $R^i\in\CR(d_i)$ are involutive irreducible R-matrices -- the superscript is just a label, not a power -- and the off-diagonal terms $U$ from Lemma~\ref{lemma:restrictability} have been removed up to equivalence $\sim$ with the help of part~\ref{item:boxreduction}. 
    
    Now, in view of part~\ref{item:irred-1F}, each $R^i$ is either $\pm1_{d_i}$ (the subscript indicates the dimension, i.e. $\pm1_{d_i}\in\CR(d_i)$) or equivalent to $\pm F_{d_i}$ (where again, the subscript indicates the dimension). Without loss of generality, assume the first $m$ R-matrices are trivial (for some $0\leq m\leq n$) and the remaining $n-m$ R-matrices are flips, i.e. there are signs $\eps_1,\ldots,\eps_n\in\{\pm1\}$ such that $R^1=\eps_11_{d_1},\ldots,R^m=\eps_m1_{d_m}$ and $R^{m+1}=\eps_{m+1}F_{d_{m+1}},\ldots,R^n=\eps_n F_{d_n}$.
    
    A look at \eqref{eq:simpleR} shows that the flip in dimension $d_i$ is simple, in fact it can be written as $F_{d_i}=1_1\boxplus\ldots\boxplus1_1$ ($d_i$ terms). 
    Hence we arrive at 
    \begin{align*}
        R \sim \eps_11_{d_1}\boxplus\ldots\boxplus\eps_m1_{d_m}\boxplus
        \eps_{m+1}
        (\underbrace{1_1\boxplus\ldots\boxplus1_1}_{d_{m+1} \text{ terms}})\boxplus\ldots\boxplus
        \eps_n(\underbrace{1_1\boxplus\ldots\boxplus1_1}_{d_n \text{ terms}})=:N,
    \end{align*}
    which shows that $R\sim N$ with $N$ of the claimed simple form.
\end{proof}

The normal form result was already known from \cite{LechnerPennigWood:2019}, but we have now a new perspective on it from the point of view of Yang-Baxter endomorphisms. This analysis identifies two greatly simplifying features of the involutive case: On the one hand, every involutive $R$ is completely reducible in the sense explained above, and on the other hand, there exist only very few irreducible involutive R-matrices.

For general R-matrices, neither a reduction scheme nor a classification of irreducible elements, are currently known\footnote{See Section~\ref{section:2dRmatrices} for an example of a non-trivial irreducible R-matrix.}. We hope to come back to this question in a future work.

\section{Ergodicity and Fixed Points}\label{section:ergodicity}

Fixed point subalgebras of automorphisms and endomorphisms of $\CO_d$ have not been investigated systematically but in few cases. For instance $\CO_d^\varphi = {\mathbb C}$, but  there exists an order two quasi-free automorphism $\lambda_f$ of $\CO_2$, $f=S_1 S_2^* + S_2 S_1^*$, such that ${\CO_2}^{\lambda_f} \simeq \CO_2$ \cite{ChoiLatremoliere:2012}. More interestingly, $\CO_2^{\lambda_{-1}} \simeq \CO_4$, as it is the $C^*$-subalgebra of $\CO_2$ generated by $S_i S_j$, $1 \leq i,j \leq 2$. This example is the fixed point algebra of the R-matrix $R=-1 \in \CR(2)$. 

In this section, we discuss fixed point algebras of Yang-Baxter endomorphisms~$\lambda_R$ at the level of the $C^*$-algebras $\CO_d$, $\CF_d$ and the von Neumann algebras $\CM$, $\CN$. What is special in the Yang-Baxter context is that fixed point algebras of $\la_R$ are closely related to the relative commutants $\CL_R\subset\CN$, $\CL_R\subset\CM$, as we demonstrate now.

\begin{proposition}\label{proposition:generalfixedpointresults}
    Let $R\in\CR(d)$. 
    \begin{enumerate}
        \item\label{item:FP-M} $\CM^{\lambda_R}\subset\bigcap\limits_{n\geq1}\lambda_R^n(\CM)\subset\CL_R'\cap\CM$.
        \item\label{item:FP-N} $\CN^{\lambda_R}=\bigcap\limits_{n\geq1}\lambda_R^n(\CN)=\CL_R'\cap\CN$.
        \item\label{item:FP-stability} Let $i,j\in\{1,\ldots,d\}$. Then $S_i^*\CM^{\lambda_R}S_j\subset \CM^{\lambda_R}$ and $S_i^*\CN^{\lambda_R}S_j\subset \CN^{\lambda_R}$.
    \end{enumerate}    
\end{proposition}
\begin{proof}
    \ref{item:FP-M} The first inclusion is trivial. For the second one, let $x\in\bigcap_{n\geq1}\lambda_R^n(\CM)$ and $m\in\Nl_0$. Then $x=\lambda_R^{m+2}(y)$ for some $y\in\CM$, and taking into account that $R\in\CM_{R,2}=(\lambda_R^2,\lambda_R^2)$, we find
    \begin{align*}
        \varphi^m(R)x
        =
        \lambda_R^m(R)\lambda_R^{m+2}(y)
        =
        \lambda_R^m(R\lambda_R^2(y))
        =
        \lambda_R^m(\lambda_R^2(y)R)
        =
        x\varphi^m(R).
    \end{align*}
    Since $m$ was arbitrary, this implies $x\in\CL_R'\cap\CM$.
    
    \ref{item:FP-N} Exactly as in part \ref{item:FP-M} we have the two ``$\subset$'' inclusions, and it remains to show $\CL_R'\cap\CN\subset\CN^{\la_R}$. Let $x\in\CL_R'\cap\CN$, i.e. $[x,\varphi^n(R)]=0$ for all $n\in\N_0$. Then 
    \begin{align*}
        \lambda_R(x)=\lim_{n\to\infty}R\cdots\varphi^n(R)x\varphi^n(R)^*\cdots R^*=x,
    \end{align*}
    i.e. $x\in\CN^{\lambda_R}$.
    
    \ref{item:FP-stability} Let $x\in\CM^{\lambda_R}$. Taking into account that $x$ commutes with $R$ by part~\ref{item:FP-M}, we have 
    \begin{align*}
        \lambda_R(S_i^*xS_j)=
        S_i^*R^*\lambda_R(x)RS_j
        =
        S_i^*R^*xRS_j
        =S_i^*xS_j.
    \end{align*}
\end{proof}

\begin{remark}\leavevmode
    \begin{enumerate}
		\item In standard terminology, an endomorphism $\la$ of a von Neumann algebra $\CN$ is called {\em ergodic} if $\CN^{\la}=\Cl$ and a {\em shift} if $\bigcap_{n\geq1}\la^n(\CN)=\Cl$. We have thus shown that that ${\la_R}|_{\CN}$ is ergodic if and only if~${\la_R}|_{\CN}$ is a shift. Furthermore, the fixed point algebra coincides with the relative commutant of $\CL_R\subset\CN$. Hence ${\la_R}|_\CN$ is ergodic if and only if $\CL_R\subset\CN$ is irreducible. 
		\item We will later discuss an example where $\CN^{\la_R}$ is infinite-dimensional, i.e. in particular $\CL_R\subset\CN$ has infinite index.
        \item All statements of this proposition hold without changes on the level of the $C^*$-algebras, i.e. $\CO_d^{\lambda_R}\subset\bigcap\limits_{n\geq1}\lambda_R^n(\CO_d)\subset\CB_R'\cap\CO_d$ and $\CF_d^{\lambda_R}=\bigcap\limits_{n\geq1}\lambda_R^n(\CF_d)=\CB_R'\cap\CF_d$.
    \end{enumerate}
\end{remark}

It is currently not clear if one has equalities in Prop.~\ref{proposition:generalfixedpointresults}~\ref{item:FP-M}, or if $\CM^{\la_R}\subset\CN^{\la_R}$ for all non-trivial 
~$R$. We next show that at least ergodicity of $\la_R$ can be decided on the level of the type II factor $\CN$.

For this and following results, we will make use of a (von Neumann version of) family of linear maps $E^n:\CM\to\CN$, $n\in\Z$, introduced in \cite{Cuntz:1977}, namely ($n\geq0$)
\begin{align}\label{eq:Fn}
 E^n(x)=\int_\T\alpha_z(x{S_1^*}^n)
 ,\qquad
 E^{-n}(x)=\int_\T\alpha_z(S_1^nx)
 ,
\end{align}
where $\alpha_z=\lambda_{z\cdot1}$ are the gauge automorphisms, integration is over the circle $z\in\T$ w.r.t. $\frac{dz}{2\pi iz}$, and the choice of $S_1$ as a reference generator is by convention. We also introduce the closely related spectral components $x^{(n)}\in\CM^{(n)}$ of $x$ as
\begin{align}\label{eq:spectral-components}
 x^{(n)}
 :=
 \int\alpha_z(x)z^{-n}
 =
  \begin{cases}
    E^n(x)S_1^n & n\geq0\\
    {S_1^*}^{-n}E^n(x) & n<0 
  \end{cases}
  .
\end{align}
Recall that $x=0$ is equivalent to $x^{(n)}=0$ for all $n\in\Zl$ \cite{Takesaki:1973,Haagerup:1989}. Moreover, we clearly have $(x^*)^{(n)} = (x^{(-n)})^*$ for all $x\in\CM$ and all $n\in\Z$.

For any unitary $U\in\CU(\CF_d)$, the endomorphism $\la_U$ commutes with the gauge action, so that the fixed point algebra $\CM^{\lambda_U}$ is globally ${\mathbb T}$-invariant and for any $x\in\CM^{\lambda_U}$, also all its spectral components $x^{(n)}$ are fixed points of $\la_U$. This applies in particular to R-matrices $R\in\CU(\CF_d^2)$.

\medskip

\begin{proposition}\label{prop:ergodicityFvsO}
	Let $U\in\CU(\CF_d)$. If $\CF_d^{\lambda_U}=\Cl$ then $\CO_d^{\lambda_U} = \Cl$, and if $\CN^{\lambda_U}= \Cl$ then $\CM^{\lambda_U} = \Cl$.
\end{proposition}
\begin{proof}
	Let $x \in\CO_d^{\lambda_U}$. If it was nontrivial, it would not lie in $\CF_d$ and then it would have a nonzero spectral component. Without loss of generality, we may then assume that $x^{(n)} \neq 0$ for some $n >0$, and as remarked above, $x^{(n)}\in\CO_d^{\la_U}$. Now, both $x^{(n)} (x^{(n)})^*$ and $(x^{(n)})^* x^{(n)}$ are fixed points in $\CF_d$ and thus positive scalars, say $\mu$ and $\nu$. It follows immediately that $\nu$ must be equal to $\mu$ and thus $x^{(n)}$ is a multiple of a unitary. However, it is easy to see that this is in conflict with the KMS condition
	(recall that $\lambda_{d^{-it}1}$ is the modular group w.r.t. the state $\omega = \tau \circ E^0$).
	
	The proof for the von Neumann algebras $\CM$, $\CN$ is identical.
\end{proof}

Prop.~\ref{prop:ergodicityFvsO} implies that $\la_R$ is ergodic if and only if ${\la_R}|_{\CN}$ is ergodic. In this case, we will simply say that $R\in\CR$ is ergodic. 

\begin{remark}\label{remark:stability-of-ergodicity}
    \leavevmode
   \begin{enumerate}
        \item It is clear that the equivalence relations $R\Msim S$ and $R\Nsim S$ (Def.~\ref{def:equivalence}) provide automorphisms of $\CM$ and $\CN$ that identify the fixed point algebras of $\la_R$ and~$\la_S$. In particular, the ``type 1'' and ``type 2'' cases of $\sim$ equivalences (see p.~\ref{page:types}) preserve ergodicity. 
        
        \item $R$ is ergodic if and only if $R^*$ is ergodic because
        \begin{align}
            \CN^{\la_{R^*}}=\CL_{R^*}'\cap\CN=\CL_{R}'\cap\CN=\CN^{\la_R}.
        \end{align}
        
        \item Clearly $\CO_d^{\la_R}$ is stable under any endomorphism $\la_u$ that commutes with $\la_R$. For example, if the unitary $u$ is a fixed point, then $\lambda_R \lambda_u = \lambda_{uR}$, and this coincides with $\lambda_u \lambda_R$ if and only if $\varphi(u)$ commutes with $R$. However, in general $\CO_d^{\lambda_R}$ is not $\varphi$-invariant.
    \end{enumerate}
\end{remark}

We now turn to an explicit characterization of ergodicity. Let $H_R:\CN\to\CN^{\la_R}$ denote the unique $\tau$-preserving conditional expectation onto the fixed point algebra. As $\la_R$ preserves $\tau$, the ergodic theorem allows us to write $H_R$ as 
\begin{align}\label{eq:ergodic-mean}
	H_R(x)
	&=
	\slim_{n\to\infty}\frac{1}{n}\sum_{k=0}^{n-1}\la_R^k(x)
	,\qquad
	x\in\CN.
\end{align}
Also recall that $E_n$ denotes the $\tau$-preserving conditional expectation $\CN\to\CF_d^n$, which acts by tracing out all tensor factors except the first $n$ (in particular, $E_0=\tau$).
    
\begin{theorem}\label{thm:ergodicity}
    Let $R\in\CR(d)$. The following are equivalent:
    \begin{enumerate}
		\item\label{item:ergodicity-condition} $E_1(RxR^*)=\tau(x)$ for all $x\in\CF_d^1$.
		\item\label{item:ergodicity-condition+} $E_n(\varphi^{n-1}(R)x\varphi^{n-1}(R^*))=E_{n-1}(x)$ for all $n\in\Nl,x\in\CF_d^n$.
		\item\label{item:1ergodicity} $H_R(x)=\tau(x)$ for all $x\in\CF_d^1$.
		\item\label{item:ergodicity} $R$ is ergodic.
    \end{enumerate}
    If $R$ is ergodic, then so are all its cabling powers $R^{(n)}$, $n\in\Nl$.
\end{theorem}

We will refer to the condition in part \ref{item:ergodicity-condition} as ``the ergodicity condition'' in the following. 

\pagebreak

\begin{remark}\leavevmode
 \begin{enumerate}
  \item In matrix notation, the ergodicity condition reads as follows: Let $(e_k)_{k=1}^d$ be the standard basis of $\Cl^d$, and let $R^{ij}_{kl}:=\langle e_i\ot e_j,R(e_k\ot e_l)\rangle$. Then the ergodicity condition is equivalent to
	\begin{align}\label{eq:ergodicity-matrixform}
		\sum_{n,m=1}^d R^{im}_{kn}  \overline{R_{ln}^{jm}}
		=
		\delta^i_j\,\delta^k_l\qquad i,j,k,l\in\{1,\ldots,d\},
	\end{align}
	as can be seen by choosing $x$ as the matrix unit ${\tt e}_{kl}\in M_d$. In the special case of involutive R-matrices equivalent to the flip, Wassermann have a proof of an analogue of Thm.~\ref{thm:ergodicity} already in \cite{Wassermann:1987}, also based on the condition \eqref{eq:ergodicity-matrixform}.
	
	\item The ergodicity condition is best understood in graphical notation. Noting that $E_1$ acts as a normalized right partial trace on $\CF_d^2$, we have the following graphical representation:
	\begin{figure}[h]
		\includegraphics[width=45mm]{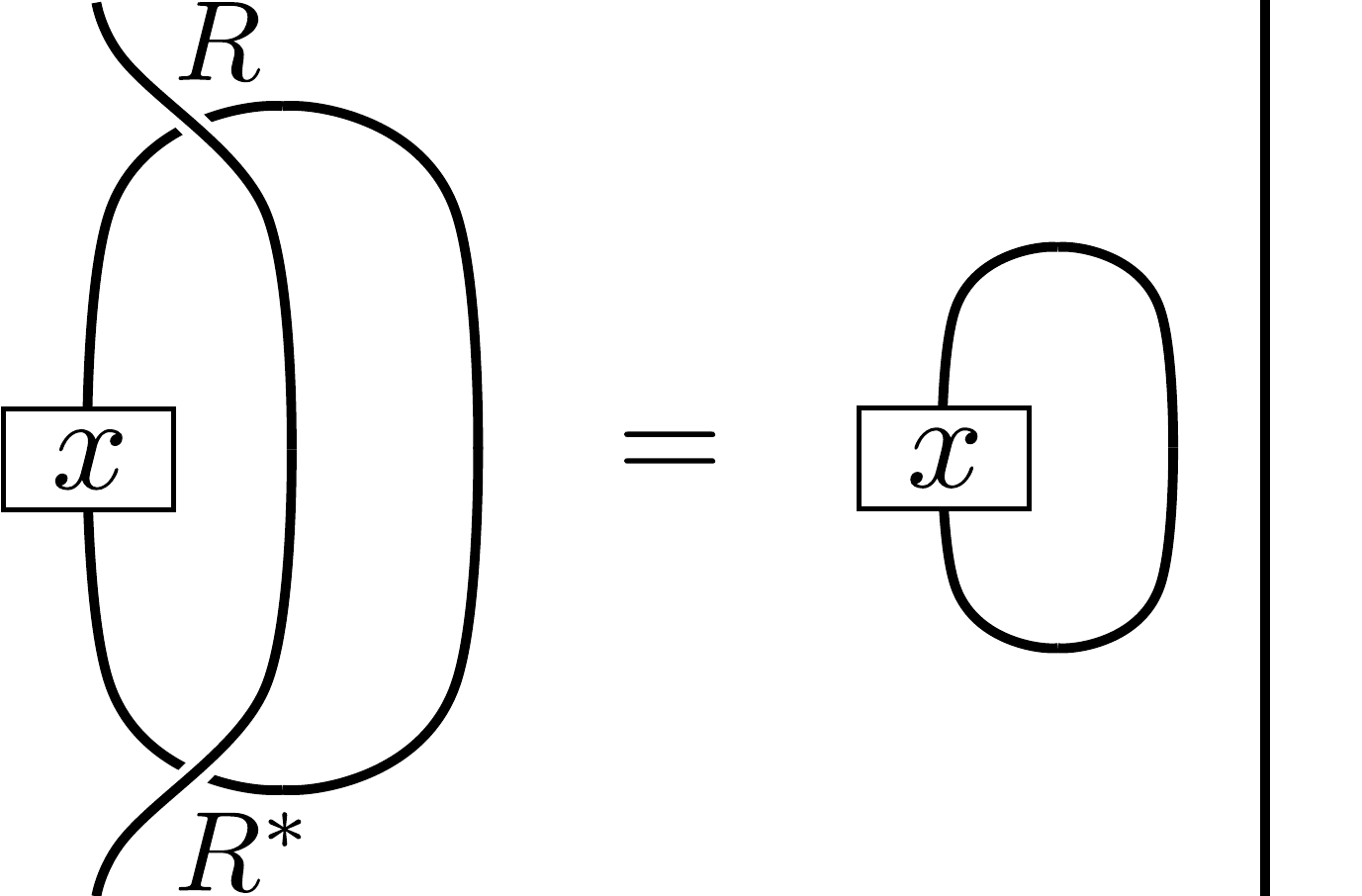}
		\caption{\small The ergodicity condition in graphical notation. Note that this is trivially satisfied for $R=F$, and trivially violated for $R=1$.}
	\end{figure}
	
	We also note the graphical representation of the (equivalent) condition in part~\ref{item:ergodicity-condition+}: Since $E_n$ acts as the normalized partial trace on the rightmost tensor factor of $\CF_d^{n+1}$, it is apparent that condition \ref{item:ergodicity-condition+} reads in graphical notation
	\begin{figure}[h]
		\includegraphics[width=70mm]{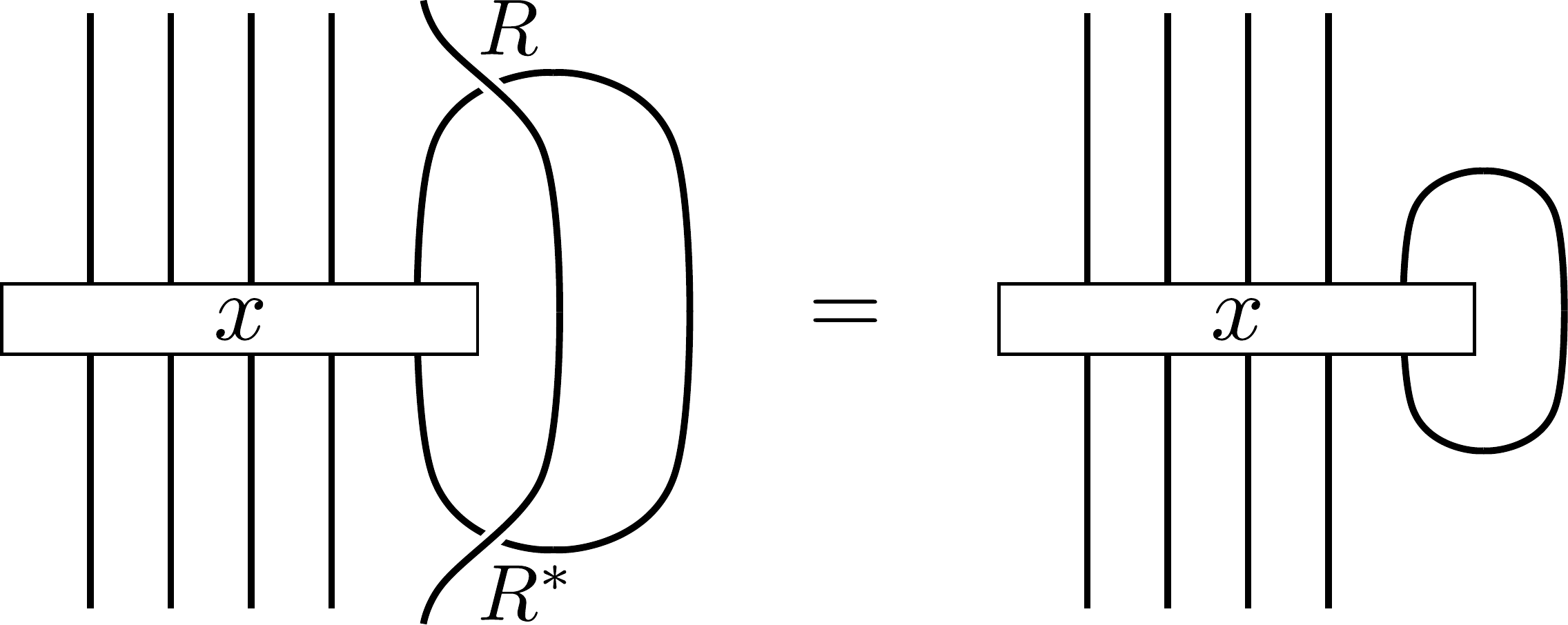}
	\end{figure}
	
	\item The ergodicity condition also appears in \cite{ContiPinzari:1996}, where it was shown to imply that the left inverse $\phi_R$ is localized in the sense that for any $n\in\Nl$ there exists a $k\in\Nl$ such that $\phi_R(\CF_d^n)\subset\CF_d^k$.
 \end{enumerate}

\end{remark}

\begin{proof}
	\ref{item:ergodicity-condition}$\implies$\ref{item:ergodicity-condition+} We give a proof by induction in $n$, the case $n=1$ being equivalent to~\ref{item:ergodicity-condition}. For the induction step, note that the definition of $E_n$ implies $S_i^*E_n(\cdot)S_j=E_{n-1}(S_i^*\,\cdot\,S_j)$ for any $i,j$. Thus we have, $i,j\in\{1,\ldots,d\}$, $x\in\CF_d^{n+1}$,
	\begin{align*}
		S_i^*E_{n+1}(\varphi^n(R)x\varphi^n(R^*))S_j
		&=
		E_{n}(S_i^*\varphi^n(R)x\varphi^n(R^*)S_j)
		\\
		&=
		E_{n}(\varphi^{n-1}(R)S_i^*xS_j\varphi^{n-1}(R^*))
		.
	\end{align*}
	As $S_i^*xS_j\in\CF_d^n$, this simplifies by induction assumption to $E_{n-1}(S_i^*xS_j)=S_i^*E_n(x)S_j$. Since $i,j$ were arbitrary, this finishes the proof.
	
	\ref{item:ergodicity-condition+}$\implies$\ref{item:1ergodicity} Let $x\in\CF_d^1$, $n\in\Nl$ and $y\in\CF_d^n$. Noting that $\varphi^{k-1}(R)$ commutes with $y$ for $k-1\geq n$, we calculate
    \begin{align}
        \nonumber
        \tau(yH_R(x))
        &=
        \lim_{m\to\infty}\frac{1}{m}\sum_{k=0}^{m-1}\tau(y\la^k_R(x))
        \\
        \nonumber
        &=
        \lim_{m\to\infty}\frac{1}{m}\sum_{k=0}^{m-1}\tau(y\varphi^{k-1}(R)\cdots RxR^*\cdots\varphi^{k-1}(R^*))
        \\
        \nonumber
        &=
        \lim_{m\to\infty}\frac{1}{m}\bigg\{
        \sum_{k=0}^{n}\tau(y\,{}_kRx({}_kR)^*)
        \\
        &\qquad\qquad\!+
        \sum_{k=n+1}^{m-1}\tau(y\varphi^{n-1}(R)\cdots RxR^*\cdots\varphi^{n-1}(R^*))
        \bigg\}
        \nonumber
        \\
        &=
        \tau(y\varphi^{n-1}(R)\cdots RxR^*\cdots\varphi^{n-1}(R^*))
        .
        \label{eq:ergodic-calculcation}
    \end{align}
    We now insert $E_n$ into the trace and use \ref{item:ergodicity-condition+} iteratively to arrive at 
    \begin{align*}
     \tau(yH_R(x))
     &=
     \tau(yE_n(\varphi^{n-1}(R)\cdots RxR^*\cdots\varphi^{n-1}(R)^*))
     \\
     &=
     \tau(yE_{n-1}(\varphi^{n-2}(R)\cdots RxR^*\cdots\varphi^{n-2}(R)^*))
     \\
     &=
     \tau(yE_{0}(x))
     \\
     &=
     \tau(y)\tau(x)
     .
    \end{align*}
	As $n$ was arbitrary and the trace is faithful, this implies $H_R(x)=\tau(x)$, i.e. we have shown \ref{item:1ergodicity}.
	
	\medskip
	
	\ref{item:1ergodicity}$\implies$\ref{item:ergodicity} To amplify \ref{item:1ergodicity} to ergodicity, we will use the cabling maps $c_n$ and cabling powers $R^{(n)}$, $n\in\Nl$. The first step is to realize that if $R$ satisfies the ergodicity condition, then so does $R^{(n)}$, i.e.
	\begin{align*}
		E_{d^n,1}(R^{(n)}c_n(x)(R^{(n)})^*)=\tau(x)
		,\qquad x\in\CF_d^n.
	\end{align*}
    Applying $c_n^{-1}$, this condition is seen to be equivalent to
    \begin{align*}
	 E_n({}_nR_n\cdot x\cdot{}_n(R^*)_n)=\tau(x),\qquad x\in\CF_d^n,
	\end{align*}
    which can be proven by induction in $n$ with the help of the ergodicity condition for $R$, expressed as in \ref{item:ergodicity-condition+} (and is obvious in graphical notation).

	Let $n\in\Nl$ and $x\in\CF_d^n$. Then $c_n(x)\in\CF_{d^n}^1$, and since $R^{(n)}$ satisfies \ref{item:ergodicity-condition} and thus also~\ref{item:1ergodicity}, we have $H_{R^{(n)}}(c_n(x))=\tau(c_n(x))=\tau(x)$ and  therefore
    \begin{align}\label{eq:muexpt}
     \tau(x) = (c_n^{-1}\circ H_{R^{(n)}}\circ c_n)(x), \qquad x\in\CF_d^n.
    \end{align}
    We now recall that $c_n^{-1}\circ\la_{R^{(n)}}\circ c_n=\la_R^n$ as endomorphisms of $\CN_d$ \eqref{eq:cabling-and-powers}. Expressing $H_{R^{(n)}}$ as an ergodic mean as in \eqref{eq:ergodic-mean}, we then see that $H_{R,n}:=c_n^{-1}\circ H_{R^{(n)}}\circ c_n$ is the $\tau$-preserving conditional expectation from $\CN_d$ onto its fixed point subalgebra $\CN_d^{\la_R^n}$.
    
	Eqn.~\eqref{eq:muexpt} states that $H_{R,n}$ acts as the trace on $\CF_d^n$. As clearly $\CN_d^{\la_R}\subset\CN_d^{\la_R^n}$, also the conditional expectation $H_R$ acts as the trace on $\CF_d^n$. In other words, $\tau(yH_R(x))=\tau(y)\tau(x)$ for all $y\in\CN_d$ and all $x$ in the algebraic infinite tensor product $\bigcup_n\CF_d^n$. By continuity, this extends to $\tau(yH_R(x))=\tau(y)\tau(x)$ for all $x,y\in\CN_d$, which is equivalent to ergodicity, $H_R=\tau$, by the faithfulness of $\tau$.
	
	\ref{item:ergodicity}$\implies$\ref{item:ergodicity-condition} 
	Let $x\in\CF_d^1$. According to the calculation \eqref{eq:ergodic-calculcation} in the proof of \ref{item:ergodicity-condition+}$\implies$\ref{item:1ergodicity}, specialized to $n=1$, we have for all $y\in\CF_d^1$
    \begin{align*}
		\tau(yH_R(x))
        &=
        \tau(yRxR^*)
        =
        \tau(yE_1(RxR^*)).
    \end{align*}
    If $\la_R$ is ergodic, we have $H_R(x)=\tau(x)$. As $E_1(RxR^*)$ is an element of $\CF_d^1$, and $y\in\CF_d^1$ was arbitrary, we see that $E_1(RxR^*)=\tau(x)$, i.e. \ref{item:ergodicity-condition} holds.       
\end{proof}

As an application of Thm.~\ref{thm:ergodicity}, we show that diagonal R-matrices (Def.~\ref{def:diagonalR}) are ergodic. 

\begin{corollary}\label{cor:diagonalsareergodic}
    Diagonal R-matrices are ergodic.
\end{corollary}
\begin{proof}
	A diagonal R-matrix is of the form $R=\la_u(S)$ with $u\in\CU(\CD_d^1)$ and $S\in\CR(d)$ of the form $S^{ij}_{kl}=c_{lk}\delta^i_l\delta^j_k$, $i,j,k,l\in\{1,\ldots,d\}$ with parameters $c_{lk}\in\T$. It is a straightforward calculation to verify the ergodicity condition \eqref{eq:ergodicity-matrixform} for $S$. Since $R\Msim S$ (type 1), it follows that $R$ is ergodic as well.
\end{proof}

\begin{remark}
    Any non-trivial fixed point $x=\la_R(x)=RxR^*\in\CF_d^1$ satisfies $E_1(RxR^*)=x$ and therefore violates the ergodicity condition. Conversely, if some $x\in\CF_d^1$ violates the ergodicity condition, then the argument in the proof \ref{item:ergodicity}$\implies$\ref{item:ergodicity-condition} of Thm.~\ref{thm:ergodicity} shows that $H_R(x)\neq\tau(x)$. That is, we have a non-trivial fixed point $H_R(x)\in\CN^{\la_R}$ in this case. However, typically $H_R(x)$ will not lie in $\CF_d^1$ or even $\CF_d$, but only in its weak closure $\CN$.
    
    One might therefore expect that the condition that $\la_R$ admits no non-trivial fixed points in $\CF_d^1$, namely
    \begin{align}\label{eq:F1ergodicity}
        \Cl\stackrel{!}{=}(\CF_d^1)^{\la_R}=\{x\in\CF_d^1\,:\,RxR^*=x\},
    \end{align}
    is strictly weaker than the ergodicity condition for general $R$. We will prove this later by an example.
\end{remark}

In order to compare the fixed point algebras on the $C^*$- and von Neumann level, we add another result,  which shows that the fixed point algebra on the $C^*$-level is, in a sense, not too big when $R$ is not a scalar. Recall that if a unital $C^*$-algebra $A$ is simple and purely infinite then for every nonzero $x\in A$ there exist $y,z \in A$ such that $yxz=1$ \cite[Thm.~V.5.5]{Davidson:1996_2}.

\begin{proposition}\label{prop:simplepurelyinfinite}
	Let $R \in \CR(d)$. If $\CO_d^{\lambda_R}$ is simple and purely infinite then $R= \mu 1$, where $\mu \in {\mathbb T}$ is an $n$-th root of unity for some positive integer $n$. 
\end{proposition}
\begin{proof}
	Suppose that the fixed point algebra is simple purely infinite. Then it is not contained in $\CF_d$, and thus there exists some $x \in \CO_d^{\lambda_R}$ with $x^{(n)} \neq 0$ for some $n>0$. Now, from the equality $\lambda_R(x^{(n)}) = x^{(n)}$, taking into account the fact that $R$ is unitary and $x^{(n)}$ commutes with $\CB_R$, we get
	\begin{align*}
		\|R \varphi(R) \cdots &\varphi^{k+n-1}(R) x^{(n)} \varphi^{k-1}(R)^* \cdots \varphi(R)^* R^* - x^{(n)} \| 
		\\
		&= 
		\|\varphi^k(R)\cdots\varphi^{k+n-1}(R)x^{(n)}
		-
		\varphi^{k-1}(R^*)\cdots R^*x^{(n)}R\cdots\varphi^{k-1}(R)\|
		\\
		&= 
		\|\varphi^k(R) \cdots \varphi^{k+n-1}(R)  x^{(n)} - x^{(n)} \| 
		\to 0
	\end{align*}
	when $k \to \infty$. Pick $y,z \in\CO_d^{\lambda_R}$ such that $y x^{(n)} z = 1$. Then,
	\begin{align*}
		\| \varphi^k \big(R \cdots \varphi^{n-1}(R) \big) - 1 \| 
		& = \| \varphi^k(R) \cdots \varphi^{k+n-1}(R) - 1 \| \\
		& = \| y (\varphi^k(R) \cdots \varphi^{k+n-1}(R)  x^{(n)} - x^{(n)})z \| \\ 
		& \leq \| \varphi^k(R) \cdots \varphi^{k+n-1}(R)  x^{(n)} - x^{(n)} \| \; \|y\| \; \| z\| 
		\longrightarrow 0 
	\end{align*}
	as $k \to \infty$. Since $\varphi$ is unital and isometric, we get $R \cdots \varphi^{n-1}(R) =1$.
	However, $R^* \in \CR(d)$, implying that $\la_{R^*}$ is not surjective and $\lambda^n_{R^*} = \lambda_{\varphi^{n-1}(R^*) \cdots \varphi(R^*)R^*}$ is not the identity, unless $R = \mu 1$ with $\mu^n = 1$.
\end{proof}

Conversely, if $\mu \in {\mathbb T}$ is a primitive $n$-th root of 1 then it is not difficult to see that~$\CO_d^{\lambda_{\mu 1}}$ is isomorphic to $\CO_{d^n}$, while
if $\mu \in {\mathbb T}$ has infinite order one has $\CO_d^{\lambda_{\mu 1}} = \CF_d$.

So far, we have not ruled out completely  the possibility that $\CO_d^{\lambda_R} \not\subset \CF_d$, but we have already restricted the isomorphism class of the fixed point algebra. The next result shows that at least there are no {\em algebraic} fixed points outside $\CF_d$ if $R$ is non-trivial. It also shows that \eqref{eq:F1ergodicity} captures precisely the absence of non-trivial algebraic fixed points.

Here and in the following, we write ${}^0\CO_d\subset\CO_d$ for the algebraic part of $\CO_d$, i.e. the unital ${}^*$-algebra of polynomials in the generators $S_1,\ldots,S_d$ and their adjoints, and ${}^0\CF_d:={}^0\CO_d\cap\CF_d=\bigcup_{n\in\N}\CF_d^n={}^0\CN$ for the algebraic part of $\CF_d$. We also use the shorthand notations ${}^0\CO_d^{\la_R}:={}^0\CO_d\cap\CO_d^{\la_R}$ and ${}^0\CF_d^{\la_R}:={}^0\CF_d\cap\CF_d^{\la_R}$.

\begin{proposition}\label{prop:algebraicfixedpoints}
	Let $R\in\CR(d)$.
	\begin{enumerate}
		\item\label{item:algebraicfixedpointsareinUHF} If $R\not\in\C$, then all algebraic fixed points of $\la_R$ are contained in $\CF_d$, i.e.
		\begin{align}
			{}^0\CO_d^{\lambda_R}={}^0\CF_d^{\lambda_R}.
		\end{align}
		\item\label{item:fixedpointsinF1} ${}^0\CF_d^{\lambda_R}=\C$ if and only if $(\CF_d^1)^{\lambda_R}=\C$.
	\end{enumerate}
\end{proposition}
\begin{proof}
	\ref{item:algebraicfixedpointsareinUHF} Let $x\in{}^0\CO_d$ be an algebraic fixed point of $\la_R$ that is not contained in $\CF_d$, without loss of generality assumed to be selfadjoint. As $x\not\in\CF_d=\CO_d^{(0)}$, it has a non-zero spectral component $x^{(n)}$, $n>0$, which also lies in ${}^0\CO_d^{\la_R}$. We may therefore express it as $x^{(n)}=E^n(x)S_1^n$ with $E^n(x)\in\CF_d^k$ for some $k\in\N_0$. Then, for all multi indices $\alpha,\beta$ of length $|\alpha|=|\beta|=k$ we have $t_{\alpha,\beta}:=S_\alpha^* E^n(x) S_\beta\in\C$.
 
	Now define $T:=S_\alpha^* x^{(n)} S_\beta=S_\alpha^* E^n(x)S_1^nS_\beta$ where we have chosen $\alpha,\beta$ such that $T\neq0$; this is possible because $x^{(n)}\neq0$. By virtue of Prop.~\ref{proposition:generalfixedpointresults}~\ref{item:FP-stability}, $T$ is a fixed point. Furthermore, $T$ can be expressed as
	\begin{align*}
		T
		=
		S_\alpha^* E^n(x)S_1^nS_\beta=\sum_{\gamma:|\gamma|=k}S_\alpha^* E^n(x)S_\gamma\,S_\gamma^*S_1^nS_\beta
		=
		\sum_{\gamma:|\gamma|=k}t_{\alpha,\gamma}S_\gamma^*S_1^nS_\beta.
	\end{align*}
	As the multi indices $\beta$ and $\gamma$ have the same length $k$ for all terms in the sum, we see that $T$ is a linear combination of products of $n$ generators $S_{i_1}\cdots S_{i_n}$. In particular, $T$ is a (non-zero) multiple of an isometry.
  
	To conclude the proof, note that as a consequence of $R$ being an element of $\CF_d^2$, and in view of the form of $T$, we have $(T^*)^2RT^2\in\C$. But as a fixed point, $T$ commutes with $R$ (cf. Prop.~\ref{proposition:generalfixedpointresults}~\ref{item:FP-M}). Therefore 
	\begin{align*}
	\C\ni(T^*)^2RT^2=(T^*)^2T^2R,
	\end{align*}
	and as $(T^*)^2T^2$ is a non-zero scalar, the triviality of $R$ follows.
	
	\ref{item:fixedpointsinF1} The implication $\implies$ is trivial. For the reverse implication, let $x\in(\CF_d^k)^{\lambda_R}$ for some $k\in\N$. Then, by Prop.~\ref{proposition:generalfixedpointresults}~\ref{item:FP-stability}, $S_{i_1}^*\cdots S_{i_{k-1}}^*xS_{j_{k-1}}\cdots S_{j_1}\in(\CF_d^1)^{\lambda_R}=\C$ for all $i_l,j_l$. Thus 
	\begin{align*}
	 x=\sum_{l=1}^{k-1}\sum_{i_l,j_l=1}^dS_{i_{k-1}}\cdots S_{i_1}\left(S_{i_1}^*\cdots S_{i_{k-1}}^*xS_{j_{k-1}}\cdots S_{j_1}\right)S_{j_1}^*\cdots S_{j_{k-1}}^*\in\CF_d^{k-1},
	\end{align*}
	and inductively it follows that $x\in(\CF_d^1)^{\lambda_R}=\C$.
\end{proof}

We now compare the ergodicity condition and the condition $(\CF_d^1)^{\la_R}=\Cl$ in more detail. It turns out that they have quite different behavior with respect to taking box sums.

\begin{lemma}\label{lemma:boxplus-and-ergodicity}
    Let $R,S\in\CR$.
    \begin{enumerate}
     \item\label{item:ergodicity-boxplus} $R\boxplus S$ satisfies the ergodicity condition if and only if both $R$ and $S$ do.
     \item\label{item:F1ergodicity-boxplus} $\la_{R\boxplus S}$ has no non-trivial algebraic fixed points.
    \end{enumerate}
\end{lemma}
\begin{proof}
    \ref{item:ergodicity-boxplus} Let us view $R\in\CR(d)\subset\End(V\ot V)$, $S\in\CR(d')\subset\End(W\ot W)$ with $\dim V=d$, $\dim W=d'$, and pick orthonormal bases $\{e_i:i=1,\ldots,d\}$ of $V$ and $\{f_j:j=1,\ldots,d'\}$ of $W$. We denote the orthogonal projection from $V\oplus W$ onto $V$ and $W$ by $p$ and $p^\perp$, respectively.
    
    Recall that $E_1$ acts as the normalized right partial trace on $\End((V\oplus W)\ot(V\oplus W))$. Writing $U:=R\boxplus S$ as a shorthand, we have, $x\in\End(V\oplus W)$,
    \begin{align*}
     (d+d')\langle e_i,E_1&(UxU^*)e_j\rangle
     \\
     &=
     \sum_{k=1}^d\langle e_i\ot e_k,UxU^*(e_j\ot e_k)\rangle
     +
     \sum_{l=1}^{d'}\langle e_i\ot f_l,UxU^*(e_j\ot f_l)\rangle
     \\
     &=
    \sum_{k=1}^d\langle e_i\ot e_k,RpxpR^*(e_j\ot e_k)\rangle
     +
     \delta^i_j\sum_{l=1}^{d'}\langle f_l,p^\perp xp^\perp f_l\rangle
     .
    \end{align*}
    The ergodicity condition demands that for every $x$, this equals 
    \begin{align*}
     (d+d')\langle e_i,\tau(x)e_j\rangle
     =
     \delta^i_j\sum_{k=1}^d\langle e_k,pxp e_k\rangle
     +
     \delta^i_j\sum_{l=1}^{d'}\langle f_l,p^\perp xp^\perp f_l\rangle
     .
    \end{align*}
    Comparing the expressions, we see that the ergodicity condition for $R\boxplus S$ implies the ergodicity condition for $R$. Analogously, one shows that ergodicity of $S$ is necessary for ergodicity of $R\boxplus S$.
    
    To check that this is sufficient, we also have to consider the ``mixed'' expectation values of $E_1(UxU^*)$ between vectors in $V$ and $W$, namely $\langle e_i,E_1(UxU^*)f_j\rangle$. But since $R\boxplus S$ acts as the flip on mixed tensors, it follows that these necessarily vanish, in agreement with the ergodicity condition. Hence ergodicity of $R$ and $S$ is also sufficient for ergodicity of $R\boxplus S$.

    \ref{item:F1ergodicity-boxplus} We need to show that the only $x\in\End(V\oplus W)$ commuting with $U=R\boxplus S$ are multiples of the identity (cf.~Prop.~\ref{prop:algebraicfixedpoints}). We have
    \begin{align*}
        UxU^*(p\ot p)
        &=
        U(pxp\ot p+p^\perp xp\ot p)R^*
        =
        R(pxp\ot p)R^*+(p\ot p^\perp xp)FR^*
        ,\\
        x(p\ot p)
        &=
        pxp\ot p+p^\perp xp\ot p.
    \end{align*}
    As $R$ commutes with $p\ot p$, this implies $p^\perp xp=0$, and analogously $pxp^\perp=0$.
    
    Similarly,
    \begin{align*}
     UxU^*(p\ot p^\perp)
     &=
     U(xp^\perp\ot p)F
     =
     U(p^\perp xp^\perp\ot p)F
     =
     p\ot p^\perp xp^\perp,
     \\
     x(p\ot p^\perp)
     &=
     pxp\ot p^\perp.
    \end{align*}
    Taking partial traces, we find $pxp=c\cdot p$, $p^\perp xp^\perp =c\cdot p^\perp$ with $c\in\Cl$. Thus $x=c\in\Cl$, and \eqref{eq:F1ergodicity} is satisfied.
\end{proof}

This result gives us many R-matrices that are not ergodic but do not have any non-trivial algebraic fixed points either. Consider an involutive R-matrix~$N$ of normal form, i.e. 
\begin{align}\label{eq:normal2}
 N=\bigboxplus_{i=1}^n\eps_i1_{d_i}
\end{align}
for some $n\in\Nl$, with signs $\eps_i\in\{\pm1\}$ and dimensions $d_i\in\Nl$, $\sum_{i=1}^nd_i=d$ (see Thm.~\ref{theorem:involutivecase}~\ref{item:normalform}). Then Lemma~\ref{lemma:boxplus-and-ergodicity}~\ref{item:F1ergodicity-boxplus} shows that $N$ has non-trivial fixed points if and only if it is trivial, namely $n=1$ and $N=\pm1$. We also know if $d_1=\ldots=d_n=1$, then $N$ is diagonal and hence ergodic (Cor.~\ref{cor:diagonalsareergodic}). But all other normal forms $N$, and in fact all R-matrices $R$ equivalent to them, are {\em not} ergodic, as we show next.

\begin{proposition}\label{proposition:ergodicity-and-characters}
    Let $R$ be ergodic. Then 
    \begin{align}
     \|\phi_R(R)\|_2^2=\tau(R^*\varphi(R))=\frac{1}{d^2}.
    \end{align}
	If $R$ is ergodic and involutive, it is of diagonal type, i.e. $R\sim N$ for a normal form \eqref{eq:normal2} with $d_1=\ldots=d_n=1$.
\end{proposition}
\begin{proof}
    We consider the ergodicity condition \eqref{eq:ergodicity-matrixform} with $i=k$ and $j=l$. Summing over $i,j$ gives
    \begin{align*}
        d^{-2}
        &=
        d^{-3}\sum_{i,j=1}^d\delta_j^i
        =
        d^{-3}\sum_{i,j,n,m=1}^d R^{im}_{in}  (R^*)^{jn}_{jm}
        =
        \tau(\phi_F(R)\phi_F(R^*)).
	\end{align*}
	Recalling that $\phi_F(R)=\phi_R(R)$, this gives $\|\phi_R(R)\|_2^2=d^{-2}$ as claimed. Furthermore,
	\begin{align*}
		\tau(\phi_R(R)\phi_F(R^*)).
        &=
        \tau(R\la_R(\phi_F(R^*)))
        =
        \tau(R\phi_F(R^*))
        =
        \tau(\varphi(R)R^*).
	\end{align*}
	We now specialize to the case that $R=R^*$ is involutive. Then we may express $\tau_R(b_1b_2)=\tau(\varphi(R)R)$, the value of a three-cycle in the character $\tau_R$, in terms of the Thoma parameters $\alpha_k,\beta_l$ of $R$. Recall that $d\alpha_k,d\beta_l\in\Nl$ are the dimensions~$d_i$ of the normal form of $R$, summing to $d$. Thus, by \eqref{eq:ThomaOnCycle},
    \begin{align*}
     d=d^3\tau(\varphi(R)R)=\sum_k(d\alpha_k)^3+\sum_l(d\beta_l)^3
     =
     \sum_{i=1}^n d_i^3
     \geq
     \sum_{i=1}^n d_i
     =d.
    \end{align*}
    It follows that $d_i=1$ for all $i$.
\end{proof}

We now want to demonstrate the fact hinted at earlier -- there exist R-matrices $R$ such that $\la_R$ is ergodic on the $C^*$-algebra $\CO_d$, but not on the von Neumann algebra $\CM$ (or, analogously, ergodic on $\CF_d$ but not on $\CN$). For this, we need a result that improves the absence of non-trivial algebraic fixed points (Prop.~\ref{prop:algebraicfixedpoints}) to absence of non-trivial fixed points in $\CO_d$.

The arguments in the following proof are generalisations of arguments given in \cite{MatsumotoTomiyama:1993}. Note that the Yang-Baxter equation is not used here.

\begin{proposition}\label{prop:SiGeneratorNoFixedPoints}
    Let $U\in\CU(\CF_d)$ and $v\in\CU(\CF_d^1)$ such that there exists $i\in\{1,\ldots,d\}$ with $vS_i=z\cdot S_i$ for some $z\in\T$. If $S_i\in(\la_v,\la_U)$, then $\CO_d^{\la_U}=\C$. 
\end{proposition}
\begin{proof}
    In view of Prop.~\ref{prop:ergodicityFvsO} it is enough to show that $\CF_d^{\la_U}=\C$. Let $x\in\CF_d^{\lambda_U}$ be a fixed point. Writing $T:=S_i$ for the intertwiner, the assumption $T\in(\la_v,\la_U)$ implies
    \begin{align}
        T\la_v(x)=\la_U(x)T=xT
        \;\implies\;
        x=\la_v^{-1}(T^*xT).
    \end{align}
    Since $\la_v^{-1}(T)=v^{-1}S_i=\frac{1}{z}\,T$, we see that $\la_v^{-1}$ commutes with $\ad T^*$. We therefore have $x=T^*\la_v^{-1}(x)T$, which we may iterate to 
    \begin{align}
     x=(T^*)^n\la_v^{-n}(x)T^n,\qquad n\in\N.
    \end{align}
    We now show that this implies $x\in\C$. Indeed, if $x$ lies in $\CF_d^m$ for some $m\in\N$, then so does $\lambda_v^{-n}(x)$, and thus ${T^*}^n\lambda_v^{-n}(x)T^n\in\C$ for all $n\geq m$. This already shows that $\la_U$ admits no non-trivial algebraic fixed points. 
    
    If $x\in\CF_d$ is a non-algebraic fixed point of $\la_U$, we consider a sequence $(x_k)_{k\in\N}\subset{}^0\CF_d$ converging in norm to $x$. For any $k$, there exists $n(k)\in\N$ such that for all $n\geq n(k)$, we have ${T^*}^n\lambda_v^{-n}(x_k)T^n=\mu_k\cdot1$ for an $n$-independent complex number $\mu_k$. Given $k,l\in\N$, we then have for $n\geq\max\{n(k),n(l)\}$
    \begin{align*}
     |\mu_k-\mu_l|=\|{T^*}^n\lambda_v^{-n}(x_k-x_l)T^n\|
     \leq
     \|x_k-x_l\|,
    \end{align*}
    and it follows that $\mu_k$ converges to a limit $\mu$ as $k\to\infty$. 
    
    To show that $x=\mu\cdot1$, let $n,k\in\N$ be arbitrary. We have
    \begin{align*}
     \|x-\mu\|
     &=
     \|{T^*}^n\la_v^{-n}(x)T^n-\mu\|
     \\
     &\leq
     \|{T^*}^n\la_v^{-n}(x-x_k)T^n\|+\|{T^*}^n\la_v^{-n}(x_k)T^n-\mu_k\|+|\mu_k-\mu|
     \\
     &\leq
     \|x-x_k\|+\|{T^*}^n\la_v^{-n}(x_k)T^n-\mu_k\|+|\mu_k-\mu|.
    \end{align*}
    Given $\eps>0$, we can choose $k$ large enough such that $\|x-x_k\|<\eps$ and $|\mu-\mu_k|<\eps$. Choosing $n>n(k)$, we also have ${T^*}^n\la_v^{-n}(x_k)T^n-\mu_k=0$ and conclude $\|x-\mu\|<2\eps$.
\end{proof}

We mention as an aside that this proposition still holds when $U$ is an arbitrary unitary in $\CO_d$. Since we will not need this stronger version, we refrain from giving the proof.

Let us now look at an explicit example.

\begin{example}
	Consider the normal form R-matrix $N:=1_2\boxplus1_1\in\CR(3)$. We claim that
	\begin{align}\label{eq:ergodicity-mismatch}
	  \CO_3^{\la_N}=\Cl,\qquad \CN^{\la_N}\neq\Cl.
	\end{align}
	The non-ergodicity of $\la_N$ on $\CN$, i.e. $\CN^{\la_N}\neq\Cl$, follows from Prop.~\ref{proposition:ergodicity-and-characters} because $N$ is an involutive normal form with dimensions $d_1=2,d_2=1$. 
	
	To demonstrate ergodicity of $\la_N$ on $\CO_3$, we will verify the conditions of Prop.~\ref{prop:SiGeneratorNoFixedPoints} with $v=1$ and $i=3$, i.e. show that $S_3$ is an intertwiner from $\id$ to $\la_N$. We have to show $S_3S_i=NS_iS_3$ and $S_3S_i^*=S_i^*NS_3$ for $i=1,2,3$ (note that $N=N^*$).
	
	The R-matrix is here $N=\sum_{j,k,l,m=1}^3 N^{jk}_{lm}S_jS_kS_m^*S_l^*$ and its matrix elements satisfy $N^{kj}_{3l}=\delta^j_3\delta^k_l=N^{jk}_{l3}$ by definition of $N$ (note that $N=FNF$). Thus, $i=1,2,3$,
	\begin{align*}
	 NS_iS_3 
	 =
	 \sum_{j,k,l,m=1}^3 N^{jk}_{lm}S_jS_kS_m^*S_l^*S_iS_3 = 
	 \sum_{j,k=1}^3 N^{jk}_{i3}S_jS_k
	 = S_3S_i
	\end{align*}
	and
	\begin{align*}
	 S_i^*NS_3 
	 =
	 S_i^*\sum_{j,k,l,m=1}^3 N^{jk}_{lm}S_jS_kS_m^*S_l^*S_3
	 = 
	 \sum_{k,m=1}^3 N^{ik}_{3m}S_kS_m^*,
	 = S_3S_i^*
	\end{align*}
	which finishes the proof. With a little more effort, one shows
	\begin{align*}
	\lambda_N(S_1)&=S_1S_1S_1^*+S_1S_2S_2^*+S_3S_1S_3^*,\\
	\lambda_N(S_2)&=S_2S_1S_1^*+S_2S_2S_2^*+S_3S_2S_3^*,\\
	\lambda_N(S_3)&=S_1S_3S_1^*+S_2S_3S_2^*+S_3S_3S_3^*.
	\end{align*}
	For completeness, we also mention that in this example, $\CM_{N,1}\cong\C\oplus\C$ (Cor.~\ref{corollary:relativecommutant}), i.e. $\lambda_N\cong\mu\oplus\id$ with some irreducible non-trivial endomorphism $\mu$. Since the intertwiner $T$ for $\mu\prec\lambda_N$ must generate together with $S_3$ a copy of $\CO_2$, which does not exist within $\CO_3$, this decomposition can only hold on the level of the associated von Neumann algebras, i.e. $T\in\CM\supset\CO_3$.
\end{example}

In Section~\ref{section:2dRmatrices}, we discuss another example in which the algebraic part of the fixed point algebra is infinite dimensional and can be described explicitly (Prop.~\ref{prop:fixedpointsinexample}).

\bigskip

To conclude this section, we compare ergodicity and irreducibility. Note that $(\CF_d^1)^{\la_R}$ and $\CM_{R,1}$ (or $\CN_{R,1}$) are commuting subalgebras of $\CF_d^1$ because trivially $(\CF_d^1)^{\la_R}\subset\la_R(\CF_d^1)$. This leads to the following observation, independent of the Yang-Baxter equation.

\begin{lemma}\label{lemma:dprime}
    Let $R\in\CU(\CF_d^2)$ with $d$ prime. Then either $\CM_{R,1}=\C$ or $(\CF_d^1)^{\lambda_R}=\C$.
\end{lemma}
\begin{proof}
    Let $p\in\CM_{R,1}$ and $q\in(\CF_d^1)^{\lambda_R}$ be orthogonal projections. Then $R^*pR=\varphi(p)$ \eqref{eq:SelfIntertwinersGeneral} and $q=\la_R(q)=RqR^*$, and therefore
    \begin{align*}
     pq
     &=
     R\varphi(p)R^*RqR^*
     =
     R\varphi(p)qR^*
     .
    \end{align*}
    As $p$ and $q$ commute, $pq=p\wedge q$. Evaluating in $\tau$ gives $\tau(p\wedge q)=\tau(R\varphi(p)qR^*)=\tau(p)\tau(q)$, which is equivalent to $d\,{\rm Tr}(p\wedge q)={\rm Tr}(p){\rm Tr}(q)$ with Tr the matrix trace of $\CF_d^1\cong M_d$. Taking into account that as selfadjoint projections, $p$, $q$, and $p\wedge q$ have traces in $\{0,\ldots,d\}$, and that $d$ is prime, it follows that ${\rm Tr}(p)\in\{0,d\}$ or ${\rm Tr}(q)\in\{0,d\}$. Thus either $p$ or $q$ has to be a trivial projection.
\end{proof}

If $d=n\cdot m$ is not prime, there exist R-matrices such that $\lambda_R$ is reducible and has non-trivial fixed points in $\CF_d^1$. Such R-matrices can be constructed as tensor products $R=S\boxtimes T$, where $R\in\CR(n)$ is chosen such that $\la_R$ is reducible (e.g., the flip) and $S\in\CR(m)$ is chosen such that $\la_S$ has non-trivial fixed points in~$\CF_m^1$ (see Section~\ref{section:tensorproducts}).

\bigskip

So far we do not know any R-matrices that are both irreducible and ergodic. It is possible that irreducibility implies the existence of non-trivial fixed points.

\section{Two-dimensional R-matrices}\label{section:2dRmatrices}

As a concrete family of examples, we consider in this section R-matrices in dimension $d=2$. In \cite{Hietarinta:1992_8}, all solutions to the Yang-Baxter equation have been computed, including non-unitary and non-involutive ones. In \cite{Dye:2003_2}, the unitary solutions have been singled out: $\CR(2)$ consists precisely of all those matrices $R$ which are of the form $R=(Q\ot Q)R_i(Q\ot Q)^{-1}$, where $R_i$, $i=1,\ldots,4$, is one of the following R-matrices and $Q\in\End\C^2$ is invertible and satisfies certain restrictions ensuring that $R$ is unitary\footnote{In this section (only), the notation $R_i$ refers to the specific R-matrices listed here, and not to \eqref{eq:Un-nU}.}.

\begin{align}
 R_1&=q\cdot1,\qquad q\in\T,\\
 R_2
 &=
 \left(
	\begin{array}{cccc}
	 p\\
	 &&q\\
	 &r\\
	 &&&s
	\end{array}
 \right),\qquad p,q,r,s\in\T,
 \\
 R_3
 &=
 \left(
	\begin{array}{cccc}
	 &&&p\\
	 &q\\
	 &&q\\
	 r
	\end{array}
 \right),\qquad q,p\cdot r\in\T,
 \\
 R_4
 &=
 \frac{q}{\sqrt{2}}
        \left(
            \begin{array}{rrrr}
            1 & 1 \\
            -1 & 1 \\
            &&1&-1\\
            &&1 & 1
            \end{array}
        \right),\qquad q\in\T.
		\label{R4}
\end{align}
Note that $R_3$ is not always unitary because only $|pr|=1$ is required, and also $Q$ is not necessarily unitary.

For our purposes, it is better to present the elements of $\CR(2)$ in the form $\la_u(R_i)\cong (u\ot u)R_i(u\ot u)^{-1}$, where both $u\in\CF_2^1$ and $R_i\in\CF_2^2$ are unitary.

\begin{theorem}\label{thm:2dRmatrices}
	A matrix $R\in\CF_2^2$ lies in $\CR(2)$ if and only if there exists $u\in\CU(\CF_2^1)$ and $i\in\{1,\ldots,4\}$ such that $R=\la_u(R_i)$, where all parameters $p,q,r,s$ appearing in the representatives $R_1,\ldots,R_4$ have modulus 1.
\end{theorem}
\begin{proof}
	The ``if'' part of the statement follows by noting that when the parameters $p,q,r,s$ have modulus 1, then $R_1,\ldots,R_4\in\CR(2)$. For the ``only if'' statement, we first note that for $Q=\genfrac{(}{)}{0pt}{1}{1\;\;0}{0\;\;a}$ with $a=\sqrt{|p|}$, the transformed matrix $(Q\ot Q)R_3(Q\ot Q)^{-1}$ is of the same form as $R_3$, but with all parameters having unit modulus. We may therefore without loss of generality take all parameters to have unit modulus, i.e. all representatives $R_1,\ldots,R_4$ to be unitary.
	
	Let now $R=(Q\ot Q)R_i(Q\ot Q)^{-1}$ for some invertible $Q\in\End\C^2$ and $R_i$ unitary. Then $R^*=R^{-1}$ is equivalent to $R_i$ commuting with $|Q|^2\ot |Q|^2$, where $|Q|^2=Q^*Q$. Thus $R_i$ also commutes with $|Q|\ot |Q|$. Proceeding to the polar decomposition $Q=U|Q|$, $U\in\CU(\CF_2^1)$, we then have
	\begin{align*}
		R&=(Q\ot Q)R_i(Q\ot Q)^{-1}
		=
		(U\ot U)(|Q|\ot|Q|)R(|Q|^{-1}\ot|Q|^{-1})(U^{-1}\ot U^{-1})
		\\
		&=
		(U\ot U)R(U^{-1}\ot U^{-1})
		=
		\la_U(R_i).
	\end{align*}
	This establishes that $R$ is of the claimed form.
\end{proof}

In Cuntz algebra notation, the representatives $R_1,\ldots,R_4$ take the form
\begin{align}
	R_1
	&=q\cdot 1,\\
	R_2
	&=
	p\,S_1S_1S_1^*S_1^*+
	q\,S_1S_2S_1^*S_2^*+
	r\,S_2S_1S_2^*S_1^*+		
	s\,S_2S_2S_2^*S_2^*,
	\\
	R_3
	&=
	p\,S_1S_1S_2^*S_2^*+
	q\,S_1S_2S_2^*S_1^*+
	q\,S_2S_1S_1^*S_2^*+		
	r\,S_2S_2S_1^*S_1^*,
	\\
	R_4
	&=
	\frac{q}{\sqrt{2}}\left(1+(S_1S_1^*-S_2S_2^*)\varphi(-S_1S_2^*+S_2S_1^*)\right)
	.
\end{align}
By explicit calculations, one verifies that if $R=\la_u(R_i)$, then also its adjoint~$R^*$ and its flipped version $FRF$ are of this form, i.e. $R^*=\la_{u'}(R_i)$ and $FRF=\la_{u''}(R_i)$ for suitable $u',u''\in\CU(\CF_2^1)$, and the same\footnote{The only non-trivial thing to do is to find $u\in\CU(\CF_2^1)$ such that $FR_4F=\la_u(R_4)$; here $u=\frac{1}{\sqrt{2}}\genfrac{(}{)}{0pt}{1}{-1\;\;i}{-i\;\;1}$ works.} $i$. In particular, equivalences of type 1 and type 3 (see p.~\pageref{page:types}) leave the families $\{\la_u(R_i)\,:\,u\in\CU(\CF_d^1)\}$ invariant.

However, type 2 equivalences can change the representative $R_i$. Indeed, $\la_u(R_3)=R_3$ for $u=\genfrac{(}{)}{0pt}{1}{0\;\;a}{1\;\;0}$ with $a=\sqrt{p/q}$, but $\varphi(u)R_3\varphi(u)^*$ equals the second representative $R_2$ after suitable identification of parameters.

\medskip

Below we give a table summarizing key features of the endomorphisms corresponding to the R-matrices $R=\la_u(R_i)$, $i=1,\ldots,4$. Note that irreducibility and ergodicity of $R$ do not depend on $u$ as both properties are invariant under type 1 equivalences. The index in the third column is $[\CN:\la_R(\CN)]=\Ind_{E_R}(\la_R)$.

\begin{center}
{\small
\begin{tabular}{|c|clcl|}
\hline
  \# & Representative & $\CM_{R,1}$ & $\Ind.$ & Fixed point algebras
 \\
 \hline\hline
 1&$q\cdot1$ & $\C$ (automorphism) & $1$ & 
 \begin{tabular}{@{}l}
 	$\CO_2^{\la_R}\cong\CF_2$ \;\;${\rm ord}(q)=\infty$\\
	$\CO_2^{\la_R}\cong\CO_{2^{{\rm ord}(q)}}$ 
 \end{tabular}
 \\
 \hline
 2&${\left(
	\begin{array}{cccc}
	 p\\
	 &&q\\
	 &r\\
	 &&&s
	\end{array}
 \right)}$
 & 
 \begin{tabular}{@{}l@{}l}
	$M_2$ & \;$p=r$, $q=s$\\
	$\C\oplus\C$ & \;else
 \end{tabular}
 &$4$&$\CN^{\la_R}=\C$
 \\
 \hline
 3&${\left(
	\begin{array}{cccc}
	 &&&p\\
	 &q\\
	 &&q\\
	 r
	\end{array}
 \right)}$
 & \begin{tabular}{@{}l@{}l}
    $\C\oplus\C$ & \;$q^2=pr$\\
    $\C$ & \;$q^2\neq pr$
 \end{tabular}
 & $4$&$\CN^{\la_R}=\C$
 \\
 \hline
 4&${\frac{q}{\sqrt{2}}
        \left(
            \begin{array}{rrrr}
            1 & 1 \\
            -1 & 1 \\
            &&1&-1\\
            &&1 & 1
            \end{array}
        \right)}$
 & $\C$
 & $2$ & \begin{tabular}{@{}l}
	$\dim\CF_2^{\la_R}=\infty$\\
	see Prop.~\ref{prop:fixedpointsinexample}
 \end{tabular}
 \\
 \hline
\end{tabular}
}
\end{center}

\bigskip 

\noindent{\em Proof of the claims in the table:} We go through families 1--4. The R-matrices in family 1 define automorphisms (hence $\Ind\la_R=1$), and the form of the fixed point algebra has been commented on before (remark after Prop.~\ref{prop:simplepurelyinfinite}). 

For the diagonal R-matrices in family 2, Prop.~\ref{corollary:relativecommutant}~\ref{item:decomposelambdaR-fordiagonalR} shows that $\la_R$ decomposes into two quasi-free automorphisms which are either equivalent (if $p=r$ and $q=s$) or inequivalent (if $p\neq r$ or $q\neq s$). This implies the claimed form of the relative commutant and shows $\Ind\la_R=4$ in both cases. Since $R_2$ is diagonal, its ergodicity follows from Cor.~\ref{cor:diagonalsareergodic}.

For the ``anti-diagonal'' R-matrices in family 2, one computes
\begin{align*}
 \CM_{R_3,1}=
 \{x\in\CF_2^1\,:\,R_3^*xR_3=\varphi(x)\}
 =
 \begin{cases}
  \C & q^2 \neq pr\\
  \C\oplus\C & q^2 =pr
 \end{cases}.
\end{align*}
In the second case, $\la_R$ is equivalent to the direct sum of two inequivalent automorphisms, and $\Ind\la_R=4$. In the first case, $\la_R$ is irreducible and $R$ has the three distinct eigenvalues $q,\sqrt{pr},-\sqrt{pr}$. As the cardinality of the spectrum is a lower bound for $\Ind\la_R$ \eqref{eq:lowerindexbounds}, and in $d=2$, the index of $\la_R$ may only take the values $1,2$, or $4$ \cite[Prop.~9.9]{ContiPinzari:1996}, we see $\Ind\la_R=4$ also in this case. 

Each member of family 3 is type 2 equivalent to a member of family 2, i.e. $R_3\Nsim R_2$, and the equivalence relation $\Nsim$ preserves ergodicity (Remark~\ref{remark:stability-of-ergodicity}). Hence family 3 is ergodic as well.

Due to the block form of the representative $R_4$ for the last family, $S_1S_1^*\in\CF_2^1$ is seen to be a fixed point of $\la_{R_4}$. Its fixed point algebra will be described in more detail below. Since $d=2$ is prime, $\la_R$ is irreducible (Lemma~\ref{lemma:dprime}).\hfill$\qed$

\smallskip

The R-matrix $R_4$ \eqref{R4} is special from various points of view: Up to applying quasi-free automorphisms, $R_4$  is the unique non-trivial R-matrix in $\CR(2)$ for which $\la_R$ is not ergodic, and the unique R-matrix in $\CR(2)$ with index $2$. We also mention that $R_4$ generates a representation of the Temperley-Lieb algebra at loop parameter $\delta=\frac{1}{2}$, and satisfies $R_4^4\in\C$. Furthermore, $\la_{R_4}(\CO_2)$ is the fixed point algebra of an explicit order two automorphism $\alpha\in\Aut\CO_2$ \cite{ContiFidaleo:2000}. The images of the braid group representations $\rho_R(B_n)$ are described in \cite{FrankoRowellWang:2006} in terms of extraspecial 2-groups, and its relevance for topological quantum computing is discussed in \cite{KauffmanLomonacoJr:2004_2}. A variation of $R_4$ also appears in the exchange algebra of light-cone fields in the Ising model \cite{RehrenSchroer:1987}

In view of this interest in $R_4$, it might be useful to indicate how it can be obtained systematically from the results of this article. We look for a non-trivial matrix $R\in M_2\ot M_2\cong M_4$ that is a unitary solution of the Yang-Baxter equation such that $\la_R$ is irreducible and has non-trivial fixed points in~$\CF_d^1$. Then we know that a) $R$ has trivial left and right partial traces $\phi_F(R)=\phi_F(FRF)=\tau(R)$, and b) there is a one-dimensional projection $p\in\CF_d^1$ that commutes with~$R$. Choose a basis of $\Cl^2$ such that $p=\genfrac{(}{)}{0pt}{1}{1\;\;0}{0\;\;0}$ (this amounts to applying a quasi-free automorphism to $R$). Then a), b) imply that $R$ is of the form
\begin{align}
	\left(
            \begin{array}{rrrr}
            a & b \\
            c & d \\
            &&d&-b\\
            &&-c & a
            \end{array}
        \right),
\end{align}
with $a,b,c,d\in\Cl$. At this stage, it is not difficult to implement the requirements that $R$ is unitary and solves the Yang-Baxter equation. One finds that non-triviality requires $b,c\neq0$, and the YBE then implies $d=a$ and $c=-a^2/b$. Implementing unitarity yields the form \eqref{R4}.

\medskip

To conclude this discussion, we now describe the fixed points of $\la_{R_4}$ in $\CF_2$ in more detail. To this end, we use the standard Pauli matrices $\sigma_0,\ldots,\sigma_3$ as a basis for $M_2\cong\CF_2^1$, with $\sigma_0=1$.

\begin{proposition}\label{prop:fixedpointsinexample}
 An element $x\in\CF_2^n$, $n\in\N$, is a fixed point of $\la_{R_4}$ if and only if it is a linear combination of elements of the form $\sigma_{i_1}\varphi(\sigma_{i_2})\cdots\varphi^{n-1}(\sigma_{i_n})$, where the following three conditions are satisfied:
 \begin{enumerate}
  \item $i_n\in\{0,3\}$,
  \item If $i_k\in\{0,2\}$ for some $k\in\{2,\ldots,n\}$, then $i_{k-1}\in\{0,3\}$,
  \item If $i_k\in\{1,3\}$ for some $k\in\{2,\ldots,n\}$, then $i_{k-1}\in\{1,2\}$.
 \end{enumerate}
 We have $\dim(\CF_2^n)^{\la_R}=2^n$ and $\CN^{\la_R}=({}^0\CF_2^{\la_R})''$.
\end{proposition}
\begin{proof}
  The first step is to realise that the R-matrix $R_4$ has the form 
  \begin{align*}
   R_4=\frac{q}{\sqrt{2}}(1+i\sigma_3\varphi(\sigma_2)).
  \end{align*}
  Thus $x\in\CF_2$ is a fixed point of $\la_{R_4}$ if and only if it commutes with $\varphi^m(S)$, $m\in\N_0$, where $S:=\sigma_3\varphi(\sigma_2)$ (cf.~Prop.~\ref{proposition:generalfixedpointresults}~\ref{item:FP-N}). Recall that the Pauli matrices satisfy $\sigma_i=\sigma_i^*=\sigma_i^{-1}$ and
  \begin{align}\label{eq:pauliaction}
   \sigma_i\sigma_j\sigma_i=\begin{cases}
                               +\sigma_j & j\in\{0,i\}\\
                               -\sigma_j & \text{else}
                              \end{cases}.
  \end{align}
  Let $x$ be a linear combination of elements of the form $\sigma_{i_1}\varphi(\sigma_{i_2})\cdots\varphi^{n-1}(\sigma_{i_n})$. In view of the action \eqref{eq:pauliaction}, it follows that $x$ is a fixed point if and only if each term in its expansion into this basis is a fixed point, i.e. we may take $x=\sigma_{i_1}\varphi(\sigma_{i_2})\cdots\varphi^{n-1}(\sigma_{i_n})$ without loss of generality. 
  
  Since $\sigma_2^2=1$, we have
  \begin{align*}
   \ad \varphi^{n-1}(S)(x)=\sigma_{i_1}\varphi(\sigma_{i_2})\cdots\varphi^{n-1}(\sigma_3\sigma_{i_n}\sigma_3),
  \end{align*}
  which coincides with $x$ if and only if $\sigma_3\sigma_{i_n}\sigma_3=\sigma_{i_n}$, i.e. if and only if $i_n\in\{0,3\}$ as claimed in a). Similarly, 
  \begin{align*}
   \ad \varphi^{k-1}(S)(x)
   =
   \sigma_{i_1}\varphi(\sigma_{i_2})\cdots\varphi^{k-1}(\sigma_3\sigma_{i_{k}}\sigma_3)\varphi^k(\sigma_2\sigma_{i_{k+1}}\sigma_2)
   \cdots\varphi^{n-1}(\sigma_{i_n}),
  \end{align*}
  which coincides with $x$ if and only if either $\sigma_2\sigma_{i_{k+1}}\sigma_2=\sigma_{i_{k+1}}$ and $\sigma_3\sigma_{i_{k}}\sigma_3=\sigma_{i_{k}}$ or $\sigma_2\sigma_{i_{k+1}}\sigma_2=-\sigma_{i_{k+1}}$ and $\sigma_3\sigma_{i_{k}}\sigma_3=-\sigma_{i_{k}}$. By \eqref{eq:pauliaction} this gives the listed conditions b) and~c).
  
  A dimension count gives $\dim(\CF_2^n)^{\la_R}=2^n$. 
  
  In view of the product form of $\sigma_3\varphi(\sigma_2)$, it is easy to see that $\CN^{\la_R}$ is invariant under the $\tau$-preserving conditional expectations $E_n:\CN\to\CF_2^n$. This invariance implies that any $x\in\CN^{\lambda_R}$ can be approximated weakly by the sequence of fixed points $\{E_n(x)\}_{n\in\N}$, and hence $\CN^{\la_R}=({}^0\CF_2^{\la_R})''$.
\end{proof}

This result implies in particular that $[\CN:\CL_{R_4}]=\infty$.

\section*{Acknowledgements}

We gratefully acknowledge financial support by the London Mathematical Society (Research in Pairs grant 41729, awarded in 2018), Università di Roma La Sapienza, and the Simons Center for Geometry and Physics, Stony Brook University, during the 2019 programme ``Operator Algebras and Applications'', where some of the research for this paper was carried out.

\end{document}